\documentclass[]{interact}

\usepackage{epstopdf}% To incorporate .eps illustrations using PDFLaTeX, etc.
\usepackage[caption=false]{subfig}% Support for small, `sub' figures and tables
\usepackage{color}
\usepackage{enumitem}
\usepackage{amsmath}
\usepackage{mathtools}

\newcommand{\ie}{i.\,e.~}
\newcommand{\eg}{e.\,g.~}
\newcommand{\Span}[1]{\mathrm{span}\{#1\}}
\newcommand{\Rank}[1]{\mathrm{rank}(#1)}
\newcommand{\Dim}[1]{\mathrm{dim}(#1)}
\newcommand{\C}[1]{\mathcal{C}(#1)}
\newcommand{\D}{\mathrm{d}}
\newcommand{\Lie}{\mathrm{L}}
\newcommand{\Ad}{\mathrm{ad}}

\newcommand{\Mod}{\mathrm{~mod~}}

\newcommand{\XU}{\mathcal{X}\times\mathcal{U}}

\newcommand{\TXU}{\mathcal{T}(\mathcal{X}\times\mathcal{U})}
\usepackage[natbibapa,nodoi]{apacite}% Citation support using apacite.sty. Commands using natbib.sty MUST be deactivated first!
\theoremstyle{plain}% Theorem-like structures provided by amsthm.sty
\newtheorem{theorem}{Theorem}[section]
\newtheorem{lemma}[theorem]{Lemma}
\newtheorem{corollary}[theorem]{Corollary}

\theoremstyle{definition}
\newtheorem{definition}[theorem]{Definition}

\theoremstyle{remark}
\newtheorem{remark}{Remark}

\begin{document}

	\articletype{ARTICLE TEMPLATE}% Specify the article type or omit as appropriate
	
	\title{Necessary and Sufficient Conditions for the Linearizability of Two-Input Systems by a Two-Dimensional Endogenous Dynamic Feedback}
	
	\author{
		\name{Conrad Gst\"ottner\textsuperscript{a}\thanks{CONTACT Conrad Gst\"ottner. Email: conrad.gstoettner@jku.at\newline The first author has been supported by the Austrian Science Fund (FWF) under grant number P 32151.}, Bernd Kolar\textsuperscript{b} and Markus Sch\"oberl\textsuperscript{a}}
		\affil{\textsuperscript{a}Institute of Automatic Control and Control Systems Technology, Johannes Kepler University, Linz, Austria;}
		\affil{\textsuperscript{b}Magna Powertrain Engineering Center Steyr GmbH \& Co KG, Steyrer Str. 32, 4300 St. Valentin, Austria;}
	}
	
	\maketitle
	
	\begin{abstract}
		We propose easily verifiable necessary and sufficient conditions for the linearizability of two-input systems by an endogenous dynamic feedback with a dimension of at most two.
	\end{abstract}
	
	\begin{keywords}
		Flatness, Nonlinear control systems
	\end{keywords}

	\section{Introduction}
		The concept of flatness has been introduced in control theory by Fliess, L\'evine, Martin and Rouchon, see \eg \cite{FliessLevineMartinRouchon:1992,FliessLevineMartinRouchon:1995}. For flat systems, many feed-forward and feedback problems can be solved systematically and elegantly, see \eg \cite{FliessLevineMartinRouchon:1995}. Roughly speaking, a nonlinear control system of the form
		\begin{align*}
			\begin{aligned}
				\dot{x}&=f(x,u)
			\end{aligned}
		\end{align*}
		with $\Dim{x}=n$ states and $\Dim{u}=m$ inputs is flat, if there exist $m$ differentially independent functions $y^j=\varphi^j(x,u,u_1,\ldots,u_q)$, where $u_k$ denotes the $k$-th time derivative of $u$, such that $x$ and $u$ can locally be parameterized by $y$ and its time derivatives. For this flat parameterization, we write
		\begin{align*}
			\begin{aligned}
				x&=F_x(y,y_1,\ldots,y_{r-1})\,,&&&u&=F_u(y,y_1,\ldots,y_r)
			\end{aligned}
		\end{align*}
		and refer to it as the parameterizing map with respect to the flat output $y$. If the parameterizing map is invertible, \ie $y$ and all its time derivatives which explicitly occur in the parameterizing map can be expressed solely as functions of $x$ and $u$, the system is exactly linearizable by static feedback. In this case we call $y$ a linearizing output of the static feedback linearizable system. The static feedback linearization problem has been solved completely, see \cite{JakubczykRespondek:1980,NijmeijervanderSchaft:1990}. However, for flatness there do not exist easily verifiable necessary and sufficient conditions, except for certain classes of systems, including two-input driftless systems, see \cite{MartinRouchon:1994} and systems which are linearizable by a one-fold prolongation of a suitably chosen control, see \cite{NicolauRespondek:2017}. Necessary and sufficient conditions for $(x,u)$-flatness of control affine systems with two inputs and four states can be found in \cite{Pomet:1997}.\\
		
		It is well known that every flat system can be rendered static feedback linearizable by an endogenous dynamic feedback, and conversely, every system linearizable by an endogenous dynamic feedback is flat. If a flat output is known, such a linearizing feedback can be constructed systematically, see \eg \cite{FliessLevineMartinRouchon:1999}. In this contribution we propose easily verifiable necessary and sufficient conditions for the linearizability of two-input systems by an endogenous dynamic feedback with a dimension of at most two. In \cite{GstottnerKolarSchoberl:2021}, a sequential test for checking whether a two-input system is linearizable by an endogenous dynamic feedback with a dimension of at most two has been proposed recently. The main idea of the sequential test in \cite{GstottnerKolarSchoberl:2021} is to successively split off or add endogenous dynamic feedbacks to the system in such a way that eventually a static feedback linearizable system is obtained and it is shown that the proposed algorithm succeeds if and only if the original system is indeed linearizable by an at most two-dimensional endogenous dynamic feedback. However, a major drawback of this sequential test is that it requires straightening out involutive distributions, which from a computational point of view is unfavorable. The necessary and sufficient conditions which we propose in the present contribution overcome this computational drawback. Instead of a sequence of systems a certain sequence of distributions is constructed, based on which it can be decided whether the two-input system is linearizable by an endogenous dynamic feedback with a dimension of at most two or not. Constructing these distributions and verifying the proposed conditions requires differentiation and algebraic operations only.\\
		
		It turns out that systems which are linearizable by an endogenous dynamic feedback with a dimension of at most two are actually linearizable by a special kind of endogenous dynamic feedback, namely prolongations of a suitably chosen input after a suitable static feedback transformation has been applied to the system. A complete solution for the flatness problem for the class of two-input systems which are linearizable by a one-fold prolongation of a suitably chosen control is provided in \cite{NicolauRespondek:2016}. Two-input systems which are linearizable by a two-fold prolongation of a suitable chosen control are considered in \cite{NicolauRespondek:2016-2}. However, no complete solution for the flatness problem of this class of systems is provided in \cite{NicolauRespondek:2016-2}, due to Assumption 2 therein. In Section \ref{se:examples}, we apply our results to some examples, to none of which the results in \cite{NicolauRespondek:2016-2} are applicable. Normal forms for systems which are linearizable by a one-fold prolongation can be found in \cite{NicolauRespondek:2019}, and normal forms for control affine two-input systems linearizable by a two-fold prolongation have recently been proposed in \cite{NicolauRespondek:2020}. The present contribution is greatly influenced by all these results. The novelty of our contribution are easily verifiable necessary and sufficient conditions for linearizability by a two-fold prolongation, covering also the cases to which the results in \cite{NicolauRespondek:2016-2} do not apply. These necessary and sufficient conditions are also a major improvement over the in principal verifiable but computationally inefficient necessary and sufficient conditions in \cite{GstottnerKolarSchoberl:2021}.\\
		
		The paper is organized as follows. In Section \ref{se:notation} we introduce the notation used throughout the paper. In Section \ref{se:preliminaries}, some preliminaries regarding flatness of two-input systems are presented, and in Section \ref{se:seuqnetialTest} the sequential test from \cite{GstottnerKolarSchoberl:2021} is recapitulated briefly. The main results of this contribution are presented in Section \ref{se:distributionTest}, and in Section \ref{se:examples} they are applied to practical and academic examples.
	\section{Notation}\label{se:notation}
		Let $\mathcal{X}$ be an $n$-dimensional smooth manifold, equipped with local coordinates $x^i$, $i=1,\ldots,n$. Its tangent bundle is denoted by $(\mathcal{T}(\mathcal{X}),\tau_\mathcal{X},\mathcal{X})$, for which we have the induced local coordinates $(x^i,\dot{x}^i)$ with respect to the basis $\{\partial_{x^i}\}$. We make use of the Einstein summation convention. By $\partial_xh$ we denote the $m\times n$ Jacobian matrix of $h=(h^1,\ldots,h^m)$ with respect to $x=(x^1,\ldots,x^n)$. The $k$-fold Lie derivative of a function $\varphi$ along a vector field $v$ is denoted by $\Lie_v^k\varphi$. Let $v$ and $w$ be two vector fields. Their Lie bracket is denoted by $[v,w]$, for the repeated application of the Lie bracket, we use the common notation $\Ad_v^kw=[v,\Ad_v^{k-1}w]$, $k\geq 1$ and $\Ad_v^0w=w$. Let furthermore $D_1$ and $D_2$ be two distributions. By $[v,D_1]$ we denote the distribution spanned by the Lie bracket of $v$ with all basis vector fields of $D_1$, and by $[D_1,D_2]$ the distribution spanned by the Lie brackets of all possible pairs of basis vector fields of $D_1$ and $D_2$. The first derived flag of a distribution $D$ is denoted by $D^{(1)}$ and defined by $D^{(1)}=D+[D,D]$. By $\C{D}$, we denote the Cauchy characteristic distribution of $D$. It is spanned by all vector fields $c\in D$ which satisfy $[c,D]\subset D$. The symbols $\subset$ and $\supset$ are used in the sense that they also include equality. An integer beneath the symbol $\subset$ denotes the difference of the dimensions of the distributions involved, \eg $D_1\underset{k}{\subset}D_2$ means that $D_1\subset D_2$ and $\Dim{D_2}=\Dim{D_1}+k$. We make use of multi-indices, in particular, by $R=(r_1,r_2)$ we denote the unique multi-index associated to a flat output of a system with two inputs, where $r_j$ denotes the order of the highest derivative of $y^j$ needed to parameterize $x$ and $u$ by this flat output, \ie $y_{[R]}=(y^1,y^1_1,\ldots,y^1_{r_1},y^2,y^2_1,\ldots,y^2_{r_2})$. Furthermore, we define $R\pm c=(r_1\pm c,r_2\pm c)$ with an integer $c$, and $\#R=r_1+r_2$.
	
	\section{Preliminaries}\label{se:preliminaries}
		In this section, we summarize some results regarding flatness of two-input systems. Throughout, all functions and vector fields are assumed to be be smooth and all distributions are assumed to have locally constant dimension, we consider generic points only. Consider a nonlinear two-input system of the form
		\begin{align}\label{eq:sys}
			\begin{aligned}
				\dot{x}^i&=f^i(x,u)\,,&i=1,\ldots,n
			\end{aligned}
		\end{align}
		with $\Dim{x}=n$, $\Dim{u}=2$ and $\Rank{\partial_uf}=2$.
		\begin{definition}
			The two-input system \eqref{eq:sys} is called flat if there exist two differentially independent functions $y^j=\varphi^j(x,u,u_1,\ldots,u_q)$, $j=1,2$ such that locally
			\begin{align}\label{eq:param_r}
				\begin{aligned}
					x&=F_x(y,y_1,\ldots,y_{r-1})\\
					u&=F_u(y,y_1,\ldots,y_r)\,.
				\end{aligned}
			\end{align}
			The functions $y^j=\varphi^j(x,u,u_1,\ldots,u_q)$, $j=1,2$ are called the components of the flat output $y=\varphi(x,u,u_1,\ldots,u_q)$.
		\end{definition}
		Let $y=\varphi(x,u,u_1,\ldots,u_q)$ be a flat output of \eqref{eq:sys}. We define the multi-index $R=(r_1,r_2)$ where $r_j$ is the order of the highest derivative of the component $y^j$ of the flat output which explicitly occurs in \eqref{eq:param_r}. This multi-index can be shown to be unique and with this multi-index, the flat parameterization can be written in the form
		\begin{align}\label{eq:param}
			\begin{aligned}
				x&=F_{x}(y_{[R-1]})\\
				u&=F_{u}(y_{[R]})\,.
			\end{aligned}
		\end{align}
		The flat parameterization is a submersion (it degenerates to a diffeomorphism if and only if $y$ is a linearizing output). The difference of the dimensions of the domain and the codomain of \eqref{eq:param} is denoted by $d$, \ie $d=\#R+2-(n+2)=\#R-n$. In \cite{NicolauRespondek:2016} and \cite{NicolauRespondek:2016-2}, the number $\#R+2$ is called the differential weight of the flat output. (The differential weight of a flat output with difference $d$ is thus given by $n+2+d$.) The difference $d$ is the minimal dimension of an endogenous dynamic feedback needed to render \eqref{eq:sys} static feedback linearizable such that $y$ forms a linearizing output of the closed loop system. Such a linearizing endogenous feedback can be constructed systematically, see \eg \cite{FliessLevineMartinRouchon:1999}. If we have $d=0$, the map \eqref{eq:param} degenerates to a diffeomorphism and the system is static feedback linearizable with $y$ being a linearizing output. A flat output $y$ is called a minimal flat output if its difference is minimal compared to all other possible flat outputs of the system. We define the difference $d$ of a flat system to be the difference of a minimal flat output of the system. The difference $d$ of a system therefore measures its distance from static feedback linearizability, \ie $d$ is the minimal possible dimension of an endogenous dynamic feedback needed to render the system static feedback linearizable.
		
		For \eqref{eq:sys}, we define the distributions $D_0=\Span{\partial_{u^1},\partial_{u^2}}$ and $D_i=D_{i-1}+[f,D_{i-1}]$, $i\geq 1$ on the state and input manifold $\mathcal{X}\times\mathcal{U}$, where $f=f^i(x,u)\partial_{x^i}$. 
		\begin{theorem}\label{thm:sfl}
			The two-input system \eqref{eq:sys} is linearizable by static feedback if and only if all the distributions $D_i$ are involutive and $\Dim{D_{n-1}}=n+2$.
		\end{theorem}
		For a proof of this theorem, we refer to \cite{NijmeijervanderSchaft:1990}. For a system which meets the conditions of Theorem \ref{thm:sfl}, the linearizing outputs can be computed as follows. Let $s$ be the smallest integer such that $D_s=\TXU$. In case of $\Dim{D_{s-1}}=n$ (\ie $D_{s-1}$ is of codimension $2$), the sequence of involutive distributions is of the form
		\begin{align*}
			D_0\underset{2}{\subset}D_1\underset{2}{\subset}\ldots\underset{2}{\subset}D_{s-1}\underset{2}{\subset}\TXU
		\end{align*}
		and linearizing outputs are all pairs of functions $(\varphi^1,\varphi^2)$ which satisfy $\Span{\D\varphi^1,\D\varphi^2}=D_{s-1}^\perp$. However, if $\Dim{D_{s-1}}=n+1$ (\ie $D_{s-1}$ is of codimension $1$), the sequence is of the form 
		\begin{align*}
			D_0\underset{2}{\subset}D_1\underset{2}{\subset}\ldots\underset{2}{\subset}D_{l-1}\underset{2}{\subset}D_l\underset{1}{\subset}D_{l+1}\underset{1}{\subset}\ldots\underset{1}{\subset}D_{s-1}\underset{1}{\subset}\TXU\,,
		\end{align*}
		\ie there exists an integer $l$ from which on the sequence grows in steps of one. Linearizing outputs are then all pairs of functions $(\varphi^1,\varphi^2)$ which satisfy $\Span{\D\varphi^1}=D_{s-1}^\perp$ and $\Span{\D\varphi^1,\D\Lie_f\varphi^1,\ldots,\D\Lie_f^{s-l}\varphi^1,\D\varphi^2}=D_{l-1}^\perp$.\\	
		
		The sequential test for flatness with $d\leq 2$ proposed in \cite{GstottnerKolarSchoberl:2021}, as well as the distribution test for flatness with $d\leq 2$ which we propose in this contribution, both rely on the following crucial result regarding flat two-input systems with $d\leq 2$.
		\begin{theorem}\label{thm:linearizationByProlongations}
			A system \eqref{eq:sys} with $d\leq 2$ can be rendered static feedback linearizable by $d$-fold prolonging a suitably chosen (new) input after a suitable input transformation $\bar{u}=\Phi_u(x,u)$ has been applied.
		\end{theorem}
		A proof of this result is provided in Appendix \ref{ap:supplements}. 
	\section{Sequential Test}\label{se:seuqnetialTest}
		In this section, we briefly recapitulate the main idea of the necessary and sufficient condition for flatness with $d\leq 2$ in form of the sequential test proposed in \cite{GstottnerKolarSchoberl:2021}. For details, proofs and examples, we refer to \cite{GstottnerKolarSchoberl:2021}. Let $y$ be a minimal flat output with difference $0<d\leq 2$ of the system \eqref{eq:sys}. It can be shown that the assumption $0<d\leq 2$ implies the existence of an input transformation $\bar{u}=\Phi_u(x,u)$ such that the flat parameterization of the new inputs by the flat output $y$ is of the form $\bar{u}^1=\bar{F}_u^1(y_{[R-1]})$, $\bar{u}^2=\bar{F}_u^2(y_{[R]})$ (where $\bar{F}_u=\Phi_u\circ (F_x,F_u)$). Consider the system obtained by one-fold prolonging $\bar{u}^1$, \ie
		\begin{align*}
			\begin{aligned}
				\dot{x}&=f(x,\hat{\Phi}_u(x,\bar{u}))=\bar{f}(x,\bar{u})&&&\dot{\bar{u}}^1&=\bar{u}^1_1\,,
			\end{aligned}
		\end{align*}
		with the state $(x,\bar{u}^1)$ and the input $(\bar{u}^1_1,\bar{u}^2)$. The flat output $y$ of the original system is also a flat output of the prolonged system (and conversely, it can be shown that every flat output of the prolonged system is also a flat output of the original system). Since $\bar{u}^1=\bar{F}^1_{u}(y_{[R-1]})$, we have $\bar{u}^1_1=\bar{F}^1_{u_1}(y_{[R]})$ and thus, the domain of the parameterizing map of the prolonged system with respect to the flat output $y$ is still of dimension $\#R+2$, but its codomain grew by one, \ie $y$ as a minimal flat output of the prolonged system only has a difference of $d-1$. The main idea of the sequential test in \cite{GstottnerKolarSchoberl:2021} is to find such an input (they can indeed be found systematically), prolong it in order to obtain a system whose difference is $d-1$ (where $d\leq 2$ is the difference of the original system), and since by assumption $d\leq 2$, after at most two such steps, the procedure must yield a static feedback linearizable system, otherwise, the original system must have had a difference of $d\geq 3$.\\
		
		When applying the sequential test to a system \eqref{eq:sys}, in every step, a new system is derived by either splitting off a two-dimensional endogenous dynamic feedback or by adding a one-dimensional endogenous dynamic feedback (in form of a one-fold prolongation of a certain input). How the next system is derived from the current one is decided based on the distributions $D_0=\Span{\partial_{\bar{u}^1},\partial_{\bar{u}^2}}$ and $D_1=D_0+[f,D_0]$ of the current system. If $D_1$ is involutive, it can be straightened out by a suitable state transformation $\bar{x}=\Phi_x(x)$ in order to obtain a decomposition of the system into the form
		\begin{align}\label{eq:decompCase1}
			\begin{aligned}
				\Sigma_2:&&\dot{\bar{x}}_2^{i_2}&=\bar{f}_2^{i_2}(\bar{x}_2,\bar{x}_1)\,,&i_2&=1,\ldots,n-2\\
				\Sigma_1:&&\dot{\bar{x}}_1^{i_1}&=\bar{f}_1^{i_1}(\bar{x}_2,\bar{x}_1,u)\,,&i_1&=1,2\,.
			\end{aligned}
		\end{align}
		The procedure is then continued with the subsystem $\Sigma_2$ with the state $\bar{x}_2$ and the input $\bar{x}_1$, \ie we split off a two-dimensional endogenous dynamic feedback. It follows that $\Sigma_2$ has the same flat outputs with the same differences as the original system.\\
		
		If $D_1$ is non-involutive but $D_0\subset\C{D_1}$, it can be shown that the system allows an affine input representation (AI representation)
		\begin{align*}
			\begin{aligned}
				\dot{x}&=a(x)+b_1(x)u^1+b_2(x)u^2
			\end{aligned}
		\end{align*}
		with a non-involutive input distribution $\Span{b_1,b_2}$. Based on such an AI representation, an input transformation $\bar{u}^j=m_l^j(x)u^l$, $j,l=1,2$ can be derived such that if the system indeed has a difference of $d\leq 2$, the system obtained by one-fold prolonging the new input $\bar{u}^1$ has a difference of $d-1$, \ie in such a step a one-dimensional endogenous dynamic feedback is added to the system, and under the assumption $d\leq 2$, it can be show that the feedback modified system has a difference of $d-1$ only.\\

		Finally, if $D_1$ is non-involutive and $D_0\not\subset\C{D_1}$, the system allows at most a so called partial affine input representation (PAI representation)
		\begin{align*}
			\begin{aligned}
				\dot{x}&=a(x,\bar{u}^1)+b(x,\bar{u}^1)\bar{u}^2\,.
			\end{aligned}
		\end{align*}
		This form was introduced in \cite{SchlacherSchoberl:2013}. In \cite{KolarSchoberlSchlacher:2016}, it has been shown that the existence of a PAI representation is a necessary condition for flatness. For two input systems, an input transformation such that the system takes PAI form can be derived systematically (provided the system indeed allows a PAI representation, otherwise, we can conclude that the system is not flat). It can be shown that a system which does not allow an AI representation allows at most two fundamentally different PAI representations and that in case of $d\leq 2$, the non-affine occurring inputs $\bar{u}^1$ of these possibly existing two PAI representations are the candidates for inputs whose flat parameterization with respect to a minimal flat output involves derivatives up to order $R-1$ only. So in this case, the procedure is continued with the system obtained by one-fold prolonging the non-affine occurring input $\bar{u}^1$ (if the system indeed allows two fundamentally different PAI representations, we have to continue the procedure with both of them, \ie there may occur a branching point).\\
		
		As already mentioned, the sequential test has the drawback that it requires straightening out involutive distributions in order to achieve decompositions of the form \eqref{eq:decompCase1}. In fact, also the explicit computation of an input transformation such that a system takes PAI form requires straightening out an involutive distribution.
		
%		Indeed, if \eqref{eq:sys} is flat with $d\leq 2$, then according to Theorem \ref{thm:linearizationByProlongations}, there exists an input transformation $\bar{u}=\Phi_u(x,u)$ with inverse $u=\hat{\Phi}_u(x,\bar{u})$ such that the prolonged system
%		\begin{align*}
%			\begin{aligned}
%				\dot{x}&=f(x,\hat{\Phi}_u(x,\bar{u}))\\
%				\dot{\bar{u}}^1&=\bar{u}^1_1
%			\end{aligned}&&\text{or}&&
%			\begin{aligned}
%				\dot{x}&=f(x,\hat{\Phi}_u(x,\bar{u}))\\
%				\dot{\bar{u}}^1&=\bar{u}^1_1\\
%				\dot{\bar{u}}^1_1&=\bar{u}^1_2\,,
%			\end{aligned}
%		\end{align*}
%		depending on the actual value of $d$, is static feedback linearizable. Since the flat parameterization
%		\begin{align*}
%			\begin{aligned}
%				x&=F_{x}(y_{[R-1]})\\
%				u&=F_{u}(y_{[R]})
%			\end{aligned}
%		\end{align*}
%		with respect to this flat output is a submersion with a difference of the dimensions of the domain and the codomain of $d$, but the flat parameterization of the prolonged system is a diffeomorphism, it follows that the flat parameterization of the input $\bar{u}^1$ involves derivatives of $y$ up to order $R-d$ only. 
		
	\section{Main Results}\label{se:distributionTest}
		In this section we present our main results, which are easily verifiable necessary and sufficient conditions for flatness with a difference of $d\leq 2$, in the form of Theorem \ref{thm:d2} for the case $d=2$ and Theorem \ref{thm:d1} for the case $d=1$ below. These necessary and sufficient conditions overcome the computational drawbacks of the sequential test described in the previous section, instead of a sequence of systems, a certain sequence of distributions is constructed. The distributions constructed when applying these theorems are actually closely related with the distributions $D_0$ and $D_1$ of the individual systems constructed in the sequential test, based on which in the sequential test it is decided how the next system is computed from the current one. There is actually a one-to-one correspondence between the sequential test and the conditions of Theorem \ref{thm:d2} and \ref{thm:d1}. One could prove these theorems via this one-to-one correspondence, however, in this contribution, we provide self-contained proofs which do not rely on the sequential test. A detailed proof of Theorem \ref{thm:d2} is provided in Section \ref{se:proofd2}, for Theorem \ref{thm:d1}, a brief sketch of a proof is provided in Section \ref{se:proofd1}. As already mentioned, the case $d=1$ has been solved completely in \cite{NicolauRespondek:2016}, below we explain how our necessary and sufficient conditions in the form of Theorem \ref{thm:d1} are related with those provided in \cite{NicolauRespondek:2016}. The computation of flat outputs with $d\leq 2$ of systems which meet our conditions for flatness with $d\leq 2$ is addressed in Section \ref{se:computationOfFlatOutputs}.\\
		
		Assume that the system \eqref{eq:sys} is not static feedback linearizable. We then have the involutive distribution $D_0=\Span{\partial_{u^1},\partial_{u^2}}$ and can calculate the distributions $D_i=D_{i-1}+[f,D_{i-1}]$, $i=1,\ldots,k_1$ where $k_1$ is defined to be the smallest integer such that $D_{k_1}$ is non-involutive (its existence is assured by the assumption that the system is not static feedback linearizable).
		\begin{theorem}\label{thm:d2}
			The system \eqref{eq:sys} is flat with a difference of $d=2$ if and only if:
			\begin{enumerate}[label=\textbf{\theenumi.}]
				\item The distributions $D_i$, $i=1,\ldots,k_1$ have the dimensions $\Dim{D_i}=2(i+1)$.\label{d2:1}
			\end{enumerate}
			\begin{enumerate}[label=\textbf{\theenumi a.}]
				\setcounter{enumi}{1}
				\item Either $D_{k_1-1}\subset\C{D_{k_1}}$ and then:\label{d2:2a}
				\begin{enumerate}[label=\Alph*.]
					\item Either $\Dim{\overline{D}_{k_1}}=\Dim{D_{k_1}}+1$ and then $\Dim{[f,D_{k_1}]+\overline{D}_{k_1}}=\Dim{\overline{D}_{k_1}}+1$. Define $E_{k_1+1}=\overline{D}_{k_1}$ and continue with item \ref{d2:3b}\label{d2:2aA}
					\item Or $\Dim{\overline{D}_{k_1}}=\Dim{D_{k_1}}+2$ and then $[f,\C{D_{k_1}^{(1)}}]\subset D_{k_1}^{(1)}$. Define $E_{k_1+1}=D_{k_1}^{(1)}$ and continue with item \ref{d2:4a}\ref{d2:4a2} with $k_2=k_1+1$.\label{d2:2aB}
				\end{enumerate}
			\end{enumerate}
			\begin{enumerate}[label=\textbf{\theenumi b.}]
				\setcounter{enumi}{1}
				\item Or $D_{k_1-1}\not\subset\C{D_{k_1}}$ and then there exists a vector field $v_c\in D_{k_1-1}$, $v_c\notin D_{k_1-2}$ (take $D_{k_1-2}=0$ if $k_1=1$) such that $E_{k_1-1}\subset\C{E_{k_1}}$ where $E_{k_1-1}=D_{k_1-2}+\Span{v_c}$ and $E_{k_1}=D_{k_1-1}+\Span{[v_c,f]}$.\footnote{There exist at most two distinct such pairs of distributions, the construction is explained below. If indeed two exist, a branching point occurs and we have to continue with both of them.}\label{d2:2b}
			\end{enumerate}
			\begin{enumerate}[label=\textbf{\theenumi a.}]
				\setcounter{enumi}{2}
				\item Either $E_{k_1}$ is non-involutive (only \ref{d2:2b} can yield a non-involutive distribution $E_{k_1}$) and then:\label{d2:3a}
				\begin{enumerate}[label=\Roman*.]
					\item $\Dim{\overline{E}_{k_1}}=\Dim{E_{k_1}}+1$.\label{d2:3a1}
					\item $\Dim{[f,E_{k_1}]+\overline{E}_{k_1}}=\Dim{\overline{E}_{k_1}}+1$. Define $F_{k_1+1}=\overline{E}_{k_1}$ and continue with item \ref{d2:5}\label{d2:3a2}
				\end{enumerate}
			\end{enumerate}
			\begin{enumerate}[label=\textbf{\theenumi b.}]
				\setcounter{enumi}{2}
				\item Or $E_{k_1}$ is involutive and then:\label{d2:3b}
				\begin{enumerate}[label=\Roman*.]
					\item There exists a minimal integer $k_2$ such that $E_{k_2}$ is non-involutive where $E_i=E_{i-1}+[f,E_{i-1}]$.\label{d2:3b1}
					\item The distributions $E_i$ have the dimensions $\Dim{E_i}=2i+1$ for $i=k_1+1,\ldots,k_2$.\label{d2:3b2}
				\end{enumerate}
			\end{enumerate}
			\begin{enumerate}[label=\textbf{\theenumi a.}]
				\setcounter{enumi}{3}
				\item Either $E_{k_2-1}\subset\C{E_{k_2}}$ and then:\label{d2:4a}
				\begin{enumerate}[label=\Roman*.]
					\item $\Dim{\overline{E}_{k_2}}=\Dim{E_{k_2}}+1$.\label{d2:4a1}
					\item Either $\overline{E}_{k_2}=\TXU$, or $\Dim{[f,E_{k_2}]+\overline{E}_{k_2}}=\Dim{\overline{E}_{k_2}}+1$. Define $F_{k_2+1}=\overline{E}_{k_2}$.\label{d2:4a2}
				\end{enumerate}
			\end{enumerate}
			\begin{enumerate}[label=\textbf{\theenumi b.}]
			\setcounter{enumi}{3}
				\item Or $E_{k_2-1}\not\subset\C{E_{k_2}}$ and then $F_{k_2}=E_{k_2-1}+\C{E_{k_2}}$ is involutive.\label{d2:4b}
			\end{enumerate}
			\begin{enumerate}[label=\textbf{\theenumi.}]
			\setcounter{enumi}{4}
				\item All the distributions $F_i=F_{i-1}+[f,F_{i-1}]$ are involutive and there exists an integer $s$ such that $F_s=\TXU$.\label{d2:5}
			\end{enumerate}
		\end{theorem}
	 	In Theorem \ref{thm:d2}, we have several junctions, which is graphically illustrated in Figure \ref{fig:d2}.
		\begin{figure}[h!]
			\centering
			\includegraphics[scale=1.3]{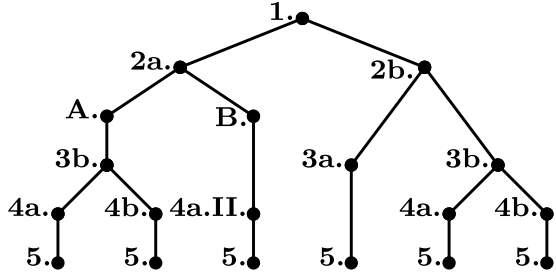}
			\caption{Overview of the possible paths in Theorem \ref{thm:d2}.}
			\label{fig:d2}
		\end{figure}
	
		Regarding flatness with a difference of $d=1$, we have the following result.
		\begin{theorem}\label{thm:d1}
			The system \eqref{eq:sys} is flat with a difference of $d=1$ if and only if:
			\begin{enumerate}[label=\textbf{\theenumi.}]
				\item The distributions $D_i$, $i=1,\ldots,k_1$ have the dimensions $\Dim{D_i}=2(i+1)$.\label{d1:1}
			\end{enumerate}
			\begin{enumerate}[label=\textbf{\theenumi a.}]
				\setcounter{enumi}{1}
				\item Either $D_{k_1-1}\subset\C{D_{k_1}}$ and then:\label{d1:2a}
				\begin{enumerate}[label=\Roman*.]
					\item $\Dim{\overline{D}_{k_1}}=\Dim{D_{k_1}}+1$.\label{d1:2a1}
					\item Either $\overline{D}_{k_1}=\TXU$, or $\Dim{[f,D_{k_1}]+\overline{D}_{k_1}}=\Dim{\overline{D}_{k_1}}+1$. Define $E_{k_1+1}=\overline{D}_{k_1}$.\label{d1:2a2}
				\end{enumerate}
			\end{enumerate}
			\begin{enumerate}[label=\textbf{\theenumi b.}]
				\setcounter{enumi}{1}
				\item Or $D_{k_1-1}\not\subset\C{D_{k_1}}$ and then $E_{k_1}=D_{k_1-1}+\C{D_{k_1}}$ is involutive.\label{d1:2b}
			\end{enumerate}
			\begin{enumerate}[label=\textbf{\theenumi.}]
				\setcounter{enumi}{2}
				\item All the distributions $E_i=E_{i-1}+[f,E_{i-1}]$ are involutive and there exists an integer $s$ such that $E_s=\TXU$.\label{d1:3}
			\end{enumerate}
		\end{theorem}
		In Theorem \ref{thm:d1}, we have exactly one junction, which is graphically illustrated in Figure \ref{fig:d1}.
		\begin{figure}[h!]
			\centering
			\includegraphics[scale=1.3]{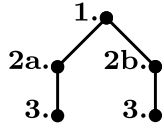}
			\caption{Overview of the possible paths in Theorem \ref{thm:d1}.}
			\label{fig:d1}
		\end{figure}
	
		Note that the items \ref{d2:3b} to \ref{d2:5} of Theorem \ref{thm:d2} in fact coincide with the items \ref{d1:1} to \ref{d1:3} of Theorem \ref{thm:d1} when $D_i$ and $k_1$ are replaced by $E_i$ and $k_2$. In \cite{NicolauRespondek:2016}, necessary and sufficient conditions for the case $d=1$ are provided via the Theorems 3.3 and 3.4 therein. These theorems are stated for the control affine case, but this is actually no restriction. It can be shown that the control affine system obtained by prolonging both inputs of a general nonlinear control system of the form \eqref{eq:sys}, \ie
		\begin{align*}
			\begin{aligned}
				\dot{x}&=f(x,u)&&&\dot{u}^1&=u^1_1&&&\dot{u}^2&=u^2_1
			\end{aligned}
		\end{align*}
		with the state $(x,u^1,u^2)$ and the input $(u^1_1,u^2_1)$, is flat with a certain difference $d$ if and only if the original system is flat with the same difference $d$. In fact, a flat output of the original system with a certain difference $d$ is also a flat output of the prolonged system with the same difference $d$ and vice versa. 
		In Theorem \ref{thm:d1}, we have exactly one junction (see Figure \ref{fig:d1}). This distinction of cases between $D_{k_1-1}\subset\C{D_{k_1}}$ and $D_{k_1-1}\not\subset\C{D_{k_1}}$ is motivated by the sequential test of the previous section (it corresponds to the distinction between AI form and PAI form in the sequential test). Our necessary and sufficient conditions for flatness with $d=1$ as stated in Theorem \ref{thm:d1} are very similar to those stated in \cite{NicolauRespondek:2016}. The main difference is in fact that in \cite{NicolauRespondek:2016} a distinction of cases is done between $\overline{D}_{k_1}\neq\TXU$ (Theorem 3.3 therein) and $\overline{D}_{k_1}=\TXU$ (Theorem 3.4 therein), instead of a distinction of cases between $D_{k_1-1}\subset\C{D_{k_1}}$ and $D_{k_1-1}\not\subset\C{D_{k_1}}$ as it is done in Theorem \ref{thm:d1}.\\
		
		It can easily be shown that all the distributions and all the conditions in Theorem \ref{thm:d2} and \ref{thm:d1} are invariant with respect to regular input transformations $\bar{u}=\Phi_u(x,u)$, \ie although the vector field $f=f^i(x,u)\partial_{x^i}$ associated with \eqref{eq:sys} and the vector field $\bar{f}=\bar{f}^i(x,\bar{u})\partial_{x^i}$ associated with the feedback modified system $\dot{x}^i=\bar{f}^i(x,\bar{u})$, where $\bar{f}^i(x,\bar{u})=f^i(x,\hat{\Phi}_u(x,\bar{u}))$ with the inverse $u=\hat{\Phi}_u(x,\bar{u})$ of the input transformation $\bar{u}=\Phi_u(x,u)$, are in general only equal modulo $D_0=\Span{\partial_{u}}=\Span{\partial_{\bar{u}}}$, the distributions $D_i$, $E_i$ and $F_i$ constructed from them in the above theorems coincide.
		\subsection{Verification of the Conditions}\label{se:verificatoinOfConditions}
			All the conditions of Theorem \ref{thm:d2} and \ref{thm:d1} are easily verifiable and require differentiation and algebraic operations only. Item \ref{d2:2b} of Theorem \ref{thm:d2} can be verified as follows. Choose any pair of vector fields $v_1,v_2\in D_{k_1-1}$ such that $D_{k_1-1}=D_{k_1-2}+\Span{v_1,v_2}$. Any vector field $v_c\in D_{k_1-1}$, $v_c\notin D_{k_1-2}$ can then be written as a non-trivial linear combination $v_c=\alpha^1v_1+\alpha^2v_2\Mod D_{k_1-2}$. Since $E_{k_1-1}=D_{k_1-2}+\Span{v_c}$ and $E_{k_1}=D_{k_1-1}+\Span{[v_c,f]}$, the condition $E_{k_1-1}\subset\C{E_{k_1}}$ implies $v_c\in\C{E_{k_1}}$ and in turn $[v_c,[v_c,f]]\in E_{k_1}$. Since for any $v_c$, we by construction have $E_{k_1}\subset D_{k_1}$, the condition $[v_c,[v_c,f]]\in E_{k_1}$ implies $[v_c,[v_c,f]]\in D_{k_1}$, which yields the necessary condition 
			\begin{align}\label{eq:algWithin}
				\begin{aligned}
					(\alpha^1)^2[v_1,[v_1,f]]+2\alpha^1\alpha^2[v_1,[v_2,f]]+(\alpha^2)^2[v_2,[v_2,f]]&\overset{!}{\in} D_{k_1}
				\end{aligned}
			\end{align}
			(where we have used that $D_{k_1-2}\subset\C{D_{k_1}}$, which can be shown based on the Jacobi identity). The following lemma states a crucial property of \eqref{eq:algWithin}.
			\begin{lemma}\label{lem:algWithinUniqueness}
				The condition \eqref{eq:algWithin} admits at most two independent non-trivial solutions $\alpha^j$.
			\end{lemma}
			A proof of this lemma is provided in Appendix \ref{ap:supplements}. (A similar result has been proven in \cite{GstottnerKolarSchoberl:2020c} in the context of a certain structurally flat triangular form. The result which we prove here is more general.) Since \eqref{eq:algWithin} admits at most two independent non-trivial solutions, there also exist at most two vector fields $v_c=\alpha^1v_1+\alpha^2v_2$ which are not collinear$\Mod D_{k_1-2}$ and meet the above criterion. Since in $E_{k_1-1}=D_{k_1-2}+\Span{v_c}$ and $E_{k_1}=D_{k_1-1}+\Span{[v_c,f]}$ only the direction of $v_c$ and only the part of $v_c$ which is not contained in $D_{k_1-2}$ matters, there also exist at most two distinct such pairs of distributions.	
		\subsection{Computation of Flat Outputs}\label{se:computationOfFlatOutputs}
			The Theorems \ref{thm:d2} and \ref{thm:d1} allow us to check whether a system is flat with a difference of $d\leq 2$. Regarding the computation of the corresponding flat outputs with $d\leq 2$, we have the following result.
			\begin{theorem}\label{thm:flatOutputs}
				Assume that the system \eqref{eq:sys} meets the conditions of Theorem \ref{thm:d2} or \ref{thm:d1}. Flat outputs with $d=1$ or $d=2$ of the system can then be determined from the sequence of involutive distributions $E_i$ or $F_i$ the same way as linearizing outputs are determined form the sequence of involutive distributions $D_i$ involved in the test for static feedback linearizability in Theorem \ref{thm:sfl}.
			\end{theorem}
			\begin{proof}
				The sufficiency parts of the proofs of Theorem \ref{thm:d2} and \ref{thm:d1} are done constructively. Based on the distributions involved in the conditions of the theorems, for each case a certain coordinate transformation such that the system takes a structurally flat triangular form is derived. The top variables in these triangular forms are then flat outputs with $d=1$ or $d=2$. For details, see sufficiency parts of the proofs of Theorem \ref{thm:d2} and \ref{thm:d1}. 
			\end{proof}
			It may happen that for the computation of flat outputs as stated in Theorem \ref{thm:flatOutputs}, a distribution which does not explicitly occur in Theorem \ref{thm:d2} or \ref{thm:d1} is needed. To be precise, if in Theorem \ref{thm:d1}, the conditions \ref{d1:2a} apply and we have $E_{k_1+1}=\TXU$ or $E_{k_1+1}\underset{1}{\subset}E_{k_1+2}$, we have to construct an involutive distribution $E_{k_1}$ which satisfies $D_{k_1-1}\underset{1}{\subset}E_{k_1}\underset{1}{\subset}D_{k_1}$ and $[f,E_{k_1}]\subset\overline{D}_{k_1}$ for the computation of flat outputs. Such a distribution always exists and despite from the case $E_{k_1+1}=\TXU$, it is also unique. 
			
			The construction is as follows. In case of $E_{k_1+1}=\TXU$, choose any function $\psi$ whose differential $\D\psi\neq 0$ annihilates $D_{k_1-1}$ and choose furthermore any pair of vector fields $v_1,v_2$ which complete $D_{k_1-1}$ to $D_{k_1}$, \ie such that $D_{k_1}=D_{k_1-1}+\Span{v_1,v_2}$. The distribution $E_{k_1}=D_{k_1-1}+\Span{(\D\psi\rfloor v_2)v_1-(\D\psi\rfloor v_1)v_2}$ can then be shown to be involutive, different choices for $\psi$ lead in general to different distributions $E_{k_1}$ (but different choices for $v_1,v_2$ have no effect). (A flat output with $d=1$ is then formed by any pair of functions $(\varphi^1,\varphi^2)$ satisfying $\Span{\D\varphi^1,\D\varphi^2}=E_{k_1}^\perp$, where we can always choose one of the components equal to $\psi$ since by construction $\D\psi\in E_{k_1}^\perp$.)
			
			In case of $E_{k_1+1}\neq \TXU$, the condition $\Dim{[f,D_{k_1}]+\overline{D}_{k_1}}=\Dim{\overline{D}_{k_1}}+1$ implies the existence of a vector field $v\in D_{k_1}$, $v\notin D_{k_1-1}$ such that $[v,f]\in\overline{D}_{k_1}$. The direction of $v$ is modulo $D_{k_1-1}$ unique and it follows that the distribution $E_{k_1}=D_{k_1-1}+\Span{v}$ constructed from it is involutive. (A flat output with $d=1$ is then formed by any pair of functions $(\varphi^1,\varphi^2)$ satisfying $\Span{\D\varphi^1}=E_{s-1}^\perp$ and $\Span{\D\varphi^1,\D\Lie_f\varphi^1,\ldots,\D\Lie_f^{s-k_1-1}\varphi^1,\D\varphi^2}=E_{k_1}^\perp$.)\\
			
			Similarly, if in Theorem \ref{thm:d2}, the conditions \ref{d2:4a} (or \ref{d2:2a}\ref{d2:2aB} and \ref{d2:4a}\ref{d2:4a2}, in which case we define $E_{k_2-1}=\C{D_{k_1}^{(1)}}$) apply and we have $F_{k_2+1}=\TXU$ or $F_{k_2+1}\underset{1}{\subset}F_{k_2+2}$, we have to construct an involutive distribution $F_{k_2}$ which satisfies $E_{k_2-1}\underset{1}{\subset}F_{k_2}\underset{1}{\subset}E_{k_2}$ and $[f,F_{k_2}]\subset\overline{E}_{k_2}$ for the computation of flat outputs. This can be done as just explained for Theorem \ref{thm:d1}, simply replace $D_{k_1-1}$ and $D_{k_1}$ by $E_{k_2-1}$ and $E_{k_2}$.\\
			
			Finally, if in Theorem \ref{thm:d2}, the conditions \ref{d2:3a} apply and we have $F_{k_1+1}\underset{1}{\subset}F_{k_1+2}$, we have to construct an involutive distribution $F_{k_1}$ which satisfies $E_{k_1-1}\underset{1}{\subset}F_{k_1}\underset{1}{\subset}E_{k_1}$ and $[f,F_{k_1}]\subset\overline{E}_{k_1}$ for the computation of flat outputs. The construction is again the same as just explained for Theorem \ref{thm:d1}, simply replace $D_{k_1-1}$ and $D_{k_1}$ by $E_{k_1-1}$ and $E_{k_1}$. Since we always have $\overline{E}_{k_1}\neq\TXU$ in this case, the distribution $F_{k_1}$ is always unique.
%		The condition $\Dim{[f,E_{k_1}]+\overline{E}_{k_1}}=\Dim{E_{k_1}}+1$ guarantees the existence of a vector field $v\in E_{k_1}$, $v\notin\C{E_{k_1}}$ such that $[v,f]\in\overline{E}_{k_1}$. The direction of $v$ is modulo $\C{E_{k_1}}$ unique and it follows that the distribution $E_{k_1}=\C{E_{k_1}}+\Span{v}$ constructed from it is involutive. 	
	\section{Examples}\label{se:examples}
		In the following we apply our results to some examples. We focus on the case $d=2$, as the novelty of this contribution are the results regarding the case $d=2$, which also cover the cases which cannot be handled with the results in \cite{NicolauRespondek:2016-2} (none of the following examples meets Assumption 2 therein). It should again be pointed out that from a computational point of view, the necessary and sufficient conditions for flatness with $d\leq 2$ in the form of Theorem \ref{thm:d2} and \ref{thm:d1} are a major improvement over the necessary and sufficient conditions in form of the sequential test proposed in \cite{GstottnerKolarSchoberl:2021}. 
		\subsection{VTOL}
			Consider the model of a planar VTOL aircraft
			\begin{align}\label{eq:vtol}
				\begin{aligned}
					\dot{x}&=v_x&&&\dot{v}_x&=\epsilon\cos(\theta)u^2-\sin(\theta)u^1\\
					\dot{z}&=v_z&&&\dot{v}_z&=\cos(\theta)u^1+\epsilon\sin(\theta)u^2-1\\
					\dot{\theta}&=\omega&&&\dot{\omega}&=u^2\,.
				\end{aligned}
			\end{align}
			This system is also treated in \eg \cite{FliessLevineMartinRouchon:1999}, \cite{SchoberlRiegerSchlacher:2010}, \cite{SchoberlSchlacher:2011}, \cite{GstottnerKolarSchoberl:2020c} or \cite{GstottnerKolarSchoberl:2021}. The distributions $D_0=\Span{\partial_{u^1},\partial_{u^2}}$ and
			\begin{align*}
				\begin{aligned}
					D_1&=\Span{\partial_{u^1},\partial_{u^2},-\sin(\theta)\partial_{v_x}+\cos(\theta)\partial_{v_z},\epsilon\cos(\theta)\partial_{v_x}+\epsilon\sin(\theta)\partial_{v_z}+\partial_\omega}
				\end{aligned}
			\end{align*}
			are involutive,
			\begin{align*}
				\begin{aligned}
					D_2&=\Span{\partial_{u^1},\partial_{u^2},-\sin(\theta)\partial_{v_x}+\cos(\theta)\partial_{v_z},\epsilon\cos(\theta)\partial_{v_x}+\epsilon\sin(\theta)\partial_{v_z}+\partial_\omega,\sin(\theta)\partial_x\\
					&\hspace{3em}-\cos(\theta)\partial_z-\omega\cos(\theta)\partial_{v_x}-\omega\sin(\theta)\partial_{v_z},\epsilon\cos(\theta)\partial_x+\epsilon\sin(\theta)\partial_z+\partial_\theta\\
					&\hspace{3em}+\epsilon\omega\sin(\theta)\partial_{v_x}-\epsilon\omega\cos(\theta)\partial_{v_z}}
				\end{aligned}
			\end{align*}
			is non-involutive, so we have $k_1=2$. With these distributions, item \ref{d2:1} of Theorem \ref{thm:d2} is met. The Cauchy characteristic distribution of $D_2$ follows as $\C{D_2}=\Span{\partial_{u^1},\partial_{u^2}}=D_0$ and thus $D_1\not\subset\C{D_2}$. So we are in item \ref{d2:2b} and have to construct a vector field $v_c\in D_1$, $v_c\notin D_0$ such that $E_1\subset\C{E_2}$ with $E_1=D_0+\Span{v_c}$ and $E_2=D_1+\Span{[v_c,f]}$. So we set $v_c=\alpha^1v_1+\alpha^2v_2$ with arbitrary vector fields $v_1,v_2$ which complete $D_0$ to $D_1$, \eg $v_1=-\sin(\theta)\partial_{v_x}+\cos(\theta)\partial_{v_z}$, $v_2=\epsilon\cos(\theta)\partial_{v_x}+\epsilon\sin(\theta)\partial_{v_z}+\partial_\omega$, and determine $\alpha^1,\alpha^2$ from \eqref{eq:algWithin}, which in the particular case reads
			\begin{align}\label{eq:algVtol}
				\begin{aligned}
					2\alpha^1\alpha^2(\cos(\theta)\partial_{v_x}+\sin(\theta)\partial_{v_z})+(\alpha^2)^2(2\epsilon\sin(\theta)\partial_{v_x}-2\epsilon\cos(\theta)\partial_{v_z})&\overset{!}{\in} D_2\,,
				\end{aligned}
			\end{align}
			and admits the two independent solutions $\alpha^1=\lambda$, $\alpha^2=0$ and $\alpha^1=0$, $\alpha^2=\lambda$, both with an arbitrary function $\lambda\neq 0$. We can choose $\lambda=1$ since only the direction of $v_c=\alpha^1v_1+\alpha^2v_2$ matters, so we simply have $v_{c,1}=v_1=-\sin(\theta)\partial_{v_x}+\cos(\theta)\partial_{v_z}$ and $v_{c,2}=v_2=\epsilon\cos(\theta)\partial_{v_x}+\epsilon\sin(\theta)\partial_{v_z}+\partial_\omega$. With both of these vector fields we indeed have $E_{1,j}\subset\C{E_{2,j}}$, where $E_{1,j}=D_0+\Span{v_{c,j}}$ and
			\begin{align}\label{eq:vtolE21}
				\begin{aligned}
					E_{2,1}&=\Span{\partial_{u^1},\partial_{u^2},-\sin(\theta)\partial_{v_x}+\cos(\theta)\partial_{v_z},\epsilon\cos(\theta)\partial_{v_x}+\epsilon\sin(\theta)\partial_{v_z}+\partial_\omega,\\
					&\hspace{3em}\sin(\theta)\partial_x-\cos(\theta)\partial_z-\omega\cos(\theta)\partial_{v_x}-\omega\sin(\theta)\partial_{v_z}}
				\end{aligned}
			\end{align}
			and
			\begin{align}\label{eq:vtolE22}
				\begin{aligned}
					E_{2,2}&=\Span{\partial_{u^1},\partial_{u^2},-\sin(\theta)\partial_{v_x}+\cos(\theta)\partial_{v_z},\epsilon\cos(\theta)\partial_{v_x}+\epsilon\sin(\theta)\partial_{v_z}+\partial_\omega,\\
					&\hspace{3em}\epsilon\cos(\theta)\partial_x+\epsilon\sin(\theta)\partial_z+\partial_\theta+\epsilon\omega\sin(\theta)\partial_{v_x}-\epsilon\omega\cos(\theta)\partial_{v_z}}\,.
				\end{aligned}
			\end{align}
			We thus have a branching point and have to continue with both of these distributions (we will be able to discard one of them in just a moment). The distributions $E_{2,j}$ are non-involutive, so we are in item \ref{d2:3a} Both meet the condition \ref{d2:3a}\ref{d2:3a1}, \ie $\Dim{\overline{E}_{2,j}}=\Dim{E_{2,j}}+1$. However, \ref{d2:3a}\ref{d2:3a2} is only satisfied by $E_{2,2}$, for $E_{2,1}$ we have $\Dim{[f,E_{2,1}]+\overline{E}_{2,1}}=\Dim{\overline{E}_{2,1}}+2$, so we can discard this branch and continue with $E_{2,2}$ only, where for ease of notation, we drop the second subscript from now on, \ie $E_2=E_{2,2}$. According to item \ref{d2:3a}\ref{d2:3a2}, we have $F_3=\overline{E}_2=\Span{\partial_{u^1},\partial_{u^2},\partial_\omega,\partial_{v_x},\partial_{v_z},\epsilon\cos(\theta)\partial_x+\epsilon\sin(\theta)\partial_z+\partial_\theta}$. Continuing this sequence as stated in item \ref{d2:5}, we obtain $F_4=\TXU$, so the conditions of item \ref{d2:5} are also met and we conclude that the system \eqref{eq:vtol} is flat with a difference of $d=2$. In conclusion, the system meets the items \ref{d2:1}, \ref{d2:2b}, \ref{d2:3a}, \ref{d2:5} (which corresponds to the 4th path from the left in Figure \ref{fig:d2}).\\
			
			According to Theorem \ref{thm:flatOutputs}, flat outputs with $d=2$ of the VTOL can thus be computed from the distributions $F_i$ the same way as linearizing outputs are determined from the sequence of involutive distributions involved in the test for static feedback linearizability. We have $F_3\underset{2}{\subset}F_4=\TXU$, flat outputs with $d=2$ are thus all pairs of functions $(\varphi^1,\varphi^2)$ satisfying $\Span{\D\varphi^1,\D\varphi^2}=F_3^\perp$. We have $F_3^\perp=\Span{\D x-\epsilon\cos(\theta)\D\theta,\D z-\epsilon\sin(\theta)\D\theta}$ and thus \eg $\varphi^1=x-\epsilon\sin(\theta)$, $\varphi^2=z+\epsilon\cos{\theta}$.
		\subsection{Academic Example 1}
			Consider the system
			\begin{align}\label{eq:sin}
				\begin{aligned}
					\dot{x}^1&=u^1\\
					\dot{x}^2&=u^2\\
					\dot{x}^3&=\sin(\tfrac{u^1}{u^2})\,,
				\end{aligned}
			\end{align}
			also considered in \cite{Levine:2009}, \cite{SchoberlSchlacher:2014}, \cite{GstottnerKolarSchoberl:2020c} and \cite{GstottnerKolarSchoberl:2021}. The distribution $D_0=\Span{\partial_{u^1},\partial_{u^2}}$ is involutive,
			\begin{align*}
				D_1&=\Span{\partial_{u^1},\partial_{u^2},\partial_{x^1}+\tfrac{1}{u^2}\cos(\tfrac{u^1}{u^2})\partial_{x^2},\partial_{x^2}-\tfrac{u^1}{(u^2)^2}\cos(\tfrac{u^1}{u^2})\partial_{x^3}}
			\end{align*}
			is non-involutive, so we have $k_1=1$ and item \ref{d2:1} is met. We have $D_0\not\subset\C{D_1}$, so we are in item \ref{d2:2b} and have to construct a vector field $v_c\in D_0=\Span{\partial_{u^1},\partial_{u^2}}$, \ie $v_c=\alpha^1\partial_{u^1}+\alpha^2\partial_{u^2}$, such that $E_0\subset\C{E_1}$ where $E_0=\Span{v_c}$ and $E_1=D_0+\Span{[v_c,f]}$. For that, we solve \eqref{eq:algWithin}, which in the particular case yields
			\begin{align*}
				\begin{aligned}
					&(\alpha^1)^2\sin(\tfrac{u^1}{u^2})(u^2)^2+2\alpha^1\alpha^2(\cos(\tfrac{u^1}{u^2})u^2-\sin(\tfrac{u^1}{u^2})u^1)u^2+\\
					&\hspace{12em}(\alpha^2)^2(\sin(\tfrac{u^1}{u^2})u^1-2\cos(\tfrac{u^1}{u^2})u^2)u^1\overset{!}{=}0\,,
				\end{aligned}
			\end{align*}
			and admits the two independent solutions $\alpha^1=\lambda u^1$, $\alpha^2=\lambda u^2$ and $\alpha^1=\lambda(u^1\tan(\tfrac{u^1}{u^2})-2u^2)$, $\alpha^2=\lambda u^2\tan(\tfrac{u^1}{u^2})$, both with an arbitrary function $\lambda\neq 0$, \eg $\lambda=1$ since only the direction of $v_c=\alpha^1\partial_{u^1}+\alpha^2\partial_{u^2}$ matters. It can easily be checked that only the vector field $v_c=u^1\partial_{u^1}+u^2\partial_{u^2}$, obtained form the first solution, satisfies $v_c\in\C{E_1}$, where $E_1=D_0+\Span{[v_c,f]}$ follows as $E_1=\Span{\partial_{u^1},\partial_{u^2},u^1\partial_{x^1}+u^2\partial_{x^2}}$. The distribution $E_1$ is non-involutive, so we are in item \ref{d2:3a} We have $\overline{E}_1=\Span{\partial_{u^1},\partial_{u^2},\partial_{x^1},\partial_{x^2}}$.
			The conditions  $\Dim{\overline{E}_1}=\Dim{E_1}+1$ and $\Dim{[f,E_1]+\overline{E}_1}=\Dim{\overline{E}_1}+1$ are met. According to item \ref{d2:3a}\ref{d2:3a2}, we have $F_2=\overline{E}_1$. Continuing this sequence as state in item \ref{d2:5}, we obtain $F_3=\TXU$, so the conditions of item \ref{d2:5} are also met and we conclude that the system \eqref{eq:sin} is flat with a difference of $d=2$. In conclusion, the system meets the items \ref{d2:1}, \ref{d2:2b}, \ref{d2:3a}, \ref{d2:5} (which corresponds to the 4th path from the left in Figure \ref{fig:d2}).\\
			
			According to Theorem \ref{thm:flatOutputs}, flat outputs with $d=2$ of the the system \eqref{eq:sin} can thus be computed from the distributions $F_i$ the same way as linearizing outputs are determined from the sequence of involutive distributions involved in the test for static feedback linearizability. However, in contrast to the previous example, not all the distributions needed for computing flat outputs explicitly occur in Theorem \ref{thm:d2}. We have $F_2\underset{1}{\subset}F_3=\TXU$, from which we find only one component $\varphi^1$ of the flat outputs, \ie $\Span{\D\varphi^1}=F_2^\perp=\Span{\D x^3}$ from which \eg $\varphi^1=x^3$ follows. In order to find a second component $\varphi^2$, we have to complete the sequence to $F_1\underset{2}{\subset}F_2\underset{1}{\subset}F_3=\TXU$, as stated below Theorem \ref{thm:flatOutputs} and then, we find $\varphi^2$ from $\Span{\D\varphi^1,\D\Lie_f\varphi^1,\D\varphi^2}=F_1^\perp$. We have $E_1=E_0+\Span{\partial_{u^1},u^1\partial_{x^1}+u^2\partial_{x^2}}$ and it follows that $[f,u^1\partial_{x^1}+u^2\partial_{x^2}]\in\overline{E}_{k_1}$. We thus have $F_1=E_0+\Span{u^1\partial_{x^1}+u^2\partial_{x^2}}=\Span{u^1\partial_{u^1}+u^2\partial_{u^2},u^1\partial_{x^1}+u^2\partial_{x^2}}$ and in turn $F_1^\perp=\Span{\D x^3,u^2\D u^1-u^1\D u^2,u^2\D x^1-u^1\D x^2}$ from which a possible second component follows as $\varphi^2=x^1-x^2\tfrac{u^1}{u^2}$. In conclusion, we have derived the flat output $\varphi^1=x^3$, $\varphi^2=x^1-x^2\tfrac{u^1}{u^2}$.\\
			
			Consider the following two examples
			\begin{align}\label{eq:productAndSqrt}
				&\begin{aligned}
					\dot{x}^1&=u^1\\
					\dot{x}^2&=u^2\\
					\dot{x}^3&=u^1u^2\,,
				\end{aligned}&
				\begin{aligned}
					\dot{x}^1&=u^1\\
					\dot{x}^2&=u^2\\
					\dot{x}^3&=\sqrt{u^1u^2}\,,
				\end{aligned}
			\end{align}
			also treated in \eg \cite{SchoberlSchlacher:2010} and \cite{SchoberlSchlacher:2015}. The first one of these two systems can be shown to be flat with a difference of $d=2$, where again the items \ref{d2:1}, \ref{d2:2b}, \ref{d2:3a}, \ref{d2:5} (which again corresponds to the 4th path from the left in Figure \ref{fig:d2}) are met. Item \ref{d2:2b} yields two different pairs of distributions $E_0$ and $E_1$ for this system, namely $E_{0,1}=\Span{\partial_{u^1}}$, $E_{1,1}=\Span{\partial_{u^1},\partial_{u^2},\partial_{x^1}+u^2\partial_{x^3}}$ and $E_{0,2}=\Span{\partial_{u^2}}$, $E_{1,2}=\Span{\partial_{u^1},\partial_{u^2},\partial_{x^2}+u^1\partial_{x^3}}$. For both branches the items \ref{d2:3a} and \ref{d2:5} are met, and we obtain $(x^2,x^3-x^1u^2)$ and $(x^1,x^3-x^2u^1)$ as possible flat outputs with $d=2$.
			
			However, the second systems in \eqref{eq:productAndSqrt} is a negative example, it does not meet the conditions for flatness with $d\leq 2$\footnote{We have $k_1=1$ and $D_0\not\subset\C{D_1}$. The conditions of Theorem \ref{thm:d1} are not met since $E_1$ of item \ref{d1:2b} is non-involutive. The conditions of Theorem \ref{thm:d2} are not met since $\Dim{\overline{E}_1}=\Dim{E_1}+2$, which violates \ref{d2:3a}\ref{d2:3a1}}, so we can conclude that if the system is flat, it must have a difference of $d\geq 3$. (The system is indeed flat with a difference of $d=3$, in \eg \cite{SchoberlSchlacher:2015} a corresponding flat output with $d=3$ has been derived.)
		\subsection{Academic Example 2}
			Consider the system
			\begin{align}\label{eq:academic}
				\begin{aligned}
					\dot{x}^1&=\arcsin(\tfrac{u^1+u^2}{x^2})-x^4&&&\dot{x}^3&=u^1\\
					\dot{x}^2&=x^4&&&\dot{x}^4&=u^2\,.
				\end{aligned}
			\end{align}
			The distribution $D_0=\Span{\partial_{u^1},\partial_{u^2}}$ is involutive,
			\begin{align*}
				D_1&=\Span{\partial_{u^1},\partial_{u^2},\partial_{x^1}+\sqrt{(x^2)^2-(u^1+u^2)^2}\partial_{x^3},\partial_{x^1}+\sqrt{(x^2)^2-(u^1+u^2)^2}\partial_{x^4}}
			\end{align*}
			is non-involutive, so we have $k_1=1$ and item \ref{d2:1} is met. We have $D_0\not\subset D_1$, so we are in item \ref{d2:2b} and have to construct a vector field $v_c\in D_0$, \ie $v_c=\alpha^1\partial_{u^1}+\alpha^2\partial_{u^2}$, such that $E_0\subset\C{E_1}$ where $E_0=\Span{v_c}$ and $E_1=D_0+\Span{[v_c,f]}$. For that, we solve \eqref{eq:algWithin}, which in the particular case yields
			\begin{align*}
				\begin{aligned}
					(\alpha^1+\alpha^2)^2&\overset{!}{=}0
				\end{aligned}
			\end{align*}
			and has the up to a multiplicative factor unique solution $\alpha^1=1$, $\alpha^2=-1$. The vector field $v_c=\partial_{u^1}-\partial_{u^2}$ obtained from this solution indeed meets $v_c\in\C{E_1}$ with $E_1=D_0+\Span
			{[v_c,f]}=\Span{\partial_{u^1},\partial_{u^2},\partial_{x^3}-\partial_{x^4}}$. The distribution $E_1$ is involutive, so we are in item \ref{d2:3b} The distribution
			\begin{align*}
				E_2&=\Span{\partial_{u^1},\partial_{u^2},\partial_{x^1}+\sqrt{(x^2)^2-(u^1+u^2)^2}\partial_{x^3},\partial_{x^1}+\sqrt{(x^2)^2-(u^1+u^2)^2}\partial_{x^4},\partial_{x^1}-\partial_{x^2}}
			\end{align*}
			is non-involutive, \ie $k_2=2$ and item \ref{d2:3b}\ref{d2:3b2} is met. We have $E_1\not\subset\C{E_2}$, so we are in item \ref{d2:4b} According to this item, we have $F_2=E_1+\C{E_2}$, which evaluates to			\begin{align*}
				F_2&=\Span{\partial_{u^1},\partial_{u^2},\partial_{x^1}-\partial_{x^2},\partial_{x^3}-\partial_{x^4}}
			\end{align*} 
			and is indeed involutive. Continuing this sequence as stated in item \ref{d2:5}, we obtain $F_3=\TXU$, so the conditions of item \ref{d2:5} are also met and we conclude that the system \eqref{eq:academic} is flat with a difference of $d=2$. In conclusion, the system meets the items \ref{d2:1}, \ref{d2:2b}, \ref{d2:3b}, \ref{d2:4b}, \ref{d2:5} (which corresponds to the 6th path from the left in Figure \ref{fig:d2}).\\
			
			According to Theorem \ref{thm:flatOutputs}, flat outputs with $d=2$ of this system are all pairs of functions $(\varphi^1,\varphi^2)$ satisfying $\Span{\D\varphi^1,\D\varphi^2}=F_2^\perp$. We have $F_2^\perp=\Span{\D x^1+\D x^2,\D x^3+\D x^4}$ and thus, \eg $\varphi^1=x^1+x^2$ and $\varphi^2=x^3+x^4$.
		\subsection{Coin on a Moving Table}
			Consider the following model of a coin rolling on a rotating table
			\begin{align}\label{eq:coin}
				\begin{aligned}
					\dot{x}&=\Omega\cos(\theta)(x\sin(\theta)-y\cos(\theta))+R\cos(\theta)u^2\\
					\dot{y}&=\Omega\sin(\theta)(x\sin(\theta)-y\cos(\theta))+R\sin(\theta)u^2\\
					\dot{\theta}&=u^1\\
					\dot{\phi}&=u^2\,,
				\end{aligned}
			\end{align}
			taken from \cite{Kai:2006} and also considered in \eg \cite{LiXuSuChu:2013} and \cite{LiNicolauRespondek:2016}. The distribution $D_0=\Span{\partial_{u^1},\partial_{u^2}}$ is involutive,
			\begin{align*}
				\begin{aligned}
					D_1&=\Span{\partial_{u^1},\partial_{u^2},\partial_\theta,R\cos(\theta)\partial_x+R\sin(\theta)\partial_y+\partial_\phi}
				\end{aligned}
			\end{align*}
			is non-involutive, so we have $k_1=1$ and item \ref{d2:1} is met. We have $\C{D_1}=D_0$, so we are in item \ref{d2:2a} We furthermore have $\Dim{\overline{D}_1}=\Dim{D_1}+2$ (the involutive closure of $D_1$ is given by $\overline{D}_1=\TXU$), so we are in the subcase \ref{d2:2a}\ref{d2:2aB} The first derived flag of $D_1$ is given by
			\begin{align*}
				\begin{aligned}
					D_1^{(1)}&=\Span{\partial_{u^1},\partial_{u^2},\partial_\theta,R\cos(\theta)\partial_x+R\sin(\theta)\partial_y+\partial_\phi,\sin(\theta)\partial_x-\cos(\theta)\partial_y}\,,
				\end{aligned}
			\end{align*}
			its Cauchy characteristic distribution follows as $\C{D_1^{(1)}}=\Span{\partial_{u^1},\partial_{u^2},R\cos(\theta)\partial_x+R\sin(\theta)\partial_y+\partial_\phi}$ and the condition $[f,\C{D_1^{(1)}}]\subset D_1^{(1)}$ is indeed met. By definition we have $E_2=D_1^{(1)}$, and have to continue with item \ref{d2:4a}\ref{d2:4a2} The item \ref{d2:4a}\ref{d2:4a2} is indeed met with $\overline{E}_2=\TXU$ and thus, also \ref{d2:5} is met (with $F_3=\overline{E}_2=\TXU$), proving that the system is flat with a difference of $d=2$. In conclusion, the system meets the items \ref{d2:1}, \ref{d2:2a}\ref{d2:2aB}, \ref{d2:4a}\ref{d2:4a2}, \ref{d2:5} (which corresponds to the 3rd path from the left in Figure \ref{fig:d2}).\\
			
			According to Theorem \ref{thm:flatOutputs} (see also below Theorem \ref{thm:flatOutputs}), flat outputs with $d=2$ are all pairs of functions $(\varphi^1,\varphi^2)$ satisfying $\Span{\D\varphi^1,\D\varphi^2}=F_2^\perp$, where $F_2$ is an arbitrary involutive distribution satisfying $E_1\underset{1}{\subset}F_2\underset{1}{\subset}E_2$, where $E_1=\C{D_{k_1}^{(1)}}$. To construct such a distribution $F_2$, we first choose an arbitrary function $\psi$ whose differential $\D\psi\neq 0$ annihilates $E_1$, \eg $\psi=\theta$, and a pair of vector fields $v_1,v_2$ which completes $E_1$ to $E_2$, \eg $v_1=\partial_\theta$, $v_2=\sin(\theta)\partial_x-\cos(\theta)\partial_y$. A possible distribution $F_2$ is then given by $F_2=E_1+\Span{(\D\psi\rfloor v_2)v_1-(\D\psi\rfloor v_1)v_2}=E_1+\Span{\sin(\theta)\partial_x-\cos(\theta)\partial_y}$, which yields $\varphi^1=\theta$, $\varphi^2=R\phi-x\cos(\theta)-y\sin(\theta)$ as a possible flat output.	
	\section{Proof of Theorem \ref{thm:d2}}\label{se:proofd2}
		\subsection*{Necessity}
			Consider a two-input system of the form \eqref{eq:sys} and assume that it is flat with a difference of $d=2$. Let $k_1$ be defined as in Theorem \ref{thm:d2}, \ie the smallest integer such that $D_{k_1}$ is non-involutive (if $D_{k_1}$ would not exist, the system would be static feedback linearizable, which contradicts with the assumption that $d=2$). According to Theorem \ref{thm:linearizationByProlongations}, there exists an input transformation $\bar{u}=\Phi_u(x,u)$ with inverse $u=\hat{\Phi}(x,\bar{u})$ such that the system obtained by two-fold prolonging the new input $\bar{u}^1$, \ie the system
			\begin{align}\label{eq:d2twoFoldProlonged}
				\begin{aligned}
					\dot{x}&=f(x,\hat{\Phi}_u(x,\bar{u}))=\bar{f}(x,\bar{u})&&&\dot{\bar{u}}^1&=\bar{u}^1_1&&&\dot{\bar{u}}^1_1&=\bar{u}^1_2
				\end{aligned}
			\end{align}
			with the state $(x,\bar{u}^1,\bar{u}^1_1)$ and the input $(\bar{u}^1_2,\bar{u}^2)$, is static feedback linearizable. Therefore, according to Theorem \ref{thm:sfl}, with the vector field $f_p=\bar{f}^{i}(x,\bar{u})\partial_{x^i}+\bar{u}^1_1\partial_{\bar{u}^1}+\bar{u}^1_2\partial_{\bar{u}^1_1}$, the distributions
			\begin{align}\label{eq:d2twoFoldProlongedDistributions}
				\begin{aligned}
					\Delta_0&=\Span{\partial_{\bar{u}^1_2},\partial_{\bar{u}^2}}\\
					\Delta_1&=\Span{\partial_{\bar{u}^1_2},\partial_{\bar{u}^1_1},\partial_{\bar{u}^2},[f_p,\partial_{\bar{u}^2}]}\\
					\Delta_2&=\Span{\partial_{\bar{u}^1_2},\partial_{\bar{u}^1_1},\partial_{\bar{u}^1},\partial_{\bar{u}^2},[f_p,\partial_{\bar{u}^2}],\Ad_{f_p}^2\partial_{\bar{u}^2}}\\
					\Delta_3&=\Span{\partial_{\bar{u}^1_2},\partial_{\bar{u}^1_1},\partial_{\bar{u}^1},\partial_{\bar{u}^2},[f_p,\partial_{\bar{u}^1}],[f_p,\partial_{\bar{u}^2}],\Ad_{f_p}^2\partial_{\bar{u}^2},\Ad_{f_p}^3\partial_{\bar{u}^2}}\\
					%\Delta_4&=\Span{\partial_{\bar{u}^1_2},\partial_{\bar{u}^1_1},\partial_{\bar{u}^1},\partial_{\bar{u}^2},[f_p,\partial_{\bar{u}^1}],[f_p,\partial_{\bar{u}^2}],\Ad_{f_p}^2\partial_{\bar{u}^1},\Ad_{f_p}^2\partial_{\bar{u}^2},\Ad_{f_p}^3\partial_{\bar{u}^2},\Ad_{f_p}^4\partial_{\bar{u}^2}}\\
					&\vdotswithin{=}\\
					\Delta_s&=\Span{\partial_{\bar{u}^1_2},\partial_{\bar{u}^1_1},\partial_{\bar{u}^1},\partial_{\bar{u}^2},\partial_x}
				\end{aligned}
			\end{align}
			are all involutive (that the integer $s$ in the last line of \eqref{eq:d2twoFoldProlongedDistributions} coincides with $s$ in item \ref{d2:5} is shown later). Some simplifications can be made, \ie there exist simpler bases for these distributions, as the following lemma asserts.
			\begin{lemma}\label{lem:d2SimplificationOfDistributions}
				The distributions \eqref{eq:d2twoFoldProlongedDistributions} can be simplified to either
				\begin{align}\label{eq:d2twoFoldProlongedDistributionsSimplifiedCase1}
					\begin{aligned}
						\Delta_0&=\Span{\partial_{\bar{u}^1_2},\partial_{\bar{u}^2}}\\
						\Delta_1&=\Span{\partial_{\bar{u}^1_2},\partial_{\bar{u}^1_1},\partial_{\bar{u}^2},[\bar{f},\partial_{\bar{u}^2}]}\\
						\Delta_2&=\Span{\partial_{\bar{u}^1_2},\partial_{\bar{u}^1_1},\underbrace{\partial_{\bar{u}^1},\partial_{\bar{u}^2}}_{D_0},[\bar{f},\partial_{\bar{u}^2}],\Ad_{\bar{f}}^2\partial_{\bar{u}^2}}\\
						\Delta_3&=\Span{\partial_{\bar{u}^1_2},\partial_{\bar{u}^1_1},\underbrace{\partial_{\bar{u}^1},\partial_{\bar{u}^2},[\bar{f},\partial_{\bar{u}^1}],[\bar{f},\partial_{\bar{u}^2}]}_{D_1},\Ad_{\bar{f}}^2\partial_{\bar{u}^2},\Ad_{\bar{f}}^3\partial_{\bar{u}^2}}\\
						%\Delta_4&=\Span{\partial_{\bar{u}^1_2},\partial_{\bar{u}^1_1},\partial_{\bar{u}^1},\partial_{\bar{u}^2},[\bar{f},\partial_{\bar{u}^1}],[\bar{f},\partial_{\bar{u}^2}],\Ad_{\bar{f}}^2\partial_{\bar{u}^1},\Ad_{\bar{f}}^2\partial_{\bar{u}^2},\Ad_{\bar{f}}^3\partial_{\bar{u}^2},\Ad_{\bar{f}}^4\partial_{\bar{u}^2}}\\
						&\vdotswithin{=}\\
						\Delta_{k_1}&=\Span{\partial_{\bar{u}^1_2},\partial_{\bar{u}^1_1},\underbrace{\partial_{\bar{u}^1},\partial_{\bar{u}^2},[\bar{f},\partial_{\bar{u}^1}],[\bar{f},\partial_{\bar{u}^2}],\ldots,\Ad_{\bar{f}}^{k_1-2}\partial_{\bar{u}^1},\Ad_{\bar{f}}^{k_1-2}\partial_{\bar{u}^2}}_{D_{k_1-2}},\Ad_{\bar{f}}^{k_1-1}\partial_{\bar{u}^2},\Ad_{\bar{f}}^{k_1}\partial_{\bar{u}^2}}\\
						\Delta_{k_1+1}&=\Span{\partial_{\bar{u}^1_2},\partial_{\bar{u}^1_1},\underbrace{\partial_{\bar{u}^1},\partial_{\bar{u}^2},[\bar{f},\partial_{\bar{u}^1}],[\bar{f},\partial_{\bar{u}^2}],\ldots,\Ad_{\bar{f}}^{k_1-1}\partial_{\bar{u}^1},\Ad_{\bar{f}}^{k_1-1}\partial_{\bar{u}^2}}_{D_{k_1-1}},\Ad_{\bar{f}}^{k_1}\partial_{\bar{u}^2},\Ad_{\bar{f}}^{k_1+1}\partial_{\bar{u}^2}}\\
						\Delta_{k_1+2}&=\Span{\partial_{\bar{u}^1_2},\partial_{\bar{u}^1_1},\underbrace{\partial_{\bar{u}^1},\partial_{\bar{u}^2},[\bar{f},\partial_{\bar{u}^1}],[\bar{f},\partial_{\bar{u}^2}],\ldots,\Ad_{\bar{f}}^{k_1}\partial_{\bar{u}^1},\Ad_{\bar{f}}^{k_1}\partial_{\bar{u}^2}}_{D_{k_1}},\Ad_{\bar{f}}^{k_1+1}\partial_{\bar{u}^2},\Ad_{\bar{f}}^{k_1+2}\partial_{\bar{u}^2}}\\
						&\vdotswithin{=}\\
						\Delta_s&=\Span{\partial_{\bar{u}^1_2},\partial_{\bar{u}^1_1},\partial_{\bar{u}^1},\partial_{\bar{u}^2},\partial_x}\,.
					\end{aligned}
				\end{align}
				or
				\begin{align}\label{eq:d2twoFoldProlongedDistributionsSimplifiedCase2}
					\begin{aligned}
						\Delta_0&=\Span{\partial_{\bar{u}^1_2},\partial_{\bar{u}^2}}\\
						\Delta_1&=\Span{\partial_{\bar{u}^1_2},\partial_{\bar{u}^1_1},\partial_{\bar{u}^2},[\bar{f},\partial_{\bar{u}^2}]}\\
						\Delta_2&=\Span{\partial_{\bar{u}^1_2},\partial_{\bar{u}^1_1},\underbrace{\partial_{\bar{u}^1},\partial_{\bar{u}^2}}_{D_0},[\bar{f},\partial_{\bar{u}^2}],[\partial_{\bar{u}^1},[\bar{f},\partial_{\bar{u}^2}]]}\\
						\Delta_3&=\Span{\partial_{\bar{u}^1_2},\partial_{\bar{u}^1_1},\underbrace{\partial_{\bar{u}^1},\partial_{\bar{u}^2},[\bar{f},\partial_{\bar{u}^1}],[\bar{f},\partial_{\bar{u}^2}]}_{D_1},[\partial_{\bar{u}^1},[\bar{f},\partial_{\bar{u}^2}]],[\bar{f},[\partial_{\bar{u}^1},[\bar{f},\partial_{\bar{u}^2}]]]}\\
						%\Delta_4&=\Span{\partial_{\bar{u}^1_2},\partial_{\bar{u}^1_1},\partial_{\bar{u}^1},\partial_{\bar{u}^2},[\bar{f},\partial_{\bar{u}^1}],[\bar{f},\partial_{\bar{u}^2}],\Ad_{\bar{f}}^2\partial_{\bar{u}^1},[\partial_{\bar{u}^1},[\bar{f},\partial_{\bar{u}^2}]],[\bar{f},[\partial_{\bar{u}^1},[\bar{f},\partial_{\bar{u}^2}]]],\Ad_{\bar{f}}^2[\partial_{\bar{u}^1},[\bar{f},\partial_{\bar{u}^2}]]]}\\
						&\vdotswithin{=}\\
						\Delta_s&=\Span{\partial_{\bar{u}^1_2},\partial_{\bar{u}^1_1},\partial_{\bar{u}^1},\partial_{\bar{u}^2},\partial_x}\,.
					\end{aligned}
				\end{align}
			\end{lemma}
			It follows from the proof of Lemma \ref{lem:d2SimplificationOfDistributions}, which is provided in Appendix \ref{ap:lemmas}, that in case of $k_1\geq 2$ we always have the form \eqref{eq:d2twoFoldProlongedDistributionsSimplifiedCase1}. The form \eqref{eq:d2twoFoldProlongedDistributionsSimplifiedCase2} is only relevant if in the case $k_1=1$ we have $\Ad_{\bar{f}}^2\partial_{\bar{u}^2}\in\Span{\partial_{\bar{u}^1_2},\partial_{\bar{u}^1_1},\partial_{\bar{u}^1},\partial_{\bar{u}^2},[\bar{f},\partial_{\bar{u}^2}]}$. In all other cases, the distributions $\Delta_i$ are indeed of the form \eqref{eq:d2twoFoldProlongedDistributionsSimplifiedCase1}.\\
			
			\paragraph*{Necessity of Item \ref{d2:1}} From the assumption $\Rank{\partial_{u}f}=2$ (\ie the assumption that the system has no redundant inputs), it immediately follows that $\Dim{D_1}=4$, which in case of $k_1=1$ already shows the necessity of item \ref{d2:1} For $k_1\geq 2$, the involutive distributions $\Delta_i$ can always be written in the form \eqref{eq:d2twoFoldProlongedDistributionsSimplifiedCase1} (see below Lemma \ref{lem:d2SimplificationOfDistributions}), based on which $\Dim{D_i}=2(i+1)$, $i=2,\ldots,k_1$ can  be shown by contradiction. Assume that $\Dim{D_i}<2(i+1)$ for some $i$ where $2\leq i\leq k_1$ and let $2\leq l\leq k_1$ be the smallest integer such that $\Dim{D_l}<2(l+1)$. Thus, $\Ad_{\bar{f}}^l\partial_{\bar{u}^1}$ and $\Ad_{\bar{f}}^l\partial_{\bar{u}^2}$ are collinear modulo $D_{l-1}$. Both of these vector fields being contained in $D_{l-1}$, \ie $\Ad_{\bar{f}}^l\partial_{\bar{u}^1}\in D_{l-1}$ and $\Ad_{\bar{f}}^l\partial_{\bar{u}^2}\in D_{l-1}$, contradicts with $D_{k_1}$ being non-involutive, it would lead to $D_{k_1}=D_{l-1}$, \ie the sequence would stop growing from $D_{l-1}$ on. 
			
			If $\Ad_{\bar{f}}^l\partial_{\bar{u}^2}\in D_{l-1}$, we have (see \eqref{eq:d2twoFoldProlongedDistributionsSimplifiedCase1})
			\begin{align}\label{eq:delta_lp1}
				\Delta_{l+1}&=\Span{\partial_{\bar{u}^1_2},\partial_{\bar{u}^1_1}}+D_{l-1}+\Span{\underbrace{\Ad_{\bar{f}}^{l+1}\partial_{\bar{u}^2}}_{\in D_l}}\,,
			\end{align}
			where $\Ad_{\bar{f}}^{l+1}\partial_{\bar{u}^2}\in D_l$ follows from the assumption $\Ad_{\bar{f}}^l\partial_{\bar{u}^2}\in D_{l-1}$ since by construction $D_l=D_{l-1}+[\bar{f},D_{l-1}]$. Since $D_{l-1}\underset{1}{\subset}D_l$, we either have $\Delta_{l+1}=\Span{\partial_{\bar{u}^1_2},\partial_{\bar{u}^1_1}}+D_l$ or $\Delta_{l+1}=\Span{\partial_{\bar{u}^1_2},\partial_{\bar{u}^1_1}}+D_{l-1}$, \ie the vector field $\Ad_{\bar{f}}^{l+1}\partial_{\bar{u}^2}$ in \eqref{eq:delta_lp1} either completes $D_{l-1}$ to $D_l$, or it is contained in $D_{l-1}$. The case $\Delta_{l+1}=\Span{\partial_{\bar{u}^1_2},\partial_{\bar{u}^1_1}}+D_l$ would lead to $\Delta_{k_1+1}=\Span{\partial_{\bar{u}^1_2},\partial_{\bar{u}^1_1}}+D_{k_1}$, which would imply that $D_{k_1}$ is involutive, contradicting with $D_{k_1}$ being non-involutive. Similarly, the case $\Delta_{l+1}=\Span{\partial_{\bar{u}^1_2},\partial_{\bar{u}^1_1}}+D_{l-1}$ would lead to $\Delta_{k_1+2}=\Span{\partial_{\bar{u}^1_2},\partial_{\bar{u}^1_1}}+D_{k_1}$, which would again imply that $D_{k_1}$ is involutive, again contradicting with $D_{k_1}$ being non-involutive.
			
			If $\Ad_{\bar{f}}^l\partial_{\bar{u}^2}\notin D_{l-1}$, we necessarily have $D_l=D_{l-1}+\Span{\Ad_{\bar{f}}^l\partial_{\bar{u}^2}}$ (since by assumption $\Dim{D_l}<2(l+1)$ and $\Dim{D_{l-1}}=2l$), which leads to $D_{k_1}=D_{k_1-1}+\Span{\Ad_{\bar{f}}^{k_1}\partial_{\bar{u}^2}}$. From the involutivity of $\Delta_{k_1}\cap\TXU=D_{k_1-2}+\Span{\Ad_{\bar{f}}^{k_1-1}\partial_{\bar{u}^2},\Ad_{\bar{f}}^{k_1}\partial_{\bar{u}^2}}\subset D_{k_1}$, it follows that $[\Ad_{\bar{f}}^{k_1-1}\partial_{\bar{u}^2},\Ad_{\bar{f}}^{k_1}\partial_{\bar{u}^2}]\in D_{k_1}$, which together with the involutivity of $D_{k_1-1}$ and the fact that $\Ad_{\bar{f}}^{k_1-1}\partial_{\bar{u}^2}\in D_{k_1-1}$ implies that $\Ad_{\bar{f}}^{k_1-1}\partial_{\bar{u}^2}\in\C{D_{k_1}}$ (recall that $D_{k_1}=D_{k_1-1}+\Span{\Ad_{\bar{f}}^{k_1}\partial_{\bar{u}^2}}$). From the Jacobi identity
			\begin{align*}
				\begin{aligned}
					[\Ad_{\bar{f}}^{k_1-1}\partial_{\bar{u}^1},[\bar{f},\Ad_{\bar{f}}^{k_1-1}\partial_{\bar{u}^2}]]+\underbrace{[\overbrace{\Ad_{\bar{f}}^{k_1-1}\partial_{\bar{u}^2}}^{\in\C{D_{k_1}}},\overbrace{[\Ad_{\bar{f}}^{k_1-1}\partial_{\bar{u}^1},\bar{f}]}^{\in D_{k_1}}]}_{\in D_{k_1}}+\underbrace{[\bar{f},\overbrace{[\Ad_{\bar{f}}^{k_1-1}\partial_{\bar{u}^2},\Ad_{\bar{f}}^{k_1-1}\partial_{\bar{u}^1}]}^{\in D_{k_1-1}}]}_{\in D_{k_1}}&=0\,,
				\end{aligned}
			\end{align*}
			it follows that $[\Ad_{\bar{f}}^{k_1-1}\partial_{\bar{u}^1},\Ad_{\bar{f}}^{k_1}\partial_{\bar{u}^2}]\in D_{k_1}$, which because of the involutivity of $D_{k_1-1}$ implies that $\Ad_{\bar{f}}^{k_1-1}\partial_{\bar{u}^1}\in\C{D_{k_1}}$. Since also $D_{k_1-2}\subset\C{D_{k_1}}$ (which can be shown based on the Jacobi identity), it follows that $D_{k_1-1}\subset\C{D_{k_1}}$ which would imply that $D_{k_1}=D_{k_1-1}+\Span{\Ad_{\bar{f}}^{k_1}\partial_{\bar{u}^2}}$ is involutive, contradicting with $D_{k_1}$ being non-involutive and completing the proof of the necessity of item \ref{d2:1}\\
			
			In the following we show the necessity of the remaining items. We distinguish between two cases, namely the case that the involutive distributions $\Delta_i$ in \eqref{eq:d2twoFoldProlongedDistributions} can be written in the form \eqref{eq:d2twoFoldProlongedDistributionsSimplifiedCase1}, and the case that these distributions can only be written in the form \eqref{eq:d2twoFoldProlongedDistributionsSimplifiedCase2}. As already mentioned, the form \eqref{eq:d2twoFoldProlongedDistributionsSimplifiedCase2} is only relevant if in the case $k_1=1$ we have $\Ad_{\bar{f}}^2\partial_{\bar{u}^2}\in\Span{\partial_{\bar{u}^1_2},\partial_{\bar{u}^1_1},\partial_{\bar{u}^1},\partial_{\bar{u}^2},[\bar{f},\partial_{\bar{u}^2}]}$, in all other cases, the distributions $\Delta_i$ are indeed of the form \eqref{eq:d2twoFoldProlongedDistributionsSimplifiedCase1}. We treat this special case in Appendix \ref{ap:supplements}, so that for the rest of the proof, we can assume that the involutive distributions $\Delta_i$ in \eqref{eq:d2twoFoldProlongedDistributions} can be written in the form \eqref{eq:d2twoFoldProlongedDistributionsSimplifiedCase1}.
			\paragraph*{Necessity of Item 2.} We either have $D_{k_1-1}\subset\C{D_{k_1}}$ or $D_{k_1-1}\not\subset\C{D_{k_1}}$. We have to show the necessity of item \ref{d2:2a} under the assumption that $D_{k_1-1}\subset\C{D_{k_1}}$, and the necessity of item \ref{d2:2b} under the assumption that $D_{k_1-1}\not\subset\C{D_{k_1}}$.
			\subsubsection*{\textit{The case $D_{k_1-1}\subset\C{D_{k_1}}$}} We have to show the necessity of item \ref{d2:2a} We have $D_{k_1}\subset\Delta_{k_1+2}\cap\TXU$ (see \eqref{eq:d2twoFoldProlongedDistributionsSimplifiedCase1}) and since $\Delta_{k_1+2}\cap\TXU$ is involutive, it follows that $\overline{D}_{k_1}\subset\Delta_{k_1+2}\cap\TXU$. In turn, we either have $\Dim{\overline{D}_{k_1}}=\Dim{D_{k_1}}+1$, which corresponds to item \ref{d2:2a}\ref{d2:2aA}, or $\Dim{\overline{D}_{k_1}}=\Dim{D_{k_1}}+2$, which corresponds to \ref{d2:2a}\ref{d2:2aB} We have to distinguish between these two possible subcases. In either subcase, the following result, proven in Appendix \ref{ap:lemmas}, will be useful.
			\begin{lemma}\label{lem:d2case1FirstDerivedFlag}
				In the case $d=2$ with $D_{k_1-1}\subset\C{D_{k_1}}$, we have $D^{(1)}_{k_1}=D_{k_1}+\Span{\Ad_{\bar{f}}^{k_1+1}\partial_{\bar{u}^2}}$.
			\end{lemma}
			Let us first consider the subcase $\Dim{\overline{D}_{k_1}}=\Dim{D_{k_1}}+1$, which corresponds to item \ref{d2:2a}\ref{d2:2aA} We have to show that $\Dim{[\bar{f},D_{k_1}]+\overline{D}_{k_1}}=\Dim{\overline{D}_{k_1}}+1$. We necessarily have $\overline{D}_{k_1}\neq\TXU$, since $\overline{D}_{k_1}=\TXU$ would imply flatness with a difference of $d=1$\footnote{Indeed, for $D_{k_1-1}\subset\C{D_{k_1}}$ with $\Dim{\overline{D}_{k_1}}=\Dim{D_{k_1}}+1$ and $\overline{D}_{k_1}=\TXU$, the conditions of Theorem \ref{thm:d1} would be met.}, which contradicts with the assumption that $d=2$. Since $\overline{D}_{k_1}=D_{k_1}^{(1)}=D_{k_1}+\Span{\Ad_{\bar{f}}^{k_1+1}\partial_{\bar{u}^2}}$ (see Lemma \ref{lem:d2case1FirstDerivedFlag}), the condition holds if and only if $\Ad_{\bar{f}}^{k_1+1}\partial_{\bar{u}^1}\neq\overline{D}_{k_1}$, which can be shown by contradiction. Assume that $\Ad_{\bar{f}}^{k_1+1}\partial_{\bar{u}^1}\in\overline{D}_{k_1}$ and thus $[\bar{f},D_{k_1}]+\overline{D}_{k_1}=\overline{D}_{k_1}$. Based on the Jacobi identity, it can then be shown that this would imply $[\bar{f},\overline{D}_{k_1}]\subset\overline{D}_{k_1}$. However, since $\overline{D}_{k_1}\neq\TXU$, this would in turn imply that the system contains an autonomous subsystem, which is in contradiction with the assumption that the system is flat.\footnote{The autonomous subsystems occurs explicitly in coordinates in which the distribution $\overline{D}_{k_1}$ is straightened out (see also Theorem 3.49 in \cite{NijmeijervanderSchaft:1990}).} Thus, $\Ad_{\bar{f}}^{k_1+1}\partial_{\bar{u}^1}\notin\overline{D}_{k_1}$ and $\Dim{[\bar{f},D_{k_1}]+\overline{D}_{k_1}}=\Dim{\overline{D}_{k_1}}+1$ indeed holds. By definition, we have $E_{k_1+1}=\overline{D}_{k_1}$ in this case, which evaluates to $E_{k_1+1}=D_{k_1}+\Span{\Ad_{\bar{f}}^{k_1+1}\partial_{\bar{u}^2}}$.\\
			
			Next, let us consider the subcase $\Dim{\overline{D}_{k_1}}=\Dim{D_{k_1}}+2$, which corresponds to item \ref{d2:2a}\ref{d2:2aB} According to Lemma \ref{lem:d2case1FirstDerivedFlag}, we have $D_{k_1}^{(1)}=D_{k_1}+\Span{\Ad_{\bar{f}}^{k_1+1}\partial_{\bar{u}^2}}$, which is non-involutive since by assumption $\Dim{\overline{D}_{k_1}}=\Dim{D_{k_1}}+2$. Note that $D_{k_1}^{(1)}$ can be written in the form
			\begin{align*}
				\begin{aligned}
					D_{k_1}^{(1)}&=\underbrace{D_{k_1-1}+\Span{\Ad_{\bar{f}}^{k_1}\partial_{\bar{u}^2},\Ad_{\bar{f}}^{k_1+1}\partial_{\bar{u}^2}}}_{=\Delta_{k_1+1}\cap\TXU}+\Span{\Ad_{\bar{f}}^{k_1}\partial_{\bar{u}^1}}\,.
				\end{aligned}
			\end{align*}
			From the involutivity of $\Delta_{k_1+1}\cap\TXU$, the fact that $[\Ad_{\bar{f}}^{k_1}\partial_{\bar{u}^2},\Ad_{\bar{f}}^{k_1}\partial_{\bar{u}^1}]\in D_{k_1}^{(1)}$ and the non-involutivity of $D_{k_1}^{(1)}$, it follows that $\C{D_{k_1}^{(1)}}=D_{k_1-1}+\Span{\Ad_{\bar{f}}^{k_1}\partial_{\bar{u}^2}}$, from which the necessity of the condition $[\bar{f},\C{D_{k_1}^{(1)}}]\subset D_{k_1}^{(1)}$ follows immediately. By definition, we have $E_{k_1+1}$ in this case, which is of course non-involutive and evaluates to $E_{k_1+1}=D_{k_1}+\Span{\Ad_{\bar{f}}^{k_1+1}\partial_{\bar{u}^2}}$.
%			Note that actually in any of the two subcases in item \ref{d2:2a}, we have $E_{k_1+1}=D_{k_1}^{(1)}=D_{k_1}+\Span{\Ad_{\bar{f}}^{k_1+1}\partial_{\bar{u}^2}}$, in the first subcase ($\Dim{\overline{D}_{k_1}}=\Dim{D_{k_1}}+1$), this distribution is involutive, in the second ($\Dim{\overline{D}_{k_1}}=\Dim{D_{k_1}}+2$), it is non-involutive.
			\subsubsection*{\textit{The case $D_{k_1-1}\not\subset\C{D_{k_1}}$}} We have to show the necessity of item \ref{d2:2b}, \ie we have to show that there exists a vector field $v_c\in D_{k_1-1}$, $v_c\notin D_{k_1-2}$ such that $E_{k_1-1}\subset\C{E_{k_1}}$ where $E_{k_1-1}=D_{k_1-2}+\Span{v_c}$ and $E_{k_1}=D_{k_1-1}+\Span{[v_c,\bar{f}]}$. Let us show that the vector field $v_c=\Ad_{\bar{f}}^{k_1-1}\partial_{\bar{u}^2}$ and the distributions $E_{k_1-1}$ and $E_{k_1}$ derived from it meet these criteria. The vector field $v_c=\Ad_{\bar{f}}^{k_1-1}\partial_{\bar{u}^2}$ meets $\Ad_{\bar{f}}^{k_1-1}\partial_{\bar{u}^2}\in D_{k_1-1}$, $\Ad_{\bar{f}}^{k_1-1}\partial_{\bar{u}^2}\notin D_{k_1-2}$ (the latter must hold since $\Ad_{\bar{f}}^{k_1-1}\partial_{\bar{u}^2}\in D_{k_1-2}$ would imply $\Dim{D_{k_1-1}}<2k_1$), and yields $E_{k_1-1}=D_{k_1-2}+\Span{\Ad_{\bar{f}}^{k_1-1}\partial_{\bar{u}^2}}$ and $E_{k_1}=D_{k_1-1}+\Span{\Ad_{\bar{f}}^{k_1}\partial_{\bar{u}^2}}$. Due to the involutivity of $D_{k_1-1}$ and the involutivity of $\Delta_{k_1}\cap\TXU=D_{k_1-2}+\Span{\Ad_{\bar{f}}^{k_1-1}\partial_{\bar{u}^2},\Ad_{\bar{f}}^{k_1}\partial_{\bar{u}^2}}\subset E_{k_1}$, it follows that $[\Ad_{\bar{f}}^{k_1-1}\partial_{\bar{u}^2},E_{k_1}]\subset E_{k_1}$ and $[D_{k_1-2},E_{k_1}]\subset E_{k_1}$, and hence $E_{k_1-1}\subset\C{E_{k_1}}$. Therefore, $v_c=\Ad_{\bar{f}}^{k_1-1}\partial_{\bar{u}^2}$, $E_{k_1-1}=D_{k_1-2}+\Span{\Ad_{\bar{f}}^{k_1-1}\partial_{\bar{u}^2}}$ and $E_{k_1}=D_{k_1-1}+\Span{\Ad_{\bar{f}}^{k_1}\partial_{\bar{u}^2}}$.
			
			\paragraph*{Necessity of Item 3.} The distribution $E_{k_1}=D_{k_1-1}+\Span{\Ad_{\bar{f}}^{k_1}\partial_{\bar{u}^2}}$ of item \ref{d2:2b} can either be involutive or non-involutive. We have to show the necessity of item \ref{d2:3a} under the assumption that $E_{k_1}=D_{k_1-1}+\Span{\Ad_{\bar{f}}^{k_1}\partial_{\bar{u}^2}}$ of item \ref{d2:2b} is non-involutive, and the necessity of item \ref{d2:3b} under the assumption that $E_{k_1}$ of item \ref{d2:2b} is involutive.
			
			In item \ref{d2:2a}\ref{d2:2aA}, the distribution $E_{k_1+1}$ is by construction involutive, so we have to show the necessity of item \ref{d2:3b} in this case. (In item \ref{d2:2a}\ref{d2:2aB}, the item \textbf{3.} is not relevant as it is skipped in the corresponding conditions.)
			
			\subsubsection*{$E_{k_1}$ of item \ref{d2:2b} being non-involutive}
			We have to show the necessity of item \ref{d2:3a} From $E_{k_1}\subset\Delta_{k_1+1}\cap\TXU$ (see \eqref{eq:d2twoFoldProlongedDistributionsSimplifiedCase1}), it immediately follows that $\overline{E}_{k_1}=\Delta_{k_1+1}\cap\TXU=D_{k_1-1}+\Span{\Ad_{\bar{f}}^{k_1}\partial_{\bar{u}^2},\Ad_{\bar{f}}^{k_1+1}\partial_{\bar{u}^2}}$ and in turn the necessity of the condition $\Dim{\overline{E}_{k_1}}=\Dim{E_{k_1}}+1$, which corresponds to \ref{d2:3a}\ref{d2:3a1} follows.\footnote{Note that we necessarily have $\overline{E}_{k_1}\neq\TXU$. Indeed, we have $\Dim{E_{k_1}}=\Dim{D_{k_1}}-1$, and we have just shown that necessarily $\Dim{\overline{E}_{k_1}}=\Dim{\overline{E}_{k_1}}+1$. Thus, $\Dim{\overline{E}_{k_1}}=\Dim{D_{k_1}}$ and $\overline{E}_{k_1}=\TXU$ would thus imply $D_{k_1}=\TXU$, which is in contradiction with $D_{k_1}$ being non-involutive.}
			
			Next, let us show the necessity of item \ref{d2:3a}\ref{d2:3a2}, \ie $\Dim{[\bar{f},E_{k_1}]+\overline{E}_{k_1}}=\Dim{\overline{E}_{k_1}}+1$. Recall that we have $E_{k_1}=D_{k_1-1}+\Span{\Ad_{\bar{f}}^{k_1}\partial_{\bar{u}^2}}$ and $\overline{E}_{k_1}=D_{k_1-1}+\Span{\Ad_{\bar{f}}^{k_1}\partial_{\bar{u}^2},\Ad_{\bar{f}}^{k_1+1}\partial_{\bar{u}^2}}$. Since $\Ad_{\bar{f}}^{k_1+1}\partial_{\bar{u}^2}\in\overline{E}_{k_1}$, the condition holds if and only if $\Ad_{\bar{f}}^{k_1}\partial_{\bar{u}^1}\notin\overline{E}_{k_1}$, which can be shown by contradiction. Assume that $\Ad_{\bar{f}}^{k_1}\partial_{\bar{u}^1}\in\overline{E}_{k_1}$ and thus $[\bar{f},E_{k_1}]+\overline{E}_{k_1}=\overline{E}_{k_1}$. Based on the Jacobi identity, it can then be shown that this would imply $[\bar{f},\overline{E}_{k_1}]\subset\overline{E}_{k_1}$. However, since $\overline{E}_{k_1}\neq\TXU$, this would in turn imply that the system contains an autonomous subsystem, which is in contradiction with the assumption that the system is flat. By definition, we have $F_{k_1+1}=\overline{E}_{k_1}=E_{k_1}+\Span{\Ad_{\bar{f}}^{k_1+1}\partial_{\bar{u}^2}}$ in this case.
			
			\subsubsection*{$E_{k_1}$ of item \ref{d2:2b}, or $E_{k_1+1}$ of item \ref{d2:2a}\ref{d2:2aA} being involutive}
			In this case, the distributions $E_i=E_{i-1}+[\bar{f},E_{i-1}]$ are defined and we have to show the necessity of item \ref{d2:3b}\\
			
			Let us first show the necessity of \ref{d2:3b}\ref{d2:3b1}, \ie that there necessarily exists an index $k_2$ such that $E_{k_2}$ is non-involutive. We show the existence of $k_2$ by contradiction. Assume that all the distributions $E_i$ are involutive. If there does not exist an integer $s$ such that $E_s=\TXU$, there exists an integer $l$ such that $[\bar{f},E_l]\subset E_l$, which would imply that the system contains an autonomous subsystem and be in contradiction with the assumption that the system is flat. If all the distributions $E_i$ are involutive and there exists an integer $s$ such that $E_s=\TXU$, it can be shown that the system would meet the conditions for flatness with $d=1$, which contradicts with the assumption that $d=2$. %E_k1 und E_k1+1 involutiv. Es gilt D_k1\subset E_k1+1 mit corank 1 und daher D_k1\subset\overline{D}_k1 mit corank 1, woraus die Cauchy von D_k1 konstruiert werden kann (sinngemäß so wie im beweis von Lemma 7.3 (4b)). E_k1 von thm:d2 stimmt dann mit E_k1 von thm:d1 überein.
			\\
			
			To show the condition on the dimensions of the distributions $E_i$, note that actually in any case, \ie independent of whether \ref{d2:2a}\ref{d2:2aA} or \ref{d2:2b} applies, we have $E_{k_1+1}=D_{k_1}+\Span{\Ad_{\bar{f}}^{k_1+1}\partial_{\bar{u}^2}}$. Indeed, in \ref{d2:2a}\ref{d2:2aA}, by definition we have $E_{k_1+1}=\overline{D}_{k_1}$, which turned out to be $E_{k_1+1}=D_{k_1}+\Span{\Ad_{\bar{f}}^{k_1+1}\partial_{\bar{u}^2}}$. In \ref{d2:2b}, we found that $E_{k_1}=D_{k_1-1}+\Span{\Ad_{\bar{f}}^{k_1}\partial_{\bar{u}^2}}$ (assumed to be involutive here) and thus again $E_{k_1+1}=D_{k_1}+\Span{\Ad_{\bar{f}}^{k_1+1}\partial_{\bar{u}^2}}$. For the distributions $E_i$, $k_1+1\leq i\leq k_2$ we thus have
			\begin{align}\label{eq:E}
				\begin{aligned}
					E_{k_1+1}&=\Span{\partial_{\bar{u}^1},\partial_{\bar{u}^2},[\bar{f},\partial_{\bar{u}^1}],[\bar{f},\partial_{\bar{u}^2}],\ldots,\Ad_{\bar{f}}^{k_1}\partial_{\bar{u}^1},\Ad_{\bar{f}}^{k_1}\partial_{\bar{u}^2},\Ad_{\bar{f}}^{k_1+1}\partial_{\bar{u}^2}}\\
					&\vdotswithin{=}\\
					E_{k_2-1}&=\Span{\partial_{\bar{u}^1},\partial_{\bar{u}^2},[\bar{f},\partial_{\bar{u}^1}],[\bar{f},\partial_{\bar{u}^2}],\ldots,\Ad_{\bar{f}}^{k_2-2}\partial_{\bar{u}^1},\Ad_{\bar{f}}^{k_2-2}\partial_{\bar{u}^2},\Ad_{\bar{f}}^{k_2-1}\partial_{\bar{u}^2}}\\
					E_{k_2}&=\Span{\partial_{\bar{u}^1},\partial_{\bar{u}^2},[\bar{f},\partial_{\bar{u}^1}],[\bar{f},\partial_{\bar{u}^2}],\ldots,\Ad_{\bar{f}}^{k_2-1}\partial_{\bar{u}^1},\Ad_{\bar{f}}^{k_2-1}\partial_{\bar{u}^2},\Ad_{\bar{f}}^{k_2}\partial_{\bar{u}^2}}
				\end{aligned}
			\end{align}
			and with the distributions \eqref{eq:d2twoFoldProlongedDistributionsSimplifiedCase1}, they are related via
			\begin{align}\label{eq:DeltaErelation}
				\begin{aligned}
					\Delta_{k_1+2}&=\Span{\partial_{\bar{u}^1_2},\partial_{\bar{u}^1_1}}+E_{k_1+1}+\Span{\Ad_{\bar{f}}^{k_1+2}\partial_{\bar{u}^2}}\\
					&\vdotswithin{=}\\
					\Delta_{k_2}&=\Span{\partial_{\bar{u}^1_2},\partial_{\bar{u}^1_1}}+E_{k_2-1}+\Span{\Ad_{\bar{f}}^{k_2}\partial_{\bar{u}^2}}\\
					\Delta_{k_2+1}&=\Span{\partial_{\bar{u}^1_2},\partial_{\bar{u}^1_1}}+E_{k_2}+\Span{\Ad_{\bar{f}}^{k_2+1}\partial_{\bar{u}^2}}\,.
				\end{aligned}
			\end{align}
			The necessity of the condition $\Dim{E_i}=2i+1$, $i=k_1+1,\ldots,k_2$ can now be shown by contradiction. Assume that $\Dim{E_i}\leq 2i$ for some $i$ where $k_1+1\leq i\leq k_2$ and let $k_1+1\leq l\leq k_2$ be the smallest integer such that $\Dim{E_l}\leq 2l$. We have
			\begin{align*}
				\begin{aligned}
					E_{l-1}&=\Span{\partial_{\bar{u}^1},\partial_{\bar{u}^2},[\bar{f},\partial_{\bar{u}^1}],[\bar{f},\partial_{\bar{u}^2}],\ldots,\Ad_{\bar{f}}^{l-2}\partial_{\bar{u}^1},\Ad_{\bar{f}}^{l-2}\partial_{\bar{u}^2},\Ad_{\bar{f}}^{l-1}\partial_{\bar{u}^2}}\\
					E_l&=\Span{\partial_{\bar{u}^1},\partial_{\bar{u}^2},[\bar{f},\partial_{\bar{u}^1}],[\bar{f},\partial_{\bar{u}^2}],\ldots,\Ad_{\bar{f}}^{l-1}\partial_{\bar{u}^1},\Ad_{\bar{f}}^{l-1}\partial_{\bar{u}^2},\Ad_{\bar{f}}^l\partial_{\bar{u}^2}}
				\end{aligned}
			\end{align*}
			and by assumption $\Dim{E_{l-1}}=2l-1$ and $\Dim{E_l}\leq 2l$. Thus, $\Ad_{\bar{f}}^{l-1}\partial_{\bar{u}^1}$ and $\Ad_{\bar{f}}^l\partial_{\bar{u}^2}$ are collinear modulo $E_{l-1}$. Both of these vector fields being contained in $E_{l-1}$, contradicts with $E_{k_2}$ being non-involutive, it would lead to $E_{k_2}=E_{l-1}$, \ie the sequence would stop growing from $E_{l-1}$ on. If $\Ad_{\bar{f}}^l\partial_{\bar{u}^2}\in E_{l-1}$, due to
			\begin{align*}
				\begin{aligned}
					\Delta_l&=\Span{\partial_{\bar{u}^1_2},\partial_{\bar{u}^1_1}}+E_{l-1}+\Span{\Ad_{\bar{f}}^l\partial_{\bar{u}^2}}
				\end{aligned}
			\end{align*}
			(see \eqref{eq:DeltaErelation}), we have $\Delta_l=\Span{\partial_{\bar{u}^1_2},\partial_{\bar{u}^1_1}}+E_{l-1}$ and it follows that also $\Delta_i=\Span{\partial_{\bar{u}^1_2},\partial_{\bar{u}^1_1}}+E_{i-1}$ for $i>l$. The involutivity of $\Delta_{k_2+1}$ would then imply that $E_{k_2}$ is involutive, which again contradicts with $E_{k_2}$ being non-involutive. 
			
			If $\Ad_{\bar{f}}\partial_{\bar{u}^2}\notin E_{l-1}$, we necessarily have $\Ad_{\bar{f}}^{l-1}\partial_{\bar{u}^1}\in E_{l-1}+\Span{\Ad_{\bar{f}}^l\partial_{\bar{u}^2}}$ and thus $E_l=E_{l-1}+\Span{\Ad_{\bar{f}}^l\partial_{\bar{u}^2}}$ and in turn $\Delta_l=\Span{\partial_{\bar{u}^1_2},\partial_{\bar{u}^1_1}}+E_l$ and it follows that also $\Delta_i=\Span{\partial_{\bar{u}^1_2},\partial_{\bar{u}^1_1}}+E_{i}$ for $i>l$. The involutivity of $\Delta_{k_2}$ would then imply that $E_{k_2}$ is involutive, which again contradicts with $E_{k_2}$ being non-involutive. Thus, $\Ad_{\bar{f}}^{l-1}\partial_{\bar{u}^1}$ and $\Ad_{\bar{f}}^l\partial_{\bar{u}^2}$ cannot be collinear modulo $E_{l-1}$ for any $k_1+1\leq l\leq k_2$ which shows that $\Dim{E_i}=2i+1$ for $i=k_1+1,\ldots,k_2$.\footnote{In case of \ref{d2:2a}\ref{d2:2aA}, there does not explicitly occur a distribution $E_{k_1}$, so for $l=k_1+1$, we cannot argue as above. However, in this case we have $E_{k_1+1}=\overline{D}_{k_1}$ and therefore, $\Dim{E_{k_1+1}}=2(k_1+1)+1$ follows immediately from $\Dim{D_{k_1}}=2(k_1+1)$ and $\Dim{\overline{D}_{k_1}}=\Dim{D_{k_1}}+1$.}
			
			\paragraph*{Necessity of Item 4.} For the distributions $E_{k_2-1}$ and $E_{k_2}$ of item \ref{d2:3b}, we either have $E_{k_2-1}\subset\C{E_{k_2}}$ or $E_{k_2-1}\not\subset\C{E_{k_2}}$. We have to show the necessity of item \ref{d2:4a} under the assumption that $E_{k_2-1}\subset\C{E_{k_2}}$, and the necessity of item \ref{d2:4b} under the assumption that $E_{k_2-1}\not\subset\C{E_{k_2}}$. Furthermore, under the assumption that the conditions of \ref{d2:2a}\ref{d2:2aB} are met, we have to show the necessity of item \ref{d2:4a}\ref{d2:4a2}
			\subsubsection*{\textit{The case $E_{k_2-1}\subset\C{E_{k_2}}$.}} We have to show the necessity of item \ref{d2:4a} %(note that \ref{d2:2a}\ref{d2:2aB} always implies $k_2=k_1+1$ and $E_{k_2-1}\subset\C{E_{k_2}}$ since $E_{k_1+1}=D_{k_1}^{(1)}$ which is by assumption non-involutive and by definition $E_{k_1}=\C{D_{k_1}^{(1)}}$ in this case). 
			The distribution $E_{k_2}$ is by assumption non-involutive. Since $E_{k_2}\subset\Delta_{k_2+1}\cap\TXU$ (see \eqref{eq:DeltaErelation}) and $\Delta_{k_2+1}\cap\TXU$ is involutive, it follows that $\overline{E}_{k_2}=\Delta_{k_2+1}\cap\TXU=E_{k_2}+\Span{\Ad_{\bar{f}}^{k_2+1}\partial_{\bar{u}^2}}$, and thus we have $\Dim{\overline{E}_{k_2}}=\Dim{E_{k_2}}+1$, which shows the necessity \ref{d2:4a}\ref{d2:4a1}
			
			For $\overline{E}_{k_2}=\TXU$, there are no additional conditions whose necessity has to be shown. For $\overline{E}_{k_2}\neq\TXU$, the necessity of $\Dim{[\bar{f},E_{k_2}]+\overline{E}_{k_2}}=\Dim{\overline{E}_{k_2}}+1$ has to be shown. Since $\Ad_{\bar{f}}^{k_2+1}\partial_{\bar{u}^2}\in\overline{E}_{k_2}$, the condition holds if and only if $\Ad_{\bar{f}}^{k_2}\partial_{\bar{u}^1}\notin\overline{E}_{k_2}$, which can be shown by contradiction. Assume that $\Ad_{\bar{f}}^{k_2}\partial_{\bar{u}^1}\in\overline{E}_{k_2}$ and thus $[\bar{f},E_{k_2}]+\overline{E}_{k_2}=\overline{E}_{k_2}$. Based on the Jacobi identity, it can then be shown that this would imply $[\bar{f},\overline{E}_{k_2}]\subset\overline{E}_{k_2}$, which would in turn imply that the system contains an autonomous subsystem and contradict with the system being flat. By definition, we have $F_{k_2+1}=\overline{E}_{k_2}=E_{k_2}+\Span{\Ad_{\bar{f}}^{k_2+1}\partial_{\bar{u}^2}}$ in this case.\\
			
			The necessity of item \ref{d2:4a}\ref{d2:4a2} under the assumption that the conditions of \ref{d2:2a}\ref{d2:2aB} are met can be shown analogously. By definition, we have $E_{k_2}=D_{k_1}^{(1)}$ (where $k_2=k_1+1$), which according to Lemma \ref{lem:d2case1FirstDerivedFlag} evaluates to $E_{k_2}=D_{k_1}+\Span{\Ad_{\bar{f}}^{k_1+1}\partial_{\bar{u}^2}}$ and is by assumption non-involutive. We have $E_{k_2}\subset\Delta_{k_1+2}\cap\TXU$ (see \eqref{eq:d2twoFoldProlongedDistributionsSimplifiedCase1}), from which $\overline{E}_{k_2}=\Delta_{k_1+2}\cap\TXU=E_{k_2}+\Span{\Ad_{\bar{f}}^{k_1+2}\partial_{\bar{u}^2}}$ follows. For $\overline{E}_{k_2}=\TXU$, there are again no additional conditions whose necessity has to be shown. For $\overline{E}_{k_2}\neq\TXU$, the necessity of $\Dim{[\bar{f},E_{k_2}]+\overline{E}_{k_2}}=\Dim{\overline{E}_{k_2}}+1$ can be shown analogously as above. By definition, we have $F_{k_2+1}=\overline{E}_{k_2}=E_{k_2}+\Span{\Ad_{\bar{f}}^{k_2+1}\partial_{\bar{u}^2}}$ in this case.\\
			\subsubsection*{\textit{The case $E_{k_2-1}\not\subset\C{E_{k_2}}$.}} We have to show the necessity of item \ref{d2:4b}, \ie we have to show that the distribution $F_{k_2}$, defined as $F_{k_2}=E_{k_2-1}+\C{E_{k_2}}$, is involutive. Regarding the Cauchy characteristic distribution of $E_{k_2}$ we have the following result,  a proof of which is provided in Appendix \ref{ap:lemmas}.
			\begin{lemma}\label{lem:cuachyEk2}
				Assume that $E_{k_2-1}\not\subset\C{E_{k_2}}$. If $k_2\geq k_1+2$, the Cauchy characteristic distribution of $E_{k_2}$ is given by
				\begin{align*}
					\begin{aligned}
						\C{E_{k_2}}&=E_{k_2-2}+\Span{\Ad_{\bar{f}}^{k_2-1}\partial_{\bar{u}^2},\Ad_{\bar{f}}^{k_2}\partial_{\bar{u}^2}-\lambda\Ad_{\bar{f}}^{k_2-2}\partial_{\bar{u}^1}}\,,
					\end{aligned}
				\end{align*}
				if $k_2=k_1+1$, it is given by
				\begin{align*}
					\begin{aligned}
						\C{E_{k_2}}&=\Delta_{k_1}\cap D_{k_1-1}+\Span{\Ad_{\bar{f}}^{k_2-1}\partial_{\bar{u}^2},\Ad_{\bar{f}}^{k_2}\partial_{\bar{u}^2}-\lambda\Ad_{\bar{f}}^{k_2-2}\partial_{\bar{u}^1}}\,.
					\end{aligned}
				\end{align*}
			\end{lemma}
			Recall that we have $E_{k_2-1}=\Span{\partial_{\bar{u}^1},\partial_{\bar{u}^2},\ldots,\Ad_{\bar{f}}^{k_2-2}\partial_{\bar{u}^1},\Ad_{\bar{f}}^{k_2-2}\partial_{\bar{u}^2},\Ad_{\bar{f}}^{k_2-1}\partial_{\bar{u}^2}}$ and that $D_{k_1-1}\subset E_{k_1}$. Due to Lemma \ref{lem:cuachyEk2}, for the distribution $F_{k_2}=E_{k_2-1}+\C{E_{k_2}}$, we thus in any case have $F_{k_2}=E_{k_2-1}+\Span{\Ad_{\bar{f}}^{k_2}\partial_{\bar{u}^2}}$ and hence $F_{k_2}=\Delta_{k_2}\cap\TXU$, implying that $F_{k_2}$ is indeed involutive.
			
			\paragraph*{Necessity of Item \ref{d2:5}} In conclusion, in \ref{d2:3a}, \ie for $E_{k_1}$ being non-involutive (in which case we define $k_2=k_1$), we have $F_{k_1+1}=\overline{E}_{k_1}=E_{k_1}+\Span{\Ad_{\bar{f}}^{k_1+1}\partial_{\bar{u}^2}}$. In \ref{d2:4a}, we have $F_{k_2+1}=\overline{E}_{k_2}=E_{k_2}+\Span{\Ad_{\bar{f}}^{k_2+1}\partial_{\bar{u}^2}}$, and in \ref{d2:4b}, we have $F_{k_2}=E_{k_2-1}+\Span{\Ad_{\bar{f}}^{k_2}\partial_{\bar{u}^2}}$. With $F_i=F_{i-1}+[\bar{f},F_{i-1}]$, in any case we thus have $F_{k_2+1}=E_{k_2}+\Span{\Ad_{\bar{f}}^{k_2+1}\partial_{\bar{u}^2}}$. To complete the necessity part of the proof, we have to show that all the distributions $F_i$, $i\geq k_2+1$ are involutive and that there exists an integer $s$ such that $F_s=\TXU$, which is indeed the case since in any case, it follows that the distributions $F_i$ and $\Delta_i$ are related via $\Delta_i=\Span{\partial_{\bar{u}^1_2},\partial_{\bar{u}^1_1}}+F_i$ and thus
			\begin{align*}
				F_{k_2+1}&=\Delta_{k_2+1}\cap\TXU\\
				F_{k_2+2}&=\Delta_{k_2+2}\cap\TXU\\
				&\vdotswithin{=}\\
				F_s&=\Delta_s\cap\TXU=\TXU\,.
			\end{align*}
		\subsection*{Sufficiency}
			Consider a two-input system of the form \eqref{eq:sys} and assume that it meets the conditions of Theorem \ref{thm:d2}. To cover all the possible paths which are illustrated in Figure \ref{fig:d2}, we again have to distinguish between several cases. The distinction is done based on the difference of the integers $k_1$ and $k_2$, \ie the indices of the first non-involutive distribution $D_{k_1}$ and the second non-involutive distribution $E_{k_2}$. The cases in which $k_2\geq k_1+2$, and thus $E_{k_1+1}$ is involutive, are similar an can be proven together. The cases in which $E_{k_1}$ is involutive but $E_{k_1+1}$ is non-involutive, are also similar and can be proven together. The remaining case in which $E_{k_1}$ is non-involutive (in which case we define $k_2=k_1$), is proven separately. For each case, a coordinate transformation such that the system takes a certain structurally flat triangular form can be derived. For the cases $k_2=k_1+1$ and $k_2=k_1$ we derive such a transformation explicitly, the case $k_2\geq k_1+2$ can be handled analogously, it is in fact slightly simpler than the other two cases and we do not treat this case in detail here. In each case, we will need the following result, proven in Appendix \ref{ap:lemmas}.
			\begin{lemma}\label{lem:missingDistributions}
				Let $G_{k-1}$ be an involutive distribution and $G_k$ a non-involutive distribution on $\XU$ such that $G_{k-1}\underset{2}{\subset}G_k$, $G_{k-1}\subset\C{G_k}$, $\Dim{\overline{G}_k}=\Dim{G_k}+1$ and for some vector field $f$ on $\XU$, we have $G_k=G_{k-1}+[f,G_{k-1}]$ and either $\overline{G}_k=\TXU$ or $\Dim{[f,G_k]+\overline{G}_k}=\Dim{\overline{G}_k}+1$. Then, there exists an involutive distribution $H_k$ such that $G_{k-1}\underset{1}{\subset}H_k\underset{1}{\subset}G_k$ and $H_k+[f,H_k]=\overline{G}_k$.
			\end{lemma}
			\paragraph*{\textit{The Case $k_2\geq k_1+2$.}} In total, four different subcases are possible, namely
			\begin{enumerate}[label=(\alph*)]
				\item \ref{d2:2a}\ref{d2:2aA} followed by \ref{d2:4a}, corresponding to $D_{k_1-1}\subset\C{D_{k_1}}$ and $E_{k_2-1}\subset\C{E_{k_2}}$.
				\item \ref{d2:2a}\ref{d2:2aA} followed by \ref{d2:4b}, corresponding to $D_{k_1-1}\subset\C{D_{k_1}}$ and $E_{k_2-1}\not\subset\C{E_{k_2}}$.
				\item \ref{d2:2b} followed by \ref{d2:4a}, corresponding to $D_{k_1-1}\not\subset\C{D_{k_1}}$ and $E_{k_2-1}\subset\C{E_{k_2}}$. 
				\item \ref{d2:2b} followed by \ref{d2:4b}, corresponding to $D_{k_1-1}\not\subset\C{D_{k_1}}$ and $E_{k_2-1}\not\subset\C{E_{k_2}}$.
			\end{enumerate}
			In any of these cases, we have the following sequence of involutive distributions
			\begin{align}\label{eq:d2_k2geqk1+2_invSequence}
				\begin{aligned}
					&D_0\underset{2}{\subset}\ldots\underset{2}{\subset}D_{k_1-1}\underset{1}{\subset}E_{k_1}\underset{2}{\subset}E_{k_1+1}=\overline{D}_{k_1}\underset{2}{\subset}\ldots\underset{2}{\subset}E_{k_2-1}\underset{1}{\subset}F_{k_2}\underset{2}{\subset}F_{k_2+1}=\overline{E}_{k_2}\subset\ldots\subset E_s\,.
				\end{aligned}
			\end{align}
			In the cases in which $D_{k_1-1}\subset\C{D_{k_1}}$ and/or $E_{k_2-1}\subset\C{E_{k_2}}$, the existence of $E_{k_1}$ and/or $F_{k_2}$ is guaranteed by Lemma \ref{lem:missingDistributions}. That the inclusion $F_{k_2}\underset{2}{\subset}F_{k_2+1}$ is indeed of corank $2$ follows from $F_{k_2}\underset{1}{\subset}E_{k_2}\underset{1}{\subset}\overline{E}_{k_2}=F_{k_2+1}$ in this case. 
			
			In case of $D_{k_1-1}\not\subset\C{D_{k_1}}$, the distribution $E_{k_1}$ occurs explicitly in the corresponding conditions of Theorem \ref{thm:d2}, $E_{k_1+1}=\overline{D}_{k_1}$ follows from the assumption that $E_{k_1+1}$ is involutive. Indeed, we have $E_{k_1+1}=E_{k_1}+[f,E_{k_1}]=D_{k_1}+\Span{[f, [v_c,f]]}$, from which $\overline{D}_{k_1}=E_{k_1+1}$ follows. Similarly, in case of $E_{k_2-1}\not\subset\C{E_{k_2}}$, $F_{k_2+1}=\overline{E}_{k_2}$ can be shown as follows. By assumption, $F_{k_2}=E_{k_2-1}+\C{E_{k_2}}$ and $F_{k_2+1}=F_{k_2}+[f,F_{k_2}]$ are involutive and $E_{k_2}$ is non-involutive. It is immediate, that $\C{E_{k_2}}$ necessarily contains a vector field $v_c\in E_{k_2}$, $v_c\notin E_{k_2-1}$, otherwise, $F_{k_2}=E_{k_2-1}$ which would lead to $F_{k_2+1}=E_{k_2}$ and contradict with $E_{k_2}$ being non-involutive. On the other hand, $\C{E_{k_2}}$ can only contain one vector field which is not already contained in $E_{k_2-1}$, otherwise, $F_{k_2}=E_{k_2}$ which again contradicts with $E_{k_2}$ being non-involutive. Thus, $F_{k_2}=E_{k_2-1}+\Span{v_c}$ where $v_c\in E_{k_2}$, $v_c\notin E_{k_2-1}$. Therefore, we have $F_{k_2+1}=E_{k_2}+\Span{[f,v_c]}$. It is immediate that $[f,v_c]\notin E_{k_2}$ (otherwise $F_{k_2+1}=E_{k_2}$). Thus, we have $\overline{E}_{k_2}=F_{k_2+1}$. That $F_{k_2}\underset{2}{\subset}F_{k_2+1}$ follows readily from these considerations in this case.\\
			
			Based on the distributions \eqref{eq:d2_k2geqk1+2_invSequence}, a coordinate transformation such that the system takes the structurally flat triangular form
			\begin{align*}
				\begin{aligned}
					\dot{\bar{x}}_s&=f_s(\bar{x}_s,\bar{x}_{s-1})\\
					&\vdotswithin{=}\\
					\dot{\bar{x}}_{k_2+2}&=f_{k_2+2}(\bar{x}_s,\ldots,\bar{x}_{k_2+1})\\
					\dot{\bar{x}}_{k_2+1}&=\tilde{f}_{k_2+1}(\bar{x}_s,\ldots,\bar{x}_{k_2}^1,\tilde{x}_{k_2-1}^2)\\
					\dot{\bar{x}}_{k_2}^1&=\tilde{f}_{k_2}^1(\bar{x}_s,\ldots,\tilde{x}_{k_2-1})\\
					\dot{\tilde{x}}_{k_2-1}&=\tilde{f}_{k_2-1}(\bar{x}_s,\ldots,\bar{x}_{k_2-2})\\
					&\vdotswithin{=}\\
					\dot{\bar{x}}_{k_1+2}&=\tilde{f}_{k_1+2}(\bar{x}_s,\ldots,\bar{x}_{k_1+1})\\
					\dot{\bar{x}}_{k_1+1}&=\tilde{f}_{k_1+1}(\bar{x}_s,\ldots,\bar{x}_{k_1}^1,\tilde{x}_{k_1-1}^2)\\
					\dot{\bar{x}}_{k_1}^1&=\tilde{f}_{k_1}^1(\bar{x}_s,\ldots,\tilde{x}_{k_1-1})\\
					\dot{\tilde{x}}_{k_1-1}&=\tilde{f}_{k_1-1}(\bar{x}_s,\ldots,\bar{x}_{k_1-2})\\
					&\vdotswithin{=}\\
					\dot{\bar{x}}_2&=\tilde{f}_2(\bar{x}_s,\ldots,\bar{x}_1)\\
					\dot{\bar{x}}_1&=\tilde{f}_1(\bar{x}_s,\ldots,\bar{x}_1,u)\,,
				\end{aligned}
			\end{align*}
			can be derived (in case of $k_1=1$, the variables $\tilde{x}_{k_1-1}$ are actually inputs instead of states). It follows that $y=\bar{x}_s$ forms a flat output with a difference of $d=2$. If $\Dim{\bar{x}_s}=1$, we necessarily have $\Dim{\bar{x}_l}=2$ but $\Dim{f_{l+1}}=1$ for some $l\in\{k_2+1,\ldots,s-1\}$. In this case, a flat output with a difference of $d=2$ is given by $y=(\bar{x}_s,\varphi)$ where $\varphi=\varphi(\bar{x}_s,\ldots,\bar{x}_l)$ is an arbitrary function such that $\Rank{\partial_{\bar{x}_l}(f_{l+1},\varphi)}=2$. Furthermore, it follows that there cannot exist a flat output with $d\leq 1$. The conditions of Theorem \ref{thm:sfl} cannot be met due to the non-involutivity of $D_{k_1}$ and it can easily be shown that the conditions of Theorem \ref{thm:d1} cannot be met either.
			\paragraph*{\textit{The Case $k_2=k_1+1$.}} In total, three different subcases are possible, namely
			\begin{enumerate}[label=(\alph*)]
				\item \ref{d2:2a}\ref{d2:2aB} followed by \ref{d2:4a}\ref{d2:4a2}, corresponding to $D_{k_1-1}\subset\C{D_{k_1}}$.\label{it:a}
				\item \ref{d2:2b} followed by \ref{d2:4a}, corresponding to $D_{k_1-1}\not\subset\C{D_{k_1}}$ and $E_{k_1}\subset\C{E_{k_1+1}}$.\label{it:b}
				\item \ref{d2:2b} followed by \ref{d2:4b}, corresponding to $D_{k_1-1}\not\subset\C{D_{k_1}}$ and $E_{k_1}\not\subset\C{E_{k_1+1}}$.\label{it:c}
			\end{enumerate}
			In all of these cases, we have the following sequence of involutive distributions
			\begin{align}\label{eq:d2_k2=k1+1_invSequence}
				\begin{aligned}
					D_0\underset{2}{\subset}\ldots\underset{2}{\subset}D_{k_1-1}\underset{1}{\subset}E_{k_1}\underset{1}{\subset}F_{k_1+1}\underset{2}{\subset}F_{k_1+2}=\overline{E}_{k_1+1}\subset F_{k_1+3}\subset\ldots\subset F_s\,.
				\end{aligned}
			\end{align}
			Let us show this in detail for the three possible cases \ref{it:a} - \ref{it:c}.\\
			
			\noindent
			\ref{it:a} In \ref{d2:2a}\ref{d2:2aB}, only $E_{k_1+1}$ is defined explicitly. We set $E_{k_1}=\C{D_{k_1}^{(1)}}$ in this case. We have the following results on this Cauchy characteristic distribution, see Appendix \ref{ap:lemmas} for a proof.
			\begin{lemma}\label{lem:cauchy}
				Assume that the conditions of item \ref{d2:2a}\ref{d2:2aB} are met, \ie $D_{k_1-1}\subset\C{D_{k_1}}$, $\Dim{\overline{D}_{k_1}}=\Dim{D_{k_1}}+2$ and $[f,\C{D_{k_1}^{(1)}}]\subset D_{k_1}^{(1)}$. Then, $D_{k_1-1}\underset{1}{\subset}\C{D_{k_1}^{(1)}}\underset{1}{\subset}D_{k_1}$ and $\C{D_{k_1}^{(1)}}+[f,\C{D_{k_1}^{(1)}}]=D_{k_1}^{(1)}$.
			\end{lemma}
			By definition, we have $E_{k_1+1}=D_{k_1}^{(1)}$ and according to Lemma \ref{lem:cauchy}, we have $E_{k_1+1}=E_{k_1}+[f,E_{k_1}]$ with $E_{k_1}=\C{D_{k_1}^{(1)}}$. Therefore, Lemma \ref{lem:missingDistributions} applies, which guarantees the existence of a distribution $F_{k_1+1}$ such that $E_{k_1}\underset{1}{\subset}F_{k_1+1}\underset{1}{\subset}E_{k_1+1}$ and $F_{k_1+1}+[f,F_{k_1+1}]=\overline{E}_{k_1+1}$ (it turns out that $F_{k_1+1}$ is unique if and only if $\overline{D}_{k_1}\neq\TXU$).\\
			
			\noindent
			\ref{it:b} In \ref{d2:2b}, the distribution $E_{k_1}$ occurs explicitly, the existence of an involutive distribution $F_{k_1+1}$ such that $E_{k_1}\underset{1}{\subset}F_{k_1+1}\underset{1}{\subset}E_{k_1+1}$ and $F_{k_1+1}+[f,F_{k_1+1}]=\overline{E}_{k_1+1}$ is guaranteed by Lemma \ref{lem:missingDistributions} (it turns out that $F_{k_1+1}$ is unique if and only if $\overline{E}_{k_1+1}\neq\TXU$).\\
						
			\noindent
			\ref{it:c} In \ref{d2:2b} and \ref{d2:4b}, the distributions $E_{k_1}$ and $F_{k_1+1}$ occur explicitly. That $\Dim{F_{k_1+1}}=\Dim{E_{k_1}}+1$ as well as $F_{k_1+2}=\overline{E}_{k_1+1}$ and $\Dim{F_{k_1+2}}=\Dim{F_{k_1+1}}+2$ can be shown the same way as above for the case $k_2\geq k_1+2$.\\
			
			Based on the distributions \eqref{eq:d2_k2=k1+1_invSequence} we will below derive a coordinate transformation such that the system takes the structurally flat triangular form
			\begin{align}\label{eq:d2sysTriangularCase2}
				\begin{aligned}
					\dot{\bar{x}}_s&=f_s(\bar{x}_s,\bar{x}_{s-1})\\
					&\vdotswithin{=}\\
					\dot{\bar{x}}_{k_1+3}&=f_{k_1+3}(\bar{x}_s,\ldots,\bar{x}_{k_1+2})\\
					\dot{\bar{x}}_{k_1+2}&=\tilde{f}_{k_1+2}(\bar{x}_s,\ldots,\bar{x}_{k_1+1}^1,\tilde{x}_{k_1-1}^2)\\
					\dot{\bar{x}}_{k_1+1}^1&=\tilde{f}_{k_1+1}^1(\bar{x}_s,\ldots,\bar{x}_{k_1}^1,\tilde{x}_{k_1-1}^2)\\
					\dot{\bar{x}}_{k_1}^1&=\tilde{f}_{k_1}^1(\bar{x}_s,\ldots,\tilde{x}_{k_1-1})\\
					\dot{\tilde{x}}_{k_1-1}&=\tilde{f}_{k_1-1}(\bar{x}_s,\ldots,\bar{x}_{k_1-2})\\
					&\vdotswithin{=}\\
					\dot{\bar{x}}_2&=\tilde{f}_2(\bar{x}_s,\ldots,\bar{x}_1)\\
					\dot{\bar{x}}_1&=\tilde{f}_1(\bar{x}_s,\ldots,\bar{x}_1,u)\,,
				\end{aligned}
			\end{align}
			(in case of $k_1=1$, the variables $\tilde{x}_{k_1-1}$ are actually inputs instead of states) from which again flatness with a difference of $d=2$ can be deduced.
			\paragraph*{\textit{The Case $k_2=k_1$.}} In this case, by assumption the conditions of item \ref{d2:2b} and item \ref{d2:3a} are met.
			We have the following sequence of involutive distributions (we assume $k_1\geq 2$ here, see Remark \ref{rem:k1} below for the cases $k_1=1$)
			\begin{align}\label{eq:d2_k2=k1_invSequence}
				\begin{aligned}
					D_0\underset{2}{\subset}\ldots\underset{2}{\subset}D_{k_1-2}\underset{1}{\subset}E_{k_1-1}\underset{1}{\subset}F_{k_1}\underset{2}{\subset}F_{k_1+1}=\overline{E}_{k_1}\subset F_{k_1+2}\subset\ldots\subset F_s\,.
				\end{aligned}
			\end{align}
			That $E_{k_1-1}$ is involutive follows from the fact that $E_{k_1-1}=\C{E_{k_1}}$ in this case. (Indeed, since $E_{k_1}$ is non-involutive, we necessarily have $\Dim{\C{E_{k_1}}}\leq\Dim{E_{k_1}}-2$. By construction we have $E_{k_1-1}\subset\C{E_{k_1}}$ and $E_{k_1-1}\underset{2}{\subset}E_{k_1}$, which thus implies $E_{k_1-1}=\C{E_{k_1}}$.) The existence of $F_{k_1}$ such that $E_{k_1-1}\underset{1}{\subset}F_{k_1}\underset{1}{\subset}E_{k_1}$ and $F_{k_1}+[f,F_{k_1}]=\overline{E}_{k_1}$ is guaranteed by Lemma \ref{lem:missingDistributions} (and $F_{k_1}$ is unique since $\overline{E}_{k_1}=\TXU$ is not possible as it would imply $D_{k_1}=\TXU$).
			Based on this sequence, we will derive a coordinate transformation such that the system takes the structurally flat triangular form
			\begin{align}\label{eq:d2sysTriangularCase3}
				\begin{aligned}
					\dot{\bar{x}}_s&=f_s(\bar{x}_s,\bar{x}_{s-1})\\
					&\vdotswithin{=}\\
					\dot{\bar{x}}_{k_1+2}&=f_{k_1+2}(\bar{x}_s,\ldots,\bar{x}_{k_1+1})\\
					\dot{\bar{x}}_{k_1+1}&=\tilde{f}_{k_1+1}(\bar{x}_s,\ldots,\bar{x}_{k_1}^1,\tilde{x}_{k_1-2}^2)\\
					\dot{\bar{x}}_{k_1}^1&=\tilde{f}_{k_1}^1(\bar{x}_s,\ldots,\bar{x}_{k_1-1}^1,\tilde{x}_{k_1-2}^2)\\
					\dot{\bar{x}}_{k_1-1}^1&=\tilde{f}_{k_1-1}^1(\bar{x}_s,\ldots,\tilde{x}_{k_1-2})\\
					\dot{\tilde{x}}_{k_1-2}&=\tilde{f}_{k_1-2}(\bar{x}_s,\ldots,\bar{x}_{k_1-3})\\
					&\vdotswithin{=}\\
					\dot{\bar{x}}_2&=\tilde{f}_2(\bar{x}_s,\ldots,\bar{x}_1)\\
					\dot{\bar{x}}_1&=\tilde{f}_1(\bar{x}_s,\ldots,\bar{x}_1,u)\,,
				\end{aligned}
			\end{align}
			(in case of $k_1=2$, the variables $\tilde{x}_{k_1-2}$ are actually inputs instead of states) from which again flatness with a difference of $d=2$ can be deduced.
			\paragraph*{\textit{The Case $k_2=k_1+1$ in Detail.}} 
			Apply a change of coordinates $\bar{x}=\Phi_x(x)$ such that all the distributions \eqref{eq:d2_k2=k1+1_invSequence} get straightened out simultaneously, \ie such that
			\begin{align*}
				\begin{aligned}
					D_i&=\Span{\partial_{\bar{u}},\partial_{\bar{x}_1},\ldots,\partial_{\bar{x}_i}}\,,&i&=0,\ldots,k_1-1\,,
				\end{aligned}
			\end{align*}
			where $\Dim{\bar{x}_i}=2$, and
			\begin{align*}
				\begin{aligned}
					E_{k_1}&=\Span{\partial_{\bar{u}},\partial_{\bar{x}_1},\ldots,\partial_{\bar{x}_{k_1-1}},\partial_{\bar{x}_{k_1}^1}}\,,
				\end{aligned}
			\end{align*}
			where $\Dim{\bar{x}_{k_1}^1}=1$, and 
			\begin{align*}
				\begin{aligned}
					F_i&=\Span{\partial_{\bar{u}},\partial_{\bar{x}_1},\ldots,\partial_{\bar{x}_{k_1-1}},\partial_{\bar{x}_{k_1}^1},\partial_{\bar{x}_{k_1+1}^1},\partial_{\bar{x}_{k_1+2}},\ldots,\partial_{\bar{x}_i}}\,,&i&=k_1+1,\ldots,s\,,
				\end{aligned}
			\end{align*}
			where $\Dim{\bar{x}_{k_1+1}^1}=1$ and $\Dim{\bar{x}_{k_1+2}}=2$. Since by construction, we have $[f,D_i]\subset D_{i+1}$ for $i=0,\ldots,k_1-2$ as well as $[f,F_i]\subset F_{i+1}$ for $i=k_1+1,\ldots,s-1$, it follows that in the new coordinates, the vector field $f$ has the triangular form
			\begin{align*}
				\begin{aligned}
					f&=\begin{bmatrix}
					f_s(\bar{x}_s,\bar{x}_{s-1})\\
					\vdots\\
					f_{k_1+3}(\bar{x}_s,\ldots,\bar{x}_{k_1+2})\\
					f_{k_1+2}(\bar{x}_s,\ldots,\bar{x}_{k_1-1})\\
					f_{k_1+1}^1(\bar{x}_s,\ldots,\bar{x}_{k_1-1})\\
					f_{k_1}^1(\bar{x}_s,\ldots,\bar{x}_{k_1-1})\\
					f_{k_1-1}(\bar{x}_s,\ldots,\bar{x}_{k_1-2})\\
					\vdots\\
					f_2(\bar{x}_s,\ldots,\bar{x}_1)\\
					f_1(\bar{x}_s,\ldots,\bar{x}_1,u)
					\end{bmatrix}\,.
				\end{aligned}
			\end{align*}
			From $D_{i+1}=D_i+[f,D_i]$, $i=0,\ldots,k_1-2$ and $F_{i+1}=F_i+[f,F_i]$, $i=k_1+1,\ldots,s-1$, it follows that the rank conditions $\Rank{\partial_uf_1}=\Dim{f_1}=2$, $\Rank{\partial_{\bar{x}_i}f_{i+1}}=\Dim{f_{i+1}}=2$ for $i=1,\ldots,k_1-2$ and $\Rank{\partial_{\bar{x}_i}f_{i+1}}=\Dim{f_{i+1}}$ for $i=k_1+2,\ldots,s-1$, hold.\\
			
			Since $D_{k_1-1}\underset{2}{\subset}D_{k_1}$, it follows that $\Rank{\partial_{\bar{x}_{k_1-1}}(f_{k_1+2},f_{k_1+1}^1,f_{k_1}^1)}=2$. The non-involutivity of $D_{k_1}$ implies that at least one of the components of $f_{k_1+2}$ must explicitly depend on $\bar{x}_{k_1-1}$ (otherwise, $D_{k_1}$ would be involutive as it would be straightened out in the $\bar{x}$ coordinates). Without loss of generality, we can thus assume that $f_{k_1+2}^2$ explicitly depends on $\bar{x}_{k_1-1}^2$, if not, permute $\bar{x}_{k_1-1}^1$ and $\bar{x}_{k_1-1}^2$ and/or $f_{k_1+2}^1$ and $f_{k_1+2}^2$. We can thus introduce $\tilde{x}_{k_1-1}^2=f_{k_1+2}^2$ with all the other coordinates left unchanged (in case of $k_1=1$ this is actually an input transformation instead of a state transformation), resulting in
			\begin{align*}
				\begin{aligned}
					\tilde{f}_{k_1+2}&=\begin{bmatrix}
						\tilde{x}_{k_1-1}^2\\
						\tilde{f}_{k_1+2}^1(\bar{x}_s,\ldots,\tilde{x}_{k_1-1})
					\end{bmatrix}\,,
				\end{aligned}
			\end{align*}
			(where $\tilde{x}_{k_1-1}=(\tilde{x}_{k_1-1}^2,\bar{x}_{k_1-1}^1)$). In any case, we have $E_{k_1}\underset{1}{\subset}D_{k_1}$.\footnote{Indeed, in \ref{d2:2b}, this follows immediately from the definition of $E_{k_1}$, for the case \ref{d2:2a}\ref{d2:2aB}, it follows from Lemma \ref{lem:cauchy}.} Thus, a linear combination of $[\partial_{\bar{x}_{k_1-1}^1},f]$ and $[\partial_{\tilde{x}_{k_1-1}^2},f]$ is contained in $E_{k_1}$, implying that $\Rank{\partial_{\tilde{x}_{k_1-1}}(\tilde{f}_{k_1+2},\tilde{f}_{k_1+1}^1)}=1$ and thus, $\tilde{f}_{k_1+2}^1$ and $\tilde{f}_{k_1+1}^1$ are actually independent of $\bar{x}_{k_1-1}^1$, so we have
			\begin{align*}
				\begin{aligned}
					\tilde{f}_{k_1+2}&=\begin{bmatrix}
						\tilde{x}_{k_1-1}^2\\
						\tilde{f}_{k_1+2}^1(\bar{x}_s,\ldots,\bar{x}_{k_1}^1,\tilde{x}_{k_1-1}^2)
					\end{bmatrix}
				\end{aligned}
			\end{align*}
			and
			\begin{align*}
				\begin{aligned}
					\tilde{f}_{k_1+1}^1&=\tilde{f}_{k_1+1}^1(\bar{x}_s,\ldots,\bar{x}_{k_1}^1,\tilde{x}_{k_1-1}^2)\,.
				\end{aligned}
			\end{align*}
			Furthermore, we have $E_{k_1}\underset{2}{\subset}E_{k_1+1}$ and thus $\Rank{\partial_{(\bar{x}_{k_1}^1,\tilde{x}_{k_1-1}^2)}(\tilde{f}_{k_1+2},\tilde{f}_{k_1+1}^1)}=2$ and since $F_{k_1+1}\underset{1}{\subset}E_{k_1+1}$, a linear combination of $[\partial_{\bar{x}_{k_1-1}^2},f]$ and $[\partial_{\tilde{x}_{k_1}^1},f]$ is contained in $F_{k_1+1}$, implying that $\Rank{\partial_{(\bar{x}_{k_1}^1,\tilde{x}_{k_1-1}^2)}\tilde{f}_{k_1+2}}=1$ and thus, $\tilde{f}_{k_1+2}^1$ is actually independent of $\bar{x}_{k_1}^1$, so we have
			\begin{align*}
				\begin{aligned}
					\tilde{f}_{k_1+2}&=\begin{bmatrix}
					\tilde{x}_{k_1-1}^2\\
					\tilde{f}_{k_1+2}^1(\bar{x}_s,\ldots,\bar{x}_{k_1+1}^1,\tilde{x}_{k_1-1}^2)
					\end{bmatrix}
				\end{aligned}
			\end{align*}
			and $\tilde{f}_{k_1+1}^1$ must explicitly depend on $\bar{x}_{k_1}^1$. Because of $F_{k_1+1}\underset{2}{\subset}F_{k_1+2}$, we have $\Rank{\partial_{(\tilde{x}_{k_1-1}^2,\bar{x}_{k_1+1}^1)}\tilde{f}_{k_1+2}}=2$, from which it follows that $\tilde{f}_{k_1+2}^1$ explicitly depends on $\bar{x}_{k_1+1}^1$. (The non-involutivity of $E_{k_1+1}$ furthermore implies that $\partial_{\tilde{x}_{k_1-1}^2}^2\tilde{f}_{k_1+2}^1\neq 0$ and/or $\partial_{\bar{x}_{k_1+1}^1}\partial_{\tilde{x}_{k_1-1}^2}\tilde{f}_{k_1+2}^1\neq 0$.) So in the just constructed coordinates, the system equations indeed take the triangular structure \eqref{eq:d2sysTriangularCase2}.
			\begin{remark}
				In case \ref{it:a}, \ie if $D_{k_1-1}\subset\C{D_{k_1}}$, by successively introducing new coordinates in \eqref{eq:d2sysTriangularCase2} from top to bottom, the system would take the triangular from proposed in \cite{GstottnerKolarSchoberl:2020a}.
			\end{remark}
			\paragraph*{\textit{The Case $k_2=k_1$ in Detail.}}			
			Apply a change of coordinates $\bar{x}=\Phi_x(x)$ such that all the distributions \eqref{eq:d2_k2=k1_invSequence} get straightened out simultaneously, \ie such that
			\begin{align*}
				\begin{aligned}
					D_i&=\Span{\partial_{\bar{u}},\partial_{\bar{x}_1},\ldots,\partial_{\bar{x}_i}}\,,&i&=0,\ldots,k_1-2\,,
				\end{aligned}
			\end{align*}
			where $\Dim{\bar{x}_i}=2$, and
			\begin{align*}
				\begin{aligned}
					E_{k_1-1}&=\Span{\partial_{\bar{u}},\partial_{\bar{x}_1},\ldots,\partial_{\bar{x}_{k_1-2}},\partial_{\bar{x}_{k_1-1}^1}}\,,
				\end{aligned}
			\end{align*}
			where $\Dim{\bar{x}_{k_1-1}^1}=1$, and 
			\begin{align*}
				\begin{aligned}
					F_i&=\Span{\partial_{\bar{u}},\partial_{\bar{x}_1},\ldots,\partial_{\bar{x}_{k_1-2}},\partial_{\bar{x}_{k_1-1}^1},\partial_{\bar{x}_{k_1}^1},\partial_{\bar{x}_{k_1+1}},\ldots,\partial_{\bar{x}_i}}\,,&i&=k_1,\ldots,s\,,
				\end{aligned}
			\end{align*}
			where $\Dim{\bar{x}_{k_1}^1}=1$ and $\Dim{\bar{x}_{k_1+1}}=2$. Since by construction, we have $[f,D_i]\subset D_{i+1}$ for $i=0,\ldots,k_1-3$ as well as $[f,F_i]\subset F_{i+1}$ for $i=k_1,\ldots,s-1$, it follows that in the new coordinates, the vector field $f$ has the triangular form 
			\begin{align*}
				\begin{aligned}
					f&=\begin{bmatrix}
					f_s(\bar{x}_s,\bar{x}_{s-1})\\
					\vdots\\
					f_{k_1+2}(\bar{x}_s,\ldots,\bar{x}_{k_1+1})\\
					f_{k_1+1}(\bar{x}_s,\ldots,\bar{x}_{k_1-2})\\
					f_{k_1}^1(\bar{x}_s,\ldots,\bar{x}_{k_1-2})\\
					f_{k_1-1}^1(\bar{x}_s,\ldots,\bar{x}_{k_1-2})\\
					f_{k_1-2}(\bar{x}_s,\ldots,\bar{x}_{k_1-3})\\
					\vdots\\
					f_2(\bar{x}_s,\ldots,\bar{x}_1)\\
					f_1(\bar{x}_s,\ldots,\bar{x}_1,u)
					\end{bmatrix}\,.
				\end{aligned}
			\end{align*}
			From $D_{i+1}=D_i+[f,D_i]$, $i=0,\ldots,k_1-3$ and $F_{i+1}=F_i+[f,F_i]$, $i=k_1,\ldots,s-1$, it follows that the rank conditions $\Rank{\partial_uf_1}=\Dim{f_1}=2$, $\Rank{\partial_{\bar{x}_i}f_{i+1}}=\Dim{f_{i+1}}=2$ for $i=1,\ldots,k_1-3$ and $\Rank{\partial_{\bar{x}_i}f_{i+1}}=\Dim{f_{i+1}}$ for $i=k_1+1,\ldots,s-1$, hold.\\
			
			It follows that at least one component of $f_{k_1+1}$ explicitly depends on $\bar{x}_{k_1-2}$. Indeed, we have $D_{k_1-2}\underset{2}{\subset}D_{k_1-1}$. If $f_{k_1+1}$ would be independent of $\bar{x}_{k_1-2}$, we would obtain $D_{k_1-1}=D_{k_1-2}+\Span{\partial_{\bar{x}_{k_1-1}^1},\partial_{\bar{x}_{k_1}^1}}=F_{k_1}$ and in turn $D_{k_1}=F_{k_1+1}$, which contradicts with $D_{k_1}$ being non-involutive. Without loss of generality, we can thus assume that $f_{k_1+1}^2$ explicitly depends on $\bar{x}_{k_1-2}^2$, if not, permute $\bar{x}_{k_1-2}^1$ and $\bar{x}_{k_1-2}^2$ and/or $f_{k_1+1}^1$ and $f_{k_1+1}^2$. We can thus introduce $\tilde{x}_{k_1-2}^2=f_{k_1+1}^2$ with all the other coordinates left unchanged (in case of $k_1=2$ this is actually an input transformation instead of a state transformation), resulting in
			\begin{align*}
				\begin{aligned}
					\tilde{f}_{k_1+1}&=\begin{bmatrix}
						\tilde{x}_{k_1-2}^2\\
						\tilde{f}_{k_1+1}^1(\bar{x}_s,\ldots,\tilde{x}_{k_1-2})
					\end{bmatrix}
				\end{aligned}
			\end{align*}
			(where $\tilde{x}_{k_1-2}=(\tilde{x}_{k_1-2}^2,\bar{x}_{k_1-2}^1)$). Since $F_{k_1}\underset{1}{\subset}E_{k_1}$, it follows that $[f,E_{k_1-1}]$ yields exactly one new direction which is not already contained in $F_{k_1}$ and thus, $\Rank{\partial_{(\bar{x}_{k_1-1}^1,\tilde{x}_{k_1-2}^2,\bar{x}_{k_1-2}^1)}\tilde{f}_{k_1+1}}=1$ and therefore, $\tilde{f}_{k_1+1}^1$ is actually independent of $\bar{x}_{k_1-2}^1$ and $\bar{x}_{k_1-1}^1$, \ie we actually have
			\begin{align*}
				\begin{aligned}
					\tilde{f}_{k_1+1}&=\begin{bmatrix}
					\tilde{x}_{k_1-2}^2\\
					\tilde{f}_{k_1+1}^1(\bar{x}_s,\ldots,\bar{x}_{k_1}^1,\tilde{x}_{k_1-2}^2)
					\end{bmatrix}\,.
				\end{aligned}
			\end{align*}
			Since $F_{k_1}\underset{2}{\subset}F_{k_1+1}$, it follows that $f_{k_1+1}^1$ explicitly depends on $\bar{x}_{k_1}^1$. Similarly, we have $E_{k_1-1}\underset{1}{\subset}D_{k_1-1}$, from which it follows that $[f,D_{k_1-2}]$ yields exactly one new direction which is not already contained in $E_{k_1-1}$ and thus, $\Rank{\partial_{(\tilde{x}_{k_1-2}^2,\bar{x}_{k_1-2}^1)}(\tilde{f}_{k_1+1},\tilde{f}_{k_1})}=1$ and therefore, $\tilde{f}_{k_1}^1$ is actually independent of $\bar{x}_{k_1-2}^1$, \ie we actually have
			\begin{align*}
				\begin{aligned}
					\tilde{f}_{k_1}^1&=\tilde{f}_{k_1}^1(\bar{x}_s,\ldots,\bar{x}_{k_1-1}^1,\tilde{x}_{k_1-2}^2)\,.
				\end{aligned}
			\end{align*}
			Since $E_{k_1-1}\underset{2}{\subset}E_{k_1}$, it follows that $\tilde{f}_{k_1}^1$ explicitly depends on $\bar{x}_{k_1-1}^1$. So in the just constructed coordinates, the system equations indeed take the triangular structure \eqref{eq:d2sysTriangularCase3}. (Note that the involutivity of $D_{k_1-1}$ implies that the functions $\tilde{f}_{k_1+1}^1$ and $\tilde{f}_{k_1}^1$ depend on $\tilde{x}_{k_1-2}^2$ in an affine manner and that furthermore $\partial_{\bar{x}_{k_1-1}^1}\partial_{\tilde{x}_{k_1-2}^2}\tilde{f}_{k_1}^1=0$. The non-involutivity of $E_{k_1}$ implies that $\partial_{\bar{x}_{k_1}^1}\partial_{\tilde{x}_{k_1-2}^2}\tilde{f}_{k_1+1}^1\neq 0$.)
			\begin{remark}
				By successively introducing new coordinates in \eqref{eq:d2sysTriangularCase3} from top to bottom, the system would take the triangular from proposed in \cite{GstottnerKolarSchoberl:2020c}.
			\end{remark}
			\begin{remark}\label{rem:k1}
				%In case of $k_1=2$, the transformation \eqref{eq:normalization} is actually an input transformation and not a state transformation since the variables $\bar{x}_{k_1-2}$ would then correspond to inputs of the system. 
				In case of $k_1=1$, the variables $\bar{x}_{k_1-1}$ would correspond to inputs of the system and the variables $\bar{x}_{k_1-2}$ would not exist. Consider the system obtained by one-fold prolonging both of its inputs, \ie
				\begin{align*}
					\begin{aligned}
						\dot{x}&=f(x,u)&&&\dot{u}^1&=u^1_1&&&\dot{u}^2&=u^2_1
					\end{aligned}
				\end{align*}
				with the state $(x,u^1,u^2)$ and the input $(u^1_1,u^2_1)$. By assumption, the original system meets the conditions of Theorem \ref{thm:d2} with $k_1=k_2=1$. It can easily be shown that this implies that the prolonged system also meets the conditions of Theorem \ref{thm:d2}, but with $k_1=k_2=2$. (The distributions involved in the conditions of Theorem \ref{thm:d2} when applying it to the original system and when applying it to the prolonged system differ only by $\Span{\partial_{u^1_1},\partial_{u^2_1}}$.) The prolonged system can thus be transformed into the corresponding triangular form \eqref{eq:d2sysTriangularCase3} as explained above, \ie the prolonged system can be proven to be flat with a difference of $d=2$. It can be shown that the prolonged system is flat with a certain difference $d$ if and only if the original system is flat with the same difference $d$. In fact, a flat output of the original system with a certain difference $d$ is also a flat output of the prolonged system with the same difference $d$ and vice versa. The prolonged system being flat with a difference of $d=2$ thus implies that the original system is flat with a difference of $d=2$.
			\end{remark}
	\section{Brief Sketch of the Proof of Theorem \ref{thm:d1}}\label{se:proofd1}
		\subsection*{Necessity}
			Consider a two-input system of the form \eqref{eq:sys} and assume that it is flat with a difference of $d=1$. According to Theorem \ref{thm:linearizationByProlongations}, there exists an input transformation $\bar{u}=\Phi_u(x,u)$ with inverse $u=\hat{\Phi}(x,\bar{u})$ such that the system obtained by one-fold prolonging the new input $\bar{u}^1$, \ie the system
			\begin{align}\label{eq:d1oneFoldProlonged}
				\begin{aligned}
					\dot{x}&=f(x,\hat{\Phi}_u(x,\bar{u}))=\bar{f}(x,\bar{u})&&&
					\dot{\bar{u}}^1&=\bar{u}^1_1
				\end{aligned}
			\end{align}
			with the state $(x,\bar{u}^1)$ and the input $(\bar{u}^1_1,\bar{u}^2)$, is static feedback linearizable. The necessity of the conditions of Theorem \ref{thm:d1} can then be shown on basis of the involutive distributions $\Delta_i=\Delta_{i-1}+[f_p,\Delta_{i-1}]$, where $\Delta_0=\Span{\partial_{\bar{u}^1_1},\partial_{\bar{u}^2}}$ and $f_p=\bar{f}^i(x,\bar{u})\partial_{x^i}+\bar{u}^1_1\partial_{\bar{u}^1}$, which are involved in the test for static feedback linearizability of the prolonged system (see Theorem \ref{thm:sfl}).
			
			The necessity part of the proof is in fact very similar to the proof of the necessity of the items \ref{d2:3b} to \ref{d2:5} of Theorem \ref{thm:d2}. As we have already noted above, these items in fact coincide with the items \ref{d1:1} to \ref{d1:3} of Theorem \ref{thm:d1} when $D_i$ and $k_1$ are replaced by $E_i$ and $k_2$. 
		\subsection*{Sufficiency}
			Consider a two-input system of the form \eqref{eq:sys} and assume that it meets the conditions of Theorem \ref{thm:d1}. In any of the two cases, \ie independent of $D_{k_1-1}\subset\C{D_{k_1}}$ (which corresponds to \ref{d1:2a}) or $D_{k_1-1}\not\subset\C{D_{k_1}}$ (which corresponds to \ref{d1:2b}), it follows that we have the following sequence of involutive distributions
			\begin{align}\label{eq:d1invSequence}
				\begin{aligned}
					&D_0\underset{2}{\subset}D_1\underset{2}{\subset}\ldots\underset{2}{\subset}D_{k_1-1}\underset{1}{\subset}E_{k_1}\underset{2}{\subset}E_{k_1+1}=\overline{D}_{k_1}\subset\ldots\subset E_s=\TXU\,,
				\end{aligned}
			\end{align}
			based on which a change of coordinates such that the system takes the structurally flat triangular form
			\begin{align}\label{eq:d1sysTriangular}
				\begin{aligned}
					\dot{\bar{x}}_s&=f_s(\bar{x}_s,\bar{x}_{s-1})\\
					&\vdotswithin{=}\\
					\dot{\bar{x}}_{k_1+2}&=f_{k_1+2}(\bar{x}_s,\ldots,\bar{x}_{k_1+1})\\
					\dot{\bar{x}}_{k_1+1}&=\tilde{f}_{k_1+1}(\bar{x}_s,\ldots,\bar{x}_{k_1}^1,\tilde{x}_{k_1-1}^2)\\
					\dot{\bar{x}}_{k_1}^1&=\tilde{f}_{k_1}^1(\bar{x}_s,\ldots,\tilde{x}_{k_1-1})\\
					\dot{\tilde{x}}_{k_1-1}&=\tilde{f}_{k_1-1}(\bar{x}_s,\ldots,\bar{x}_{k_1-2})\\
					&\vdotswithin{=}\\
					\dot{\bar{x}}_2&=\tilde{f}_2(\bar{x}_s,\ldots,\bar{x}_1)\\
					\dot{\bar{x}}_1&=\tilde{f}_1(\bar{x}_s,\ldots,\bar{x}_1,u)\,,
				\end{aligned}
			\end{align}
			can be derived (in case of $k_1=1$, the variables $\tilde{x}_{k_1-1}$ are actually inputs instead of states), from which it follows that the system is indeed flat with a difference of $d=1$. 
			\begin{remark}
				In case of \ref{d1:2a}, \ie if $D_{k_1-1}\subset\C{D_{k_1}}$, by successively introducing new coordinates in \eqref{eq:d1sysTriangular} from top to bottom, the system would take the triangular from proposed in \cite{GstottnerKolarSchoberl:2020a}.
			\end{remark}
	\bibliographystyle{apacite}
	\bibliography{Bibliography}
	\appendix
	\section{Supplements}\label{ap:supplements}
		\subsection{Proof of Lemma \ref{lem:algWithinUniqueness}}
			By assumption, we have $D_{k_1-1}=D_{k_1-2}+[f,D_{k_1-2}]$ and $D_{k_1}=D_{k_1-1}+[f,D_{k_1-1}]$ with $D_{k_1-2}$ and $D_{k_1-1}$ being involutive. For every vector field $w\in D_{k_1-2}$ and $v\in D_{k_1-1}$, the Jacobi identity
			\begin{align*}
				\begin{aligned}
					\underbrace{[v,\underbrace{[w,f]}_{\in D_{k_1-1}}]}_{\in D_{k_1-1}\subset D_{k_1}}+\underbrace{[f,\underbrace{[v,w]}_{\in D_{k_1-1}}]}_{\in D_{k_1}}+[w,\underbrace{[f,v]}_{\in D_{k}}]&=0
				\end{aligned}
			\end{align*}
			holds, from which it follows that $D_{k_1-2}\subset\C{D_{k_1}}$. We furthermore have $D_{k_1-2}\underset{2}{\subset}D_{k_1-1}\underset{2}{\subset}D_{k_1}$ and therefore, with suitable vector fields $v_i$, we have $D_{k_1-1}=D_{k_1-2}+\Span{v_1,v_2}$ and $D_{k_1}=D_{k_1-1}+\Span{[v_1,f],[v_2,f]}$. By assumption, we have $D_{k_1-1}\not\subset\C{D_{k_1}}$, which in turn implies that at least one of the vector fields $[v_1,[v_1,f]]$, $[v_1,[v_2,f]]$ or $[v_2,[v_2,f]]$\footnote{Note that we have $[v_2,[v_1,f]]=[v_1,[v_2,f]]\Mod D_{k_1}$.} in \eqref{eq:algWithin}, \ie 
			\begin{align}\label{eq:algWithinProof}
				\begin{aligned}
					(\alpha^1)^2[v_1,[v_1,f]]+2\alpha^1\alpha^2[v_1,[v_2,f]]+(\alpha^2)^2[v_2,[v_2,f]]&\overset{!}{\in} D_{k_1}\,,
				\end{aligned}
			\end{align}
			is not contained in $D_{k_1}$. If they are all linearly independent modulo $D_{k_1}$, \ie $\Dim{D_{k_1}+[D_{k_1-1},D_{k_1}]}=\Dim{D_{k_1}}+3$, then in order for \eqref{eq:algWithinProof} to hold, each coefficient has to be zero, \ie $(\alpha^1)^2=0$, $2\alpha^1\alpha^2=0$ and $(\alpha^2)^2=0$, which only admits the trivial solution $\alpha^1=\alpha^2=0$. So we only have to consider the cases $\Dim{D_{k_1}+[D_{k_1-1},D_{k_1}]}=\Dim{D_{k_1}}+2$ and $\Dim{D_{k_1}+[D_{k_1-1},D_{k_1}]}=\Dim{D_{k_1}}+1$, \ie the cases that $[D_{k_1-1},D_{k_1}]$ yields $2$ new directions or $1$ new direction with respect to $D_{k_1}$, respectively.
			\paragraph*{Case 1:} Let us first consider the case $\Dim{D_{k_1}+[D_{k_1-1},D_{k_1}]}=\Dim{D_{k_1}}+2$, \ie $H=D_{k_1}+[D_{k_1-1},D_{k_1}]=D_{k_1}+\Span{w_1,w_2}$, where $w_1$ and $w_2$ are a suitable selection of the three vector fields $[v_1,[v_1,f]]$, $[v_1,[v_2,f]]$ and $[v_2,[v_2,f]]$. We distinguish between the following subcases:
			\begin{enumerate}
				\item $H=D_{k_1}+\Span{[v_1,[v_1,f]],[v_1,[v_2,f]]}$ and thus $[v_2,[v_2,f]]=\kappa^1[v_1,[v_1,f]]+\kappa^2[v_1,[v_2,f]]\Mod D_{k_1}$, which inserted into \eqref{eq:algWithinProof} yields
				\begin{align*}
					\begin{aligned}
						\left((\alpha^1)^2+(\alpha^2)^2\kappa^1\right)[v_1,[v_1,f]]+\left(2\alpha^1\alpha^2+(\alpha^2)^2\kappa^2\right)[v_1,[v_2,f]]\overset{!}{\in}D_{k_1}\,.
					\end{aligned}
				\end{align*}
				Since $[v_1,[v_1,f]]$ and $[v_1,[v_2,f]]$ are by assumption not collinear$\Mod D_1$, the (up to a multiplicative factor) only non-trivial solution is $\alpha^1=\kappa^2$ and $\alpha^2=-2$, provided that $(\kappa^2)^2+4\kappa^1=0$, otherwise, no non-trivial solution exists.
				If subcase (1) is not applicable, the vector fields $[v_1,[v_1,f]]$ and $[v_1,[v_2,f]]$ are collinear$\Mod D_{k_1}$, \ie $[v_1,[v_2,f]]=\kappa [v_1,[v_1,f]]\Mod D_{k_1}$ or $[v_1,[v_1,f]]\in D_{k_1}$. 
				\item The subcase $[v_1,[v_2,f]]=\kappa [v_1,[v_1,f]]\Mod D_{k_1}$ yields
				\begin{align*}
					\begin{aligned}
						\left((\alpha^1)^2+2\alpha^1\alpha^2\kappa\right)[v_1,[v_1,f]]+(\alpha^2)^2[v_2,[v_2,f]]\overset{!}{\in}D_{k_1}\,.
					\end{aligned}
				\end{align*}
				Since the vector fields $[v_1,[v_1,f]]$ and $[v_2,[v_2,f]]$ are by assumption not collinear$\Mod D_{k_1}$, the condition only admits the trivial solution $\alpha^1=\alpha^2=0$.
				\item The remaining subcase $[v_1,[v_1,f]]\in D_{k_1}$ yields
				\begin{align*}
					\begin{aligned}
						2\alpha^1\alpha^2[v_1,[v_2,f]]+(\alpha^2)^2[v_2,[v_2,f]]\overset{!}{\in}D_{k_1}\,.
					\end{aligned}
				\end{align*}
				Since $[v_1,[v_2,f]]$ and $[v_2,[v_2,f]]$ are by assumption not collinear$\Mod D_{k_1}$, the (up to a multiplicative factor) only non-trivial solution is $\alpha^1=1$ and $\alpha^2=0$.
			\end{enumerate}
			\paragraph*{Case 2:} Consider the case $\Dim{D_{k_1}+[D_{k_1-1},D_{k_1}]}=\Dim{D_{k_1}}+1$, \ie $H=D_{k_1}+[D_{k_1-1},D_{k_1}]=D_{k_1}+\Span{w_1}$, where $w_1$ is either $[v_1,[v_1,f]]$, $[v_1,[v_2,f]]$ or $[v_2,[v_2,f]]$. We distinguish between the following subcases:
			\begin{enumerate}
				\item $H=D_{k_1}+\Span{[v_1,[v_1,f]]}$ and thus $[v_1,[v_2,f]]=\kappa^1[v_1,[v_1,f]]\Mod D_{k_1}$ and $[v_2,[v_2,f]]=\kappa^2[v_1,[v_1,f]]\Mod D_{k_1}$, which inserted into \eqref{eq:algWithinProof} yields
				\begin{align*}
					\begin{aligned}
						\left((\alpha^1)^2+2\alpha^1\alpha^2\kappa^1+(\alpha^2)^2\kappa^2\right)[v_1,[v_1,f]]\overset{!}{\in}D_{k_1}\,.
					\end{aligned}
				\end{align*}
				The (up to a multiplicative factor) only non-trivial solutions are $\alpha^1=-\kappa^1\pm\sqrt{(\kappa^1)^2-\kappa^2}$ and $\alpha^2=1$. If subcase (1) is not applicable, \ie $[v_1,[v_1,f]]\in D_{k_1}$, we either have $[v_2,[v_2,f]]=\kappa [v_1,[v_2,f]]\Mod D_{k_1}$, or $[v_1,[v_1,f]],[v_1,[v_2,f]]\in D_{k_1}$. 
				\item The subcase $[v_1,[v_1,f]]\in D_{k_1}$ and $[v_2,[v_2,f]]=\kappa [v_1,[v_2,f]]\Mod D_{k_1}$ yields
				\begin{align*}
					\begin{aligned}
						\left(2\alpha^1\alpha^2+(\alpha^2)^2\kappa\right)[v_1,[v_2,f]]\overset{!}{\in}D_{k_1}\,.
					\end{aligned}
				\end{align*}
				The (up to a multiplicative factor) only non-trivial solutions are $\alpha^1=1$, $\alpha^2=0$ and $\alpha^1=\kappa$, $\alpha^2=-2$.
				\item The remaining subcase $[v_1,[v_1,f]],[v_1,[v_2,f]]\in D_{k_1}$ yields
				\begin{align*}
					\begin{aligned}
						(\alpha^2)^2[v_2,[v_2,f]]\overset{!}{\in}D_{k_1}\,,
					\end{aligned}
				\end{align*}
				and the (up to a multiplicative factor) only non-trivial solution is $\alpha^1=1$, $\alpha^2=0$.
			\end{enumerate}
			In conclusion, if $\Dim{D_{k_1}+[D_{k_1-1},D_{k_1}]}=\Dim{D_{k_1}}+3$, there exists no non-trivial solution. If $\Dim{D_{k_1}+[D_{k_1-1},D_{k_1}]}=\Dim{D_{k_1}}+2$, there either exists (up to a multiplicative factor) only one solution or no non-trivial solution. If $\Dim{D_{k_1}+[D_{k_1-1},D_{k_1}]}=\Dim{D_{k_1}}+1$, there exist (up to a multiplicative factor) at most two solutions or no non-trivial solutions (complex solutions are not relevant).
		\subsection{Proof of Theorem \ref{thm:linearizationByProlongations}}
			In \cite{GstottnerKolarSchoberl:2020a} the following result on the linearization of $(x,u)$-flat two-input systems has been shown (corresponding to Corollary 4 therein).
			\begin{lemma}\label{lem:linearizationXUflat}
				If the two-input system \eqref{eq:sys} possesses an $(x,u)$-flat output with a certain difference $d$, then it can be rendered static feedback linearizable by $d$-fold prolonging a suitably chosen (new) input after a suitable input transformation $\bar{u}=\Phi_u(x,u)$ has been applied.
			\end{lemma}
			To prove Theorem \ref{thm:linearizationByProlongations}, \ie to show that two-input systems with $d\leq 2$ can be rendered static feedback linearizable by $d$-fold prolonging a suitably chosen (new) input, we only have to show that $d\leq 2$ implies $(x,u)$-flatness. To be precise, we have to show that minimal flat outputs with $d\leq 2$ are $(x,u)$-flat outputs. Below we will show this by contradiction, \ie we will show that a flat output with a difference of $d\leq 2$ which explicitly depends on derivatives of the inputs is never a minimal flat output. For that we will utilize a certain relation between the state dimension $n$ of the system, the difference $d$ of a flat output and the "generalized" relative degrees of its components, which we derive in the following.\\
			
			Given a flat output $y$, recall that $r_j$ denotes the order of the highest derivative $r_j$ of $y^j$ needed to parameterize $x$ and $u$ by this flat. For a component of a flat output which does not explicitly depend on a derivative of the inputs, we define the relative degree $k_j$ by
			\begin{align}\label{eq:relDeg}
				\begin{aligned}
					\Lie_f^{k_j-1}\varphi^j&=\varphi^j_{k_j-1}(x)\,,&&&\Lie_f^{k_j}\varphi^j&=\varphi^j_{k_j}(x,u)\,.
				\end{aligned}
			\end{align}
			The following lemma provides a relation between the state dimension $n$, $r_j$ and $k_j$.
			\begin{lemma}\label{lem:nrk}
				For an $(x,u)$-flat output $y=\varphi(x,u)$ of a two input system of the form \eqref{eq:sys}, the relations $r_1=n-k_2$ and $r_2=n-k_1$ hold.
			\end{lemma}
			\begin{proof}
				See Equation (8) in \cite{GstottnerKolarSchoberl:2020b}, \ie $n-\#K=r_1-k_1$ and $n-\#K=r_2-k_2$ (where $\#K=k_1+k_2$), from which readily the relations $r_1=n-k_2$ and $r_2=n-k_1$ follow.
			\end{proof}
			With these two relations, we immediately obtain a relation between the difference $d=\#R-n$ of an $(x,u)$-flat output and the relative degrees $k_j$ of its components.
			\begin{corollary}\label{cor:d}
				For an $(x,u)$-flat output $y=\varphi(x,u)$ of a two-input system of the form \eqref{eq:sys}, the relation $d=n-k_1-k_2$ holds.
			\end{corollary}
			The definition of the relative degree via \eqref{eq:relDeg} requires that $\varphi^j=\varphi^j(x)$ or $\varphi^j=\varphi^j(x,u)$ and always yields $k_j\geq 0$. However, it turns out that Corollary \ref{cor:d} (and in fact also Lemma \ref{lem:nrk}) analogously apply if one or both components of the flat output explicitly depend on derivatives of the inputs, \ie $y^j=\varphi^j(x,u,u_1,\ldots,u_{q_j})$ with $\varphi^j$ explicitly depending on $u^1_{q_j}$ or $u^2_{q_j}$, by setting $k_j=-q_j$. In other words, Corollary \ref{cor:d} analogously applies if we interpret an explicit dependence of $y^j=\varphi^j(x,u,u_1,\ldots,u_{q_j})$ on the $q_j$-th derivative of an input as a negative relative degree of $k_j=-q_j$. This can be shown as follows. Let $y$ be a flat output of the system \eqref{eq:sys} and consider the prolonged system
			\begin{align}\label{eq:sysProlonged}
				\begin{aligned}
					\dot{x}&=f(x,u)&&&\dot{u}&=u_1&&&\dot{u}_1&=u_2&&&\ldots&&&&\dot{u}_{p-1}&=u_p\,,
				\end{aligned}
			\end{align}
			with the state $\tilde{x}=(x,u,u_1,\ldots,u_{p-1})$, $\Dim{\tilde{x}}=\tilde{n}=n+2p$ and the input $\tilde{u}=u_p$, obtained by prolonging both inputs of the system $p$ times. The flat parameterization $\tilde{x}=F_{\tilde{x}}(y_{[\tilde{R}-1]})$, $\tilde{u}=F_{\tilde{u}}(y_{[\tilde{R}]})$ of the prolonged system with respect to $y$ is easily obtained from the corresponding flat parameterization of the original system by successive differentiation (we thus have $\tilde{R}=R+p$) and it immediately follows that we have $\tilde{d}=\#\tilde{R}-\tilde{n}=\#R-n=d$, \ie such total prolongations preserve the difference of every flat output.
			
			Let us consider the case that one component of the flat output explicitly depends on derivatives of the inputs. Without loss of generality (swap the components of the flat output if necessary), we can assume that $y^1=\varphi^1(x,u)$ has a relative degree of $k_1\geq 0$ (we have $k_1=0$ if $\varphi^1$ depends explicitly on $u^1$ or $u^2$) and $y^2=\varphi^2(x,u,u_1,\ldots,u_{q_2})$ (with $\varphi^2$ explicitly depending on $u^1_{q_2}$ or $u^2_{q_2}$). By prolonging both inputs of the system $q_2$ times, we obtain the system
			\begin{align*}
				\begin{aligned}
					\dot{x}&=f(x,u)&&&\dot{u}&=u_1&&&\dot{u}_1&=u_2&&&&\ldots&&&\dot{u}_{q_2-1}&=u_{q_2}\,,
				\end{aligned}
			\end{align*}
			with the state $\tilde{x}=(x,u,u_1,\ldots,u_{q_2-1})$, $\Dim{\tilde{x}}=\tilde{n}=n+2q_2$ and the input $\tilde{u}=u_{q_2}$. For this prolonged system, the components of the flat output $y=(\varphi^1(x,u),\varphi^2(x,u,u_1,\ldots,u_{q_2}))$ only depend on the state and the inputs and thus, Corollary \ref{cor:d} directly applies. By construction, we have $\tilde{k}_2=0$, since $\varphi^2$ explicitly depends on $\tilde{u}$. The $k_1$-th derivative of $\varphi^1$ explicitly depends on $u$ and therefore, we have to differentiate it another $q_2$ times until it explicitly depends on $\tilde{u}$ and thus, $\tilde{k}_1=k_1+q_2$. According to Corollary \ref{cor:d}, we thus have $\tilde{d}=\tilde{n}-\tilde{k}_1-\tilde{k}_2=n+2q_2-(k_1+q_2)-0=n-k_1+q_2$. Above we noticed that $d=\tilde{d}$, \ie total prolongations preserve the difference and thus, $d=n-k_1+q_2$, \ie exactly the same as Corollary \ref{cor:d} would yield when we set $k_2=-q_2$.
			
			The case that both components of the flat output explicitly depend on derivatives of the inputs, \ie $y^j=\varphi^j(x,u,u_1,\ldots,u_{q_j})$ (with $\varphi^j$ explicitly depending on $u^1_{q_j}$ or $u^2_{q_j}$), can be handled analogously and yields $d=n+q_1+q_2$, \ie exactly the same as Corollary \ref{cor:d} would yield when we set $k_1=-q_1$ and $k_2=-q_2$.\\
%			The case that both components of the flat output explicitly depend on derivatives of the inputs, \ie $y^j=\varphi^j(x,u,u_1,\ldots,u_{q_j})$ (with $\varphi^j$ explicitly depending on $u^1_{q_j}$ or $u^2_{q_j}$), can be handled analogously. Without loss of generality (swap the components of the flat output if necessary), we can assume that $q_2\geq q_1$. For the system obtained by $q_2$-fold prolonging both inputs, the components of the flat output $y=(\varphi^1(x,u,u_1,\ldots,u_{q_1}),\varphi^2(x,u,u_1,\ldots,u_{q_2}))$ only depend on the state and the inputs and thus, Corollary \ref{cor:d} directly applies. By construction, we have $\tilde{k}_2=0$, since $\varphi^2$ explicitly depends on $\tilde{u}$. The function $\varphi^1$ explicitly depends on $u_{q_1}$ and therefore, we have to differentiate it another $q_2-q_1$ times until it explicitly depends on $\tilde{u}$ and thus, $\tilde{k}_1=q_2-q_1$. According to Corollary \ref{cor:d}, we thus have $\tilde{d}=\tilde{n}-\tilde{k}_1-\tilde{k}_2=n+2q_2-(q_2-q_1)-0=n+q_1+q_2$. Above we noticed that $d=\tilde{d}$ and thus, $d=n+q_1+q_2$, \ie exactly the same as in Corollary \ref{cor:d} when we set $k_1=-q_1$ and $k_2=-q_2$.\\
			
			With these preliminary results at hand, we can now proof that minimal flat outputs with $d\leq 2$ of a two-input system of the form \eqref{eq:sys} are actually $(x,u)$-flat outputs. Let $y=(\varphi^1(x,u,u_1,\ldots,u_{q_1}),\varphi^2(x,u,u_1,\ldots,u_{q_2}))$ be a minimal flat output of \eqref{eq:sys} with $d\leq 2$. In the following we show that $q_1,q_2\leq 0$ by contradiction. Without loss of generality, we can assume that $q_1\leq q_2$ (swap the components of the flat output if necessary). Assume that $q_2\geq 1$. According to the above discussed generalization of Corollary \ref{cor:d}, we have $d=n+q_1+q_2$ and thus $n+q_1+q_2\leq 2$. Because of $n+q_1+q_2\leq 2$, we have $q_1\leq 2-n-q_2$. We can assume $n\geq 3$, since for $n=2$, the system would be static feedback linearizable with the state of the system forming a linearizing output due to the assumption $\Rank{\partial_uf}=2$, which is in contradiction to the minimality of $y$, for $n=1$, the rank condition $\Rank{\partial_uf}=2$ could not hold. Thus, $n\geq 3$ and in turn, the first component of the flat output has a (positive) relative degree of $k_1=-q_1\geq n+q_2-2$. This enables us to replace $k_1$ states of the system by $\varphi^1$ and its first $k_1-1$ derivatives by applying the state transformation
			\begin{align}\label{eq:stateTrans}
				\begin{aligned}
					\bar{x}_1^{i_1}&=\Lie_f^{i_1-1}\varphi^1(x)\,,&i_1&=1,\ldots,k_1\\
					\bar{x}_2^{i_2}&=\psi^{i_2}(x)\,,&i_2&=1,\ldots,n-k_1\,,
				\end{aligned}
			\end{align}
			where $\psi^{i_2}$ are arbitrary functions of the state, chosen such that \eqref{eq:stateTrans} is a regular state transformation. By additionally applying the input transformation $\bar{u}^1=\Lie_f^{k_1}\varphi^1$, $\bar{u}^2=g(x,u)$, with $g$ chosen such that this transformation is invertible with respect to $u$, we obtain
			\begin{align}\label{eq:statesReplaced}
				\begin{aligned}
					\dot{\bar{x}}_1^1&=\bar{x}_1^2\\
					\dot{\bar{x}}_1^2&=\bar{x}_1^3\\
					&\vdotswithin{=}\\
					\dot{\bar{x}}_1^{k_1}&=\bar{u}^1\\
					\dot{\bar{x}}_2^{i_2}&=\bar{f}_2^{i_2}(\bar{x}_1,\bar{x}_2,\bar{u}^1,\bar{u}^2)\,,&i_2&=1,\ldots,n-k_1\,.
				\end{aligned}
			\end{align}
			Recall that we have $k_1\geq n+q_2-2$ and thus $\Dim{\bar{x}_2}=n-k_1\leq 2-q_2$. Due to the assumption $q_2\geq 1$, we thus have $\Dim{\bar{x}_2}\leq 1$. However, for $\Dim{\bar{x}_2}=1$, \eqref{eq:statesReplaced} would read
			\begin{align*}
				\begin{aligned}
					\dot{\bar{x}}_1^1&=\bar{x}_1^2\\
					\dot{\bar{x}}_1^2&=\bar{x}_1^3\\
					&\vdotswithin{=}\\
					\dot{\bar{x}}_1^{n-1}&=\bar{u}^1\\
					\dot{\bar{x}}_2^{1}&=\bar{f}_2^{1}(\bar{x}_1,\bar{x}_2,\bar{u}^1,\bar{u}^2)\,,
				\end{aligned}
			\end{align*}
			so the system would be static feedback linearizable with $(\bar{x}_1^1,\bar{x}_2^1)$ forming a linearizing output, which is in contradiction to the minimality of the flat output $y$. For $\Dim{\bar{x}_2}=0$, the system would consist of a single integrator chain of length $n$, which actually contradicts with $\Rank{\partial_uf}=2$. We thus have $q_2\leq 0$ and because of $q_1\leq q_2$, also $q_1\leq 0$. In conclusion, every minimal flat output with a difference of $d\leq 2$ is an $(x,u)$-flat output. Lemma \ref{lem:linearizationXUflat} therefore guarantees that a system with $d\leq 2$ can be rendered static feedback linearizable by $d$-fold prolonging a suitable chosen (new) input after a suitable input transformation $\bar{u}=\Phi_u(x,u)$ has been applied, which completes the proof.
		\subsection{The Special Case $k_1=1$ with $\Delta_i$ of the Form \eqref{eq:d2twoFoldProlongedDistributionsSimplifiedCase2}}
			By assumption we have $k_1=1$ and $\Ad_{\bar{f}}^2\partial_{\bar{u}^2}\in\Span{\partial_{\bar{u}^1_2},\partial_{\bar{u}^1_1},\partial_{\bar{u}^1},\partial_{\bar{u}^2},[\bar{f},\partial_{\bar{u}^2}]}$, \ie the distribution $D_1$ is non-involutive and the involutive distributions $\Delta_i$ are of the form \eqref{eq:d2twoFoldProlongedDistributionsSimplifiedCase2}. It follows that we necessarily have $D_0\not\subset\C{D_1}$. Indeed, $D_0\subset\C{D_1}$ would imply $[\partial_{\bar{u}^1},[\bar{f},\partial_{\bar{u}^2}]]\in D_1$ and in turn either $\Delta_2\cap\TXU=D_1$, which contradicts with $D_1$ being non-involutive, or $\Delta_2=\Span{\partial_{\bar{u}^1_2},\partial_{\bar{u}^1_1},\partial_{\bar{u}^1},\partial_{\bar{u}^2},[\bar{f},\partial_{\bar{u}^2}]}$ and then, due to $\Ad_{\bar{f}}^2\partial_{\bar{u}^2}\in\Span{\partial_{\bar{u}^1_2},\partial_{\bar{u}^1_1},\partial_{\bar{u}^1},\partial_{\bar{u}^2},[\bar{f},\partial_{\bar{u}^2}]}$, $\Delta_3\cap\TXU=D_1$ would follow, which again contradicts with $D_1$ being non-involutive. Thus, $[\partial_{\bar{u}^1},[\bar{f},\partial_{\bar{u}^2}]]\notin D_1$ and therefore $D_0\not\subset\C{D_1}$. 
			
			Thus, we have to show that the system necessarily meets item \ref{d2:2b}, \ie that there exists a non-vanishing vector field $v_c\in D_0$ such that $v_c\in\C{E_1}$ where $E_1$ is defined as $E_1=D_0+\Span{[v_c,\bar{f}]}$. Let us show that the vector field $v_c=\partial_{\bar{u}^2}$ meets this criterion. The vector field $v_c=\partial_{\bar{u}^2}$ meets $\partial_{\bar{u}^2}\in D_0$ and yields $E_1=D_0+\Span{[\bar{f},\partial_{\bar{u}^2}]}=\Span{\partial_{\bar{u}^1},\partial_{\bar{u}^2},[\bar{f},\partial_{\bar{u}^2}]}$. The involutivity of $\Delta_1$ implies that $\partial_{\bar{u}^2}\in\C{E_1}$. Therefore, $v_c=\partial_{\bar{u}^2}$ and $E_1=\Span{\partial_{\bar{u}^1},\partial_{\bar{u}^2},[\bar{f},\partial_{\bar{u}^2}]}$. Above, we have deduced that $[\partial_{\bar{u}^1},[\bar{f},\partial_{\bar{u}^2}]]\notin D_1$, which because of $E_1\subset D_1$ implies that $[\partial_{\bar{u}^1},[\bar{f},\partial_{\bar{u}^2}]]\notin E_1$ and thus, $E_1$ is non-involutive, so we have to show next that item \ref{d2:3a} is necessarily met. 
			
			We have $E_1\subset\Delta_2\cap\TXU$ (see \eqref{eq:d2twoFoldProlongedDistributionsSimplifiedCase2}), from which the involutive closure of $E_1$ follows as $\overline{E}_1=\Delta_2\cap\TXU=\Span{\partial_{\bar{u}^1},\partial_{\bar{u}^2},[\bar{f},\partial_{\bar{u}^2}],[\partial_{\bar{u}^1},[\bar{f},\partial_{\bar{u}^2}]]}$ and thus \ref{d2:3a}\ref{d2:3a1} is met. Note that we necessarily have $\overline{E}_1\neq\TXU$. Indeed, since $\Dim{E_1}=3$ and $\Dim{\overline{E}_1}=4$, $\overline{E}_1=\TXU$ would imply $n=2$ and thus $D_1=\TXU$, which is in contradiction with $D_{k_1}$ being non-involutive. Therefore, we necessarily have $\overline{E}_1\neq\TXU$, based on which we next show that necessarily \ref{d2:3a}\ref{d2:3a2} is met, \ie that necessarily $\Dim{[\bar{f},E_1]+\overline{E}_1}=\Dim{\overline{E}_1}+1$. Since $\Ad_{\bar{f}}^2\partial_{\bar{u}^2}\in\Delta_2\cap\TXU=\overline{E}_1$, the condition holds if and only if $[\bar{f},\partial_{\bar{u}^1}]\notin\overline{E}_1$, which can be shown by contradiction. Assume that $[\bar{f},\partial_{\bar{u}^1}]\in\overline{E}_1$ and thus $[\bar{f},E_1]+\overline{E}_1=\overline{E}_1$. Based on the Jacobi identity, it can then be shown that this would imply $[\bar{f},\overline{E}_1]\subset\overline{E}_1$. However, since $\overline{E}_1\neq\TXU$, this would in turn imply that the system contains an autonomous subsystem, which is in contradiction with the assumption that the system is flat.
			
			With $F_{2}=\overline{E}_1=\Delta_2\cap\TXU$, it follows that the distributions $F_i=F_{i-1}+[\bar{f},F_{i-1}]$, $i\geq 3$ constructed in item \ref{d2:5} and the distributions $\Delta_i$ are related via $\Delta_i=\Span{\partial_{\bar{u}^1_2},\partial_{\bar{u}^1_1}}+F_i$ and thus
			\begin{align*}
				\begin{aligned}
					F_3&=\Delta_3\cap\TXU\\
					&\vdotswithin{=}\\
					F_s&=\Delta_s\cap\TXU=\TXU\,.
				\end{aligned}
			\end{align*}
			So these distributions are indeed involutive and there indeed exists an integer $s$ such that $E_s=\TXU$, which completes the necessity part of the proof for the special case $k_1=1$ with $\Ad_{\bar{f}}^2\partial_{\bar{u}^2}\in\Span{\partial_{\bar{u}^1_2},\partial_{\bar{u}^1_1},\partial_{\bar{u}^1},\partial_{\bar{u}^2},[\bar{f},\partial_{\bar{u}^2}]}$.
	\section{Proofs of Lemmas}\label{ap:lemmas}
		\subsection{Proof of Lemma \ref{lem:d2SimplificationOfDistributions}}
			We have $[f_p,\partial_{\bar{u}^j}]=-\partial_{\bar{u}^j}\bar{f}^i(x,\bar{u})\partial_{x^i}=[\bar{f},\partial_{\bar{u}^j}]$ with $\bar{f}=\bar{f}^i(x,\bar{u})\partial_{x^i}$ and thus $\Delta_1=\Span{\partial_{\bar{u}^1_2},\partial_{\bar{u}^1_1},\partial_{\bar{u}^2},[\bar{f},\partial_{\bar{u}^2}]}$ and in turn
			\begin{align*}
				\begin{aligned}
					\Delta_2&=\Span{\partial_{\bar{u}^1_2},\partial_{\bar{u}^1_1},\partial_{\bar{u}^1},\partial_{\bar{u}^2},[\bar{f},\partial_{\bar{u}^2}],[f_p,[\bar{f},\partial_{\bar{u}^2}]]}\,.
				\end{aligned}
			\end{align*}
			We have $[f_p,[\bar{f},\partial_{\bar{u}^2}]]=\Ad_{\bar{f}}^2\partial_{\bar{u}^2}+\bar{u}^1_1[\partial_{\bar{u}^1},[\bar{f},\partial_{\bar{u}^2}]]+\bar{u}^1_2[\partial_{\bar{u}^1_1},[\bar{f},\partial_{\bar{u}^2}]]=\Ad_{\bar{f}}^2\partial_{\bar{u}^2}+\bar{u}^1_1[\partial_{\bar{u}^1},[\bar{f},\partial_{\bar{u}^2}]]\Mod\Delta_1$. The involutivity of $\Delta_2$ implies that $\Ad_{\bar{f}}^2\partial_{\bar{u}^2}$ and $[\partial_{\bar{u}^1},[\bar{f},\partial_{\bar{u}^2}]]$ are collinear modulo $\Span{\partial_{\bar{u}^1_2},\partial_{\bar{u}^1_1},\partial_{\bar{u}^1},\partial_{\bar{u}^2},[\bar{f},\partial_{\bar{u}^2}]}$. In case of $k_1\geq 2$, we certainly have $\Ad_{\bar{f}}^2\partial_{\bar{u}^2}\notin\Span{\partial_{\bar{u}^1_2},\partial_{\bar{u}^1_1},\partial_{\bar{u}^1},\partial_{\bar{u}^2},[\bar{f},\partial_{\bar{u}^2}]}$. Indeed, assume that $k_1\geq 2$ and $\Ad_{\bar{f}}^2\partial_{\bar{u}^2}\in\Span{\partial_{\bar{u}^1_2},\partial_{\bar{u}^1_1},\partial_{\bar{u}^1},\partial_{\bar{u}^2},[\bar{f},\partial_{\bar{u}^2}]}$. We then have
			\begin{align*}
				\Delta_2&=\Span{\partial_{\bar{u}^1_2},\partial_{\bar{u}^1_1},\partial_{\bar{u}^1},\partial_{\bar{u}^2},[\bar{f},\partial_{\bar{u}^2}],\underbrace{[\partial_{\bar{u}^1},[\bar{f},\partial_{\bar{u}^2}]]}_{\in D_1}}
			\end{align*}
			(where $[\partial_{\bar{u}^1},[\bar{f},\partial_{\bar{u}^2}]]\in D_1$ due to the involutivity of $D_1$). Thus, either $\Delta_2=\Span{\partial_{\bar{u}^1_2},\partial_{\bar{u}^1_1}}+D_1$ or $\Delta_2=\Span{\partial_{\bar{u}^1_2},\partial_{\bar{u}^1_1}}+\Span{\partial_{\bar{u}^1},\partial_{\bar{u}^2},[\bar{f},\partial_{\bar{u}^2}]}$. It follows that the first case would lead to $\Delta_i=\Span{\partial_{\bar{u}^1_2},\partial_{\bar{u}^1_1}}+D_{i-1}$ which for $i=k_1+1$ would imply that $D_{k_1}$ is involutive, contradicting with $D_{k_1}$ being non-involutive. The second case, \ie $\Delta_2=\Span{\partial_{\bar{u}^1_2},\partial_{\bar{u}^1_1}}+\Span{\partial_{\bar{u}^1},\partial_{\bar{u}^2},[\bar{f},\partial_{\bar{u}^2}]}$ would lead to $\Delta_3=\Span{\partial_{\bar{u}^1_2},\partial_{\bar{u}^1_1}}+D_1$ and in turn $\Delta_i=\Span{\partial_{\bar{u}^1_2},\partial_{\bar{u}^1_1}}+D_{i-2}$ which for $i=k_1+2$ would imply that $D_{k_1}$ is involutive, again contradicting with $D_{k_1}$ being non-involutive. Thus, in case of $k_1\geq 2$, we always have $\Ad_{\bar{f}}^2\partial_{\bar{u}^2}\notin\Span{\partial_{\bar{u}^1_2},\partial_{\bar{u}^1_1},\partial_{\bar{u}^1},\partial_{\bar{u}^2},[\bar{f},\partial_{\bar{u}^2}]}$, and thus
			\begin{align}\label{eq:delta2case1}
				\begin{aligned}
					\Delta_2&=\Span{\partial_{\bar{u}^1_2},\partial_{\bar{u}^1_1},\partial_{\bar{u}^1},\partial_{\bar{u}^2},[\bar{f},\partial_{\bar{u}^2}],\Ad_{\bar{f}}^2\partial_{\bar{u}^2}}\,.
				\end{aligned}
			\end{align}
			The form \eqref{eq:d2twoFoldProlongedDistributionsSimplifiedCase1} then easily follows from the involutivity of the distributions $\Delta_2,\Delta_3,\ldots$ and the fact that $f_p$ and $\bar{f}$ coincide modulo $\Span{\partial_{\bar{u}^1_2},\partial_{\bar{u}^1_1}}\subset\Delta_2,\Delta_3,\ldots$
		
			However, for $k_1=1$, it can indeed happen that $\Ad_{\bar{f}}^2\partial_{\bar{u}^2}\in\Span{\partial_{\bar{u}^1_2},\partial_{\bar{u}^1_1},\partial_{\bar{u}^1},\partial_{\bar{u}^2},[\bar{f},\partial_{\bar{u}^2}]}$ and thus
			\begin{align}\label{eq:delta2case2}
				\begin{aligned}
					\Delta_2&=\Span{\partial_{\bar{u}^1_2},\partial_{\bar{u}^1_1},\partial_{\bar{u}^1},\partial_{\bar{u}^2},[\bar{f},\partial_{\bar{u}^2}],[\partial_{\bar{u}^1},[\bar{f},\partial_{\bar{u}^2}]]}\,
				\end{aligned}
			\end{align}
			(in which case we necessarily have $[\partial_{\bar{u}^1},[\bar{f},\partial_{\bar{u}^2}]]\notin D_1$). In this case, due to the involutivity of the distributions $\Delta_2,\Delta_3,\ldots$ and the fact that $f_p$ and $\bar{f}$ coincide modulo $\Span{\partial_{\bar{u}^1_2},\partial_{\bar{u}^1_1}}\subset\Delta_2,\Delta_3,\ldots$ the form \eqref{eq:d2twoFoldProlongedDistributionsSimplifiedCase2} follows.
		\subsection{Proof of Lemma \ref{lem:d2case1FirstDerivedFlag}}
			Under the assumption $d=2$ with $D_{k_1-1}\subset\C{D_{k_1}}$, the distribution $\mathcal{I}_{k_1}=\Delta_{k_1+1}\cap D_{k_1}$ is involutive, which can be shown by contradiction. Assume that $\mathcal{I}_{k_1}=\Delta_{k_1+1}\cap D_{k_1}$ is non-involutive. From \eqref{eq:d2twoFoldProlongedDistributionsSimplifiedCase1}, it follows that we have\footnote{It is immediate that $\Ad_{\bar{f}}^{k_1+1}\notin D_{k_1}$, as it would either lead to $\Delta_{k_1+1}\cap\TXU=D_{k_1}$ or $\Delta_{k_1+2}\cap\TXU=D_{k_1}$, which contradicts with $D_{k_1}$ being non-involutive.}
			\begin{align*}
				\begin{aligned}
					\mathcal{I}_{k_1}=\Delta_{k_1+1}\cap D_{k_1}=D_{k_1-1}+\Span{\Ad_{\bar{f}}^{k_1}\partial_{\bar{u}^2}}\,.
				\end{aligned}
			\end{align*}
			Since $\mathcal{I}_{k_1}\subset\Delta_{k_1+1}\cap\TXU$ and
			\begin{align*}
				\begin{aligned}
					\Delta_{k_1+1}\cap\TXU&=D_{k_1-1}+\Span{\Ad_{\bar{f}}^{k_1}\partial_{\bar{u}^2},\Ad_{\bar{f}}^{k_1+1}\partial_{\bar{u}^2}}\,,
				\end{aligned}
			\end{align*}
			is involutive, we then have $\overline{\mathcal{I}}_{k_1}=\Delta_{k_1+1}\cap\TXU$, \ie $\mathcal{I}_{k_1}\underset{1}{\subset}\overline{\mathcal{I}}_{k_1}$. We have $\mathcal{I}_{k_1}\underset{1}{\subset}D_{k_1}$ and since $D_{k_1-1}\subset\C{D_{k_1}}$, we have $[D_{k_1-1},\mathcal{I}_{k_1}]\subset D_{k_1}$. However, since $\mathcal{I}_{k_1}$ is assumed to be non-involutive, we necessarily have $[D_{k_1-1},\mathcal{I}_{k_1}]\not\subset\mathcal{I}_{k_1}$ and thus, the direction which completes $\mathcal{I}_{k_1}$ to its involutive closure is contained in $D_{k_1}$, which would lead to $\overline{\mathcal{I}}_{k_1}=D_{k_1}$ and contradict with $D_{k_1}$ being non-involutive. Thus, $\mathcal{I}_{k_1}$ is indeed involutive.\\
			
			From the Jacobi identity
			\begin{align*}
				\begin{aligned}
					\underbrace{[\Ad_{\bar{f}}^{k_1-1}\partial_{\bar{u}^1},[\bar{f},\Ad_{\bar{f}}^{k_1}\partial_{\bar{u}^2}]]}_{\in\Delta_{k_1+1}\cap\TXU\subset D_{k_1}+\Span{\Ad_{\bar{f}}^{k_1+1}\partial_{\bar{u}^2}}}+\underbrace{[\Ad_{\bar{f}}^{k_1}\partial_{\bar{u}^2},[\Ad_{\bar{f}}^{k_1-1}\partial_{\bar{u}^1},\bar{f}]]}_{\in D^{(1)}_{k_1}}+\underbrace{[\bar{f},\overbrace{[\Ad_{\bar{f}}^{k_1}\partial_{\bar{u}^2},\Ad_{\bar{f}}^{k_1-1}\partial_{\bar{u}^1}]}^{\in\mathcal{I}_{k_1}}]}_{\in D_{k_1}+\Span{\Ad_{\bar{f}}^{k_1+1}\partial_{\bar{u}^2}}}&=0\,,
				\end{aligned}
			\end{align*}
			where we have used the involutivity of $\mathcal{I}_{k_1}$ for the third term, it follows that the vector field $[\Ad_{\bar{f}}^{k_1}\partial_{\bar{u}^1},\Ad_{\bar{f}}^{k_1}\partial_{\bar{u}^2}]$, which completes $D_{k_1}$ to $D_{k_1}^{(1)}$, is contained in $D_{k_1}+\Span{\Ad_{\bar{f}}^{k_1+1}\partial_{\bar{u}^2}}$, and thus $D^{(1)}_{k_1}=D_{k_1}+\Span{\Ad_{\bar{f}}^{k_1+1}\partial_{\bar{u}^2}}$, as desired.
		\subsection{Proof of Lemma \ref{lem:cuachyEk2}}
			The vector field $\Ad_{\bar{f}}^{k_2-1}\partial_{\bar{u}^2}\in E_{k_2-1}$ is a Cauchy characteristic vector field of $E_{k_2}$. Indeed, we have (see \eqref{eq:E} and \eqref{eq:DeltaErelation})
			\begin{align}\label{eq:Ek2}
				\begin{aligned}
					E_{k_2}&=\underbrace{E_{k_2-1}+\Span{\Ad_{\bar{f}}^{k_2}\partial_{\bar{u}^2}}}_{\Delta_{k_2}\cap\TXU}+\Span{\Ad_{\bar{f}}^{k_2-1}\partial_{\bar{u}^1}}\,.
				\end{aligned}
			\end{align}
			From the Jacobi identity
			\begin{align*}
				\begin{aligned}
					[\Ad_{\bar{f}}^{k_2-1}\partial_{\bar{u}^2},[\bar{f},\Ad_{\bar{f}}^{k_2-2}\partial_{\bar{u}^1}]]+\underbrace{[\Ad_{\bar{f}}^{k_2-2}\partial_{\bar{u}^1},[\Ad_{\bar{f}}^{k_2-1}\partial_{\bar{u}^2},\bar{f}]]}_{\in\Delta_{k_2}\cap\TXU\subset E_{k_2}}+\underbrace{[\bar{f},\overbrace{[\Ad_{\bar{f}}^{k_2-2}\partial_{\bar{u}^1},\Ad_{\bar{f}}^{k_2-1}\partial_{\bar{u}^2}]}^{\in E_{k_2-1}}]}_{\in E_{k_2}}&=0\,,
				\end{aligned}
			\end{align*}
			it follows that $[\Ad_{\bar{f}}^{k_2-1}\partial_{\bar{u}^2},\Ad_{\bar{f}}^{k_2-1}\partial_{\bar{u}^1}]\in E_{k_2}$, where we have used the fact that $[\bar{f},E_{k_2-1}]\subset E_{k_2}$ and that $\Delta_{k_2}\cap\TXU$ and $E_{k_2-1}$ are involutive. The involutivity of $\Delta_{k_2}\cap\TXU$ together with $[\Ad_{\bar{f}}^{k_2-1}\partial_{\bar{u}^2},\Ad_{\bar{f}}^{k_2-1}\partial_{\bar{u}^1}]\in E_{k_2}$ implies that $\Ad_{\bar{f}}^{k_2-1}\partial_{\bar{u}^2}\in\C{E_{k_2}}$.\\
			
			The distribution $E_{k_2}$ is by assumption non-involutive. Since $E_{k_2}\subset\Delta_{k_2+1}\cap\TXU$ and $\Delta_{k_2+1}\cap\TXU$ is involutive, it follows that $\overline{E}_{k_2}=\Delta_{k_2+1}\cap\TXU=E_{k_2}+\Span{\Ad_{\bar{f}}^{k_2+1}\partial_{\bar{u}^2}}$.\\

			In case of $k_2\geq k_1+2$, it immediately follows from the Jacobi identity and the way the distributions $E_i$ are constructed that $E_{k_2-2}\subset\C{E_{k_2}}$. We have $E_{k_2-1}=E_{k_2-2}+\Span{\Ad_{\bar{f}}^{k_2-2}\partial_{\bar{u}^1},\Ad_{\bar{f}}^{k_2-1}\partial_{\bar{u}^2}}$ and we have just shown that $\Ad_{\bar{f}}^{k_2-1}\partial_{\bar{u}^2}\in\C{E_{k_2}}$. Since by assumption $E_{k_2-1}\not\subset\C{E_{k_2}}$, it follows that $\Ad_{\bar{f}}^{k_2-2}\partial_{\bar{u}^1}\notin\C{E_{k_2}}$. Since $\Ad_{\bar{f}}^{k_2-2}\partial_{\bar{u}^1}$ is contained in the involutive subdistribution $\Delta_{k_2}\cap\TXU\underset{1}{\subset}E_{k_2}$ (see \eqref{eq:Ek2}), it follows that the vector field $[\Ad_{\bar{f}}^{k_2-2}\partial_{\bar{u}^1},\Ad_{\bar{f}}^{k_2-1}\partial_{\bar{u}^1}]$ cannot be contained in $E_{k_2}$ and thus, $\overline{E}_{k_2}=E_{k_2}+\Span{[\Ad_{\bar{f}}^{k_2-2}\partial_{\bar{u}^1},\Ad_{\bar{f}}^{k_2-1}\partial_{\bar{u}^1}]}$, which in turn implies that the vector fields $[\Ad_{\bar{f}}^{k_2}\partial_{\bar{u}^2},\Ad_{\bar{f}}^{k_2-1}\partial_{\bar{u}^1}]$ and $[\Ad_{\bar{f}}^{k_2-2}\partial_{\bar{u}^1},\Ad_{\bar{f}}^{k_2-1}\partial_{\bar{u}^1}]$ are collinear modulo $E_{k_2}$, \ie $[\Ad_{\bar{f}}^{k_2}\partial_{\bar{u}^2},\Ad_{\bar{f}}^{k_2-1}\partial_{\bar{u}^1}]=\lambda [\Ad_{\bar{f}}^{k_2-2}\partial_{\bar{u}^1},\Ad_{\bar{f}}^{k_2-1}\partial_{\bar{u}^1}]\Mod E_{k_2}$. From the latter, it follows that $[\Ad_{\bar{f}}^{k_2}\partial_{\bar{u}^2}-\lambda\Ad_{\bar{f}}^{k_2-2}\partial_{\bar{u}^1},\Ad_{\bar{f}}^{k_2-1}\partial_{\bar{u}^1}]\in E_{k_2}$, which because of $\Ad_{\bar{f}}^{k_2}\partial_{\bar{u}^2}-\lambda\Ad_{\bar{f}}^{k_2-2}\partial_{\bar{u}^1}\in\Delta_{k_2}\cap\TXU$ implies that $\Ad_{\bar{f}}^{k_2}\partial_{\bar{u}^2}-\lambda\Ad_{\bar{f}}^{k_2-2}\partial_{\bar{u}^1}\in\C{E_{k_2}}$. It follows that $\C{E_{k_2}}=E_{k_2-2}+\Span{\Ad_{\bar{f}}^{k_2-1}\partial_{\bar{u}^2},\Ad_{\bar{f}}^{k_2}\partial_{\bar{u}^2}-\lambda\Ad_{\bar{f}}^{k_2-2}\partial_{\bar{u}^1}}\underset{2}{\subset}E_{k_2}$, \ie further Cauchy characteristic vector fields which are linearly independent of those we have already found cannot exist, otherwise $E_{k_2}$ would be involutive.\\
			
			In case of $k_2=k_1+1$, the distribution $E_{k_2-2}$ does not exist. However, we have the involutive distribution $\Delta_{k_1}\cap D_{k_1-1}=D_{k_1-2}+\Span{\Ad_{\bar{f}}^{k_1-1}\partial_{\bar{u}^2}}$. Recall that we have $E_{k_1}=D_{k_1-1}+\Span{\Ad_{\bar{f}}^{k_1}\partial_{\bar{u}^2}}$. Thus, it follows that $[\bar{f},\Delta_{k_1}\cap D_{k_1-1}]\subset E_{k_1}$. Based on the latter, by applying the Jacobi identity%$[v,[w,f]]+[f,[v,w]]+[w,[f,v]]=0$ where $v\in D_{k_1-1}\cap\Delta_{k_1}$ and $w\in E_{k_1}$
			, it can be shown that $\Delta_{k_1}\cap D_{k_1-1}\subset\C{E_{k_2}}$. Repeating the above proof with $E_{k_2-2}$ replaced by $D_{k_1-1}\cap\Delta_{k_1}$ gives the desired result that in case of $k_2=k_1+1$ we have $\C{E_{k_2}}=D_{k_1-1}\cap\Delta_{k_1}+\Span{\Ad_{\bar{f}}^{k_2-1}\partial_{\bar{u}^2},\Ad_{\bar{f}}^{k_2}\partial_{\bar{u}^2}-\lambda\Ad_{\bar{f}}^{k_2-2}\partial_{\bar{u}^1}}$\,.
		\subsection{Proof of Lemma \ref{lem:missingDistributions}}
			By assumption, we have $G_{k-1}\subset\C{G_k}$, which actually implies $G_{k-1}=\C{G_k}$. Indeed, since $G_k$ is non-involutive, we necessarily have $\Dim{\C{G_k}}\leq\Dim{G_k}-2$. Because of $G_{k-1}\subset\C{G_k}$ and $G_{k-1}\underset{2}{\subset}G_k$, we thus have $G_{k-1}=\C{G_k}$. Apply a change of coordinates $(z_3,z_2,z_1)=\Phi(x,u)$ such that $G_{k-1}$ and $\overline{G}_k$ get straightened out simultaneously, \ie such that in the new coordinates we have 
			\begin{align*}
				G_{k-1}&=\Span{\partial_{z_1}}\,,\\
				\overline{G}_k&=\Span{\partial_{z_1},\partial_{z_2}}\,.
			\end{align*}
			Since $G_{k-1}=\C{G_k}$, in these coordinates, there exists a basis for $G_k$ of the form $G_k=G_{k-1}+\Span{v_1,v_2}$ with the two vector fields $v_j$ being of the form $v_j=v_j^{i_2}(z_3,z_2)\partial_{z_2^{i_2}}$, $i_2=1,\ldots,\Dim{z_2}$ (where $\Dim{z_2}=3$).\\
			
			In case of $\overline{G}_k=\TXU$ (we have  $\Dim{z_3}=0$ and thus simply $v_j=v_j^{i_2}(\bar{x}_2)\partial_{\bar{x}_2^{i_2}}$ in this case), choose any function $\psi(z_2)$, $\D\psi\neq 0$ and define the vector field $v_c=(\D\psi\rfloor v_2)v_1-(\D\psi\rfloor v_1)v_2$, which is non-zero due to the non-involutivity of $G_k$ (indeed, $\D\psi\rfloor v_1=\D\psi\rfloor v_2\equiv 0$ would yield $G_k^\perp=\Span{\D\psi}$ which contradicts with $G_k$ being non-involutive). The distribution $H_k=G_{k-1}+\Span{v_c}$ is obviously involutive and we by construction have $G_{k-1}\underset{1}{\subset}H_k\underset{1}{\subset}G_k$. Note that there exist infinitely many valid choices for the function $\psi$ and different choices yield in general different distributions $H_k$.\footnote{Note that there is no need to explicitly construct the vector fields $v_j=v_j^{i_2}(\bar{x}_2)\partial_{\bar{x}_2^{i_2}}$, $j=1,2$ for deriving the distribution $H_k$ corresponding to a certain choice $\psi$. Indeed, any vector fields $w_1,w_2$ which complete $G_{k-1}$ to $G_k$ can be written as a linear combination $w_j=\beta_j^1v_1+\beta_j^2v_2\Mod G_{k-1}$ with some functions $\beta_j^i=\beta_j^i(z_2,z_1)$ and $\beta_1^1\beta_2^2-\beta_1^2\beta_2^1\neq 0$. A straight forward calculation shows that $(\D\psi\rfloor w_2)w_1-(\D\psi\rfloor w_1)w_2=(\beta_1^1\beta_2^2-\beta_1^2\beta_2^1)((\D\psi\rfloor v_2)v_1-(\D\psi\rfloor v_1)v_2)\Mod G_{k-1}$ and thus, $\tilde{H}_k=G_{k-1}+\Span{(\D\psi\rfloor w_2)w_1-(\D\psi\rfloor w_1)w_2}=D_{k_1-1}+\Span{(\D\psi\rfloor v_2)v_1-(\D\psi\rfloor v_1)v_2}=H_k$, \ie the particular choice for the vector fields $w_1,w_2$ has no effect as long as they complete $G_{k-1}$ to $G_k$.}\\
			
			Next, let us consider the case $\overline{G}_k\neq\TXU$. Note that because of $[f,G_{k-1}]\subset G_k\subset\overline{G}_k$, in the above introduced coordinates the vector field $f$ is of the form
			\begin{align*}
				f&=f_3^{i_3}(z_3,z_2)\partial_{z_3^{i_3}}+f_2^{i_2}(z_3,z_2,z_1)\partial_{z_2^{i_2}}+f_1^{i_1}(z_3,z_2,z_1)\partial_{z_1^{i_1}}\,.
			\end{align*}
			We thus have
			\begin{align*}
				[f,G_k]+\overline{G}_k&=\overline{G}_k+\Span{[v_1,f_3],[v_2,f_3]}\,,
			\end{align*}
			where $f_3=f_3^{i_3}(z_3,z_2)\partial_{z_3^{i_3}}$. Since by assumption $\Dim{[f,G_k]+\overline{G}_k}=\Dim{\overline{G}_k}+1$, the vector fields $[v_1,f_3]$ and $[v_2,f_3]$ are collinear modulo $\overline{G}_k$. Without loss of generality, we can assume that $[v_1,f_3]\notin\overline{G}_k$, implying that there exists a function $\alpha$ such that $[v_2,f_3]=\alpha[v_1,f_3]\Mod\overline{G}_k$. This function only depends on $\bar{x}_2$ and $\bar{x}_3$ since the vector fields $v_1,v_2$ and $f_3$ only depend on $z_3$ and $z_2$, \ie $\alpha=\alpha(z_3,z_2)$. Define the distribution $H_k=G_{k-1}+\Span{v_2-\alpha v_1}$. It is immediate that this distribution is involutive and since $v_1$ and $v_2$ are independent, we have $G_{k-1}\underset{1}{\subset}H_k\underset{1}{\subset}G_k$.\footnote{It follows that the direction of the vector field $v=v_2-\alpha v_1$ is modulo $G_{k-1}$ uniquely determined by the conditions $v\in G_k$, $v\notin G_{k-1}$, $[v,f]\in\overline{G}_k\neq\TXU$ and thus, the distribution $H_k$ is unique in this case. Indeed, assume that there would exist another such vector field $w\in G_k$, $w\notin G_{k-1}$, which is modulo $G_{k-1}$ not collinear with $v$ and still satisfies $[w,f]\in\overline{G}_k$. Then, we would have $G_k=G_{k-1}+\Span{v,w}$ and in turn $[f,G_k]\subset\overline{G}_k$, which is in contradiction with $\Dim{[f,G_k]+\overline{G}_k}=\Dim{\overline{G}_k}+1$.}\\
			
			So in any of the two cases, \ie $\overline{G}_k=\TXU$ or $\overline{G}_k\neq\TXU$, we have shown that there exists an involutive distribution $H_k=G_{k-1}+\Span{v}$ with some vector field $v\in G_k$, $v\notin G_{k-1}$. What is left to do is to show that $H_k+[f,H_k]=\overline{G}_k$. By construction, we have $[f,H_k]\subset\overline{G}_k$, so we only have to show that $[f,v]\notin G_k$, which can be shown by contradiction. Assume that $[f,v]\in G_k$. Due to the Jacobi identity, for every vector field $w\in G_{k-1}$, we then have
			\begin{align*}
				\underbrace{[w,\overbrace{[f,v]}^{\in G_k}]}_{\in G_k}+[v,\overbrace{[w,f]}^{\in G_k}]+\underbrace{[f,\overbrace{[v,w]}^{\in H_k}]}_{\in G_k}&=0\,,
			\end{align*}
			(where we have used that $G_{k-1}\subset\C{G_k}$ and that $[f,v]\in G_k$ implies $[f,H_k]\subset G_k$ as well as the involutivity of $H_k$). However, this implies $[v,[w,f]]\in G_k$ for every $w\in G_{k-1}$, which in turn would imply that $G_k$ is involutive and contradict with $G_k$ being non-involutive.
		\subsection{Proof of Lemma \ref{lem:cauchy}}
			We have $D_{k_1}=D_{k_1-1}+\Span{v_1,v_2}$ and since $D_{k_1-1}\subset\C{D_{k_1}}$, we have $D_{k_1}^{(1)}=D_{k_1-1}+\Span{v_1,v_2,[v_1,v_2]}$. Based on the Jacobi identity, it can easily be shown that $D_{k_1-1}\subset\C{D_{k_1}^{(1)}}$. Since by assumption $\Dim{\overline{D}_{k_1}}=\Dim{D_{k_1}}+2$, the vector fields $[v_1,[v_1,v_2]]$ and $[v_2,[v_1,v_2]]$ are collinear modulo $D_{k_1}^{(1)}$. Without loss of generality, we can assume that $[v_1,[v_1,v_2]]\notin D_{k_1}^{(1)}$ and thus, $[v_2,[v_1,v_2]]=\alpha [v_1,[v_1,v_2]]\Mod D_{k_1}^{(1)}$. The vector field $v=v_2-\alpha v_1$ by construction meets $v\in D_{k_1}$, $v\notin D_{k_1-1}$ and $[v,[v_1,v_2]]=0\Mod D_{k_1}^{(1)}$, implying that $v\in\C{D_{k_1}^{(1)}}$. It follows that $\C{D_{k_1}^{(1)}}=D_{k_1-1}+\Span{v}$ (the existence of further Cauchy characteristic vector fields would contradict with $D_{k_1}^{(1)}$ being non-involutive) and thus $D_{k_1-1}\underset{1}{\subset}\C{D_{k_1}^{(1)}}\underset{1}{\subset}D_{k_1}$.\\
			
			That $\C{D_{k_1}^{(1)}}+[f,\C{D_{k_1}^{(1)}}]=D_{k_1}^{(1)}$ can be shown as follows. We have $D_{k_1}\underset{1}{\subset}D_{k_1}^{(1)}$ and by assumption $[f,\C{D_{k_1}^{(1)}}]\subset D_{k_1}^{(1)}$, and we just have shown that $\C{D_{k_1}^{(1)}}=D_{k_1-1}+\Span{v}$ with $v\in D_{k_1}$, $v\notin D_{k_1-1}$. Therefore, $\C{D_{k_1}^{(1)}}+[f,\C{D_{k_1}^{(1)}}]=D_{k_1}^{(1)}$ holds if and only if $[f,v]\notin D_{k_1}$, which can be shown by contradiction. Assume that $[f,v]\in D_{k_1}$. Due to the Jacobi identity, for every vector field $w\in D_{k_1-1}$, we then have
			\begin{align*}
				\begin{aligned}
					\underbrace{[w,\overbrace{[f,v]}^{\in D_{k_1}}]}_{\in D_{k_1}}+[v,\overbrace{[w,f]}^{\in D_{k_1}}]+\underbrace{[f,\overbrace{[v,w]}^{\in \C{D_{k_1}^{(1)}}}]}_{\in D_{k_1}}&=0\,,
				\end{aligned}
			\end{align*}
			(where we have used that $D_{k_1-1}\subset\C{D_{k_1}}$ and that $[f,v]\in D_{k_1}$ implies $[f,\C{D_{k_1}^{(1)}}]\subset D_{k_1}$ as well as the involutivity of $\C{D_{k_1}^{(1)}}$). However, this implies that $[v,[w,f]]\in D_{k_1}$ for every $w\in D_{k_1-1}$, which in turn would imply that $D_{k_1}$ is involutive and contradict with $D_{k_1}$ being non-involutive.
\end{document}